\documentclass[11pt,pdftex,noinfoline]{imsart}
\usepackage[usenames,dvipsnames]{xcolor}
\usepackage{natbib}
\bibpunct{(}{)}{;}{a}{,}{,}
\usepackage{verbatim}
\usepackage{comment}
\usepackage[pdftex,pdfborderstyle={/S/U/W 1},colorlinks]{hyperref}
\RequirePackage[OT1]{fontenc}
\usepackage[linesnumbered]{algorithm2e}
\RequirePackage{amsthm,amsmath,amsfonts}
\RequirePackage{hypernat}
\usepackage{fullpage}
\usepackage{graphicx}
\usepackage{hyperref}
\usepackage{amsthm, amssymb, amsmath}
\usepackage{url}
\usepackage[english]{babel}
\usepackage{times}
\usepackage[T1]{fontenc}
\usepackage{bbm}
\usepackage{xspace}
\usepackage{rotating,multirow}
\usepackage{enumerate}
\usepackage{mathtools}
\usepackage{tikz}
\usepackage{float}

\usepackage{setup}
\usepackage{newcommands}
\usepackage{xr}

\begin{document}
\begin{frontmatter}
\title{New Risk Bounds for 2D Total Variation Denoising}        
\runtitle{New Risk Bounds for 2D Total Variation Denoising}  
\runtitle{New Risk Bounds for 2D Total Variation Denoising\;\;\;\;\;\;\;}

\begin{aug}
\author{\fnms{Sabyasachi} \snm{Chatterjee}}
\and
\author{\fnms{Subhajit} \snm{Goswami}}
\blfootnote{Author names are sorted alphabetically.}

\runauthor{Chatterjee, S. and Goswami S.}

\affiliation{University of Illinois at Urbana-Champaign and Tata Institute of Fundamental Research }

\address{117 Illini Hall\\
Champaign, IL 61820 \\
Email:\ \lowercase{\href{mailto:sc1706@illinois.edu}}{sc1706@illinois.edu}\\
\phantom{E-mail:\ sc1706@illinois.edu}}

\address{1, Homi Bhabha Road\\
Colaba, Mumbai 400005, India \\
Email:\ \lowercase{\href{mailto:goswami@math.tifr.res.in}}{goswami@math.tifr.res.in}\\
\phantom{E-mail:\ goswami@ihes.fr}}


\end{aug}

\begin{abstract}
2D Total Variation Denoising (TVD) is a widely used technique for image denoising. It is also an important nonparametric regression method for estimating functions with heterogenous smoothness. Recent results have shown the TVD estimator to be nearly minimax rate optimal for the class of functions with bounded variation. In this paper, we complement these worst case guarantees by investigating the adaptivity of the TVD estimator to functions which are piecewise constant on axis aligned rectangles. We rigorously show that, when the truth is piecewise constant with few pieces, the ideally tuned TVD estimator performs better than in the worst case. We also study the issue of choosing the tuning parameter. In particular, we propose a fully data driven version of the TVD estimator which enjoys similar worst case risk guarantees as the ideally tuned TVD estimator. 
\end{abstract}

\end{frontmatter}

\section{Introduction}\label{sec:intro}
Total variation denoising (TVD) is a standard technique to do noise removal in images. This technique was first proposed in~\citet{rudin1992nonlinear} and has since then been heavily used in the image processing community. 
It is well known that TVD gets rid of unwanted noise and also preserves edges in the image (see~\citet{strong2003edge}). For a survey of this technique from an image analysis point of view; see~\citet{chambolle2010introduction} and references therein.

The success of the TVD technique as a denoising mechanism motivates us to revisit this problem from a statistical perspective. In this paper, we are interested in the following statistical estimation problem. Consider observing $y = \theta^* + \sigma Z$ where $y \in \R^{n \times n}$ is a noisy matrix/image, $\theta^*$ is the true underlying matrix/image, $Z$ is a noise matrix consisting of independent standard Gaussian entries and $\sigma$ is an unknown standard deviation of the noise entries. Thus, in this setting, the image denoising problem is cast as a Gaussian mean estimation problem. Before defining the TVD estimator in this context, let us define total variation of an arbitrary matrix.

Let us denote the $n \times n$ two dimensional \emph{grid graph} by $L_n$ and denote its 
edge set by $E_n$. More precisely, the vertices in $L_n$ correspond to the pairs $(i, j) \in [n] \times [n]$ and its edge set $E_n$ consists of:
$$\mbox{all }((i, j), (k, \ell)) \in L_n \times L_n \mbox{ such that } |i - j| + |k-\ell| = 1\,.$$
We will use $L_n$ interchangeably for the graph as well as the underlying set of vertices. Now, thinking of $\theta \in \R^{n \times n}$ as a function on $L_n$ let us define
\begin{equation}\label{eq:TVdef}
\TV_{\mathrm{norm}}(\theta) \coloneqq \frac{1}{n} \sum_{(u,v) \in E_n} |\theta_{u} - \theta_{v}| = \frac{1}{n} \|D \theta\|_1
\end{equation} 
where $D$ is the usual edge vertex incidence matrix of size $2n(n - 1) \times n^2.$ The $1/n$ factor is just a normalizing factor so that if $\theta_{ij} = f(i/n,j/n)$ for some underlying differentiable function on the unit square then $\TV_{\mathrm{norm}}(\theta)$ is precisely the discretized Riemann approximation for $\int_{[0,1]^2} \big|\frac{\partial f(x,y)}{\partial x}\big| + \big|\frac{\partial f(x,y)}{\partial y}\big|.$ This $1/n$ scaling is termed as the \textit{canonical scaling} in~\cite{sadhanala2016total}.
The above notion of total variation extends the definition of \textit{variation} from differentiable functions on the unit square to arbitrary matrices. 
We can now define the TVD estimator, which is our main object of study.
\begin{equation*}
\hat{\theta}_{\textbf V} := \argmin_{\theta \in \R^{n \times n}: \TV_{\mathrm{norm}}(\theta) \leq \textbf{V}} \|y - \theta\|^2
\end{equation*}
where $\|.\|$ throughout this paper will denote the usual Frobenius norm for matrices. 
The TVD estimator is actually a family of estimators indexed by the tuning parameter $\textbf{V} > 0.$ We will measure the performance of our estimator in terms of its normalized mean squared error (MSE) defined as
\begin{equation*}
{\rm MSE}(\hat{\theta}_{\textbf{V}},\theta^*) \coloneqq \E_{\theta^*} \frac{\|\hat{\theta}_{\textbf{V}} - \theta^*\|^2}{N} 
\end{equation*}
where throughout this paper we denote $N = n^2.$ 

We defined the TVD estimator in its constrained form, however the penalized version is also popular in the literature, which is defined as follows:
\begin{equation*}
\hat{\theta}_{\lambda} \coloneqq \argmin_{\theta \in \R^{n \times n}} \|y - \theta\|^2 + \lambda\:\TV_{\mathrm{norm}}(\theta)
\end{equation*}
where $\lambda > 0$ is a tuning parameter. In this paper, we focus on the analysis of the constrained version. 

\subsection{Background and Motivation}
The $1$D version of this problem is a well studied problem (see, e.g.~\cite{tibshirani2005sparsity}) in nonparametric regression. In this setting, we again have $y = \theta^* + \sigma Z$ as before, where $y,\theta^*,Z$ are now vectors instead of matrices. The total variation of a vector $v \in \R^n$ can now be defined as
\begin{equation*}
\TV(v) \coloneqq \sum_{i = 1}^{n - 1} |v_{i + 1} - v_{i}|.
\end{equation*}
Again the above definition can be seen as a discrete Riemann approximation to $\int_{[0,1]} \big|f^{'}(x)\big| dx$ when $v_{i} = f(i/n)$ for some differentiable function $f.$ The constrained and the penalized versions of the TVD 
estimator can now be defined analogously. The penalized form seems to be more popular in the existing literature; in this case the TVD estimator is often referred to as fused lasso (see~\cite{tibshirani2005sparsity},~\cite{rinaldo2009properties}). In this 1D setting, it is known (see, e.g.~\cite{donoho1998minimax},~\cite{mammen1997locally}) that the TVD estimator is minimax rate optimal on the class of all bounded variation signals $\{\theta: \TV(\theta) 
\leq \mb V\}$ for  $\mb V > 0$. It is also shown in~\cite{donoho1998minimax} that no estimator, which is a linear function of $y$, can attain this minimax rate.

It is also worthwhile to mention here that TVD in the 1D setting has been studied as part of a general family of estimators which penalize discrete derivatives of different orders. These estimators have been studied in~\cite{steidl2006splines},~\cite{tibshirani2014adaptive} and by~\citet{kim2009ell_1} who coined the name \textit{trend filtering}. A continuous version of these estimators, where discrete derivatives are replaced by continuous derivatives, was proposed much earlier in the statistics literature by~\citet{mammen1997locally} under the name \textit{locally adaptive regression splines}.

Total variation of a signal can actually be defined over an arbitrary graph as the sum of absolute differences of the signal across edges of the graph. Trend Filtering on general graphs has been a popular research topic in the recent past; see~\cite{wang2016trend},~\cite{lin2016approximate}. A more recent paper,~\cite{ortelli2018}, studies TVD on tree graphs. The 1D setting corresponds to the chain graph on $n$ vertices whereas the 2D setting corresponds to the 2D lattice graph on $n^2 = N$ vertices.

The 2D TVD problem, while being much less studied than in its 1D counterpart, has enjoyed a recent surge of interest. Worst case performance of the TVD estimator has been studied in~\cite{hutter2016optimal},~\cite{sadhanala2016total},~\cite{ortelli2019oracle}. These results show that like in the 1D setting, the 2D TVD estimator is nearly minimax rate optimal over the class $\{\theta \in \R^{n \times n}: \TV_{\mathrm{norm}}(\theta) \leq \textbf{V}\}$ of bounded variation signals. In fact,~\cite{sadhanala2016total} also generalize the result of~\cite{donoho1998minimax} and prove that no linear function of $y$ can attain the minimax rate in the 2D setting as well.  
A representative of the state of the art risk bound for the TVD estimator in 2D setting is due 
to~\citet{hutter2016optimal} (see also~\cite{ortelli2019oracle}). They studied the penalized 
form of the TVD estimator and proved that there exist universal constants $C,\:c > 0$ such 
that by setting $\lambda = c \sigma \log n$ (where $\sigma$ is known), one gets
\begin{theorem}[H{\"u}tter and Rigollet]\label{thm:hrig}
	\begin{equation*}
	{\rm MSE}(\hat{\theta}_{\lambda},\theta^{*}) \leq C (\log n)^2 \min\{\sigma \frac{\TV_{\mathrm{norm}}(\theta^{*})}{\sqrt{N}}, \sigma^2 \frac{\|D \theta^{*}\|_0}{N}\}.
	\end{equation*}
\end{theorem} 
where $\|\cdot\|_0$ is the usual $\ell_0$ norm. 

For convenience, we will henceforth use the usual $O(\cdot)$ notation to compare sequences. We write $a_n = O(b_n)$ if there exists a constant $C > 0$ such that $a_n \leq C\:b_n$ for all sufficiently large $n$. We also use $a_n = \tilde O(b_n)$ to denote $a_n = O(b_n(\log n)^C)$ for some $C > 0$.

In words, the bound in Theorem~\ref{thm:hrig} is a minimum of two terms. The first term gives the $\ell_1$ rate scaling like $O(1/\sqrt{N})$ for bounded variation functions. The second one is the $\ell_0$ rate which can be much faster than the $O(1/\sqrt{N})$ rate if $|D \theta^{*}|_0$ is small enough. In spite of the above works, there are still a couple of unexplored aspects regarding 2D TVD, specifically its adaptivity to piecewise constant signals and minimax optimality without tuning, which are the focus of the present paper. We discuss them now.

\subsubsection{Adaptivity to piecewise constant signals} 
Observe that the total variation semi norm is a convex relaxation for the number of times the true signal $\theta^*$ changes values along the neighbouring vertices. 
This fact suggests that the TV estimator \textit{might} perform very 
well if the true signal is indeed piecewise constant. This phenomenon is now fairly well understood in the 1D setting. 
In this setting, suppose  that the true vector $\theta^*$ is piecewise constant with $k+1$ contiguous pieces or blocks. Given data $y \sim N_n(\theta^*,\sigma^2 I_n)$, an oracle estimator, which knows the locations of the jumps, would just estimate the signal $\theta^*$ by the mean of the data vector $y$ within each block. 
It can be easily checked that the oracle estimator will have MSE bounded by $\sigma^2 (k+1)/n.$ Recent works (see~\cite{dalalyan2017tvd},~\cite{lin2016approximate}) studied the penalized TVD estimator and showed that if the \textit{minimum length} of the blocks where $\theta^*$ is constant is not too small (scales like $O(n/k)$) 
and if the tuning parameter $\lambda$ is set to be equal to an appropriate function of the unknown $\sigma$ and $n,$ then an oracle risk $O(k/n)$ could be achieved up to some additional logarithmic factors in $k$ and $n.$ In~\cite{guntuboyina2017adaptive}, this adaptive behaviour was established for the ideally tuned constrained form of the estimator with slightly better log factors. Thus, we can say that in the 1D setting, the TVD estimator is optimally adaptive to piecewise constant signals.

This motivates us to wonder whether similar adaptivity holds in the 2D setting. In this paper, we investigate adaptivity to signals/matrices which are \textit{piecewise constant on $k < < N$ axis aligned rectangles}. Such adaptivity of the 2D TVD estimator has not been explored at all in the literature. Estimation of functions which are piecewise constant on axis aligned rectangles are naturally motivated by methodologies such as CART (see e.g~\citet{breiman2017classification}) which produce outputs of the same form.
Recently, adaptation to piecewise constant structure on rectangles has been of interest in the nonparametric shape constrained function estimation literature also (see Theorem $2.3$ in~\citet{chatterjee2018matrix} and Theorems $2$ and $5$ in~\citet{han2017isotonic}). See Section~\ref{sec:pc} where we discuss some even more recent (which appeared after we uploaded this paper) works about estimating piecewise constant functions on axis aligned rectangles. Here is the main question that we address in this paper.

\noindent \textit{Q1: If the underlying $\theta^*$ is piecewise constant on at most $k < < N$ axis aligned rectangles; can the ideally tuned TVD estimator attain a faster rate of convergence than the $\tilde{O}(1/\sqrt{N})$ rate}?

Basically we are asking the question whether the ideally tuned TVD estimator adapts to truths which are piecewise constant on a few axis aligned rectangles, which is a different notion of sparsity than the sparsity constraint of $\|D \theta^*\|_0$ being small. As a simple instance of $\theta^*$ being piecewise constant on rectangles, consider $\theta^*$ to be of the following form:
\[
\theta^{*} =
\left[ {\begin{array}{cc}
	\mb 0_{n \times n/2} & \mb 1_{n \times n/2} \\
	\end{array} } \right]
\]

In this case, we have $\|D \theta^*\|_0 = O(\sqrt{N})$ and $\TV_{\mathrm{norm}}(\theta^*) = O(1).$ Note that the $\ell_0$ bound of~\cite{hutter2016optimal} will give us an upper bound on the MSE scaling like $\tilde{O}(1/\sqrt{N})$ which is already given by the $\ell_1$ bound. Thus, the result of~\cite{hutter2016optimal} does not help in answering our question and suggests there is no adaptation. In Theorem~\ref{thm:adap} of this paper, we show that the {ideally} tuned TVD estimator indeed adapts to piecewise constant matrices on axis aligned rectangles and provably attains a rate of convergence scaling like $\tilde{O}(1/N^{3/4})$ which is strictly faster than the $\ell_1$ rate $\tilde{O}(1/\sqrt{N})$. However, we also show that this $\tilde{O}(1/N^{3/4})$ rate is tight and thus the TVD estimator is not able to attain the $\tilde{O}(1/N)$ parametric rate that would be achieved by an oracle estimator. This is the main contribution of this paper and is the first result of its type in the literature as far as we are aware.

\subsubsection{Minimax rate optimality without tuning}
Existing results such as Theorem~\ref{thm:hrig}, along with minimax lower bounds shown in~\cite{sadhanala2016total}, show that the $\tilde{O}(\tfrac{\textbf{V}}{\sqrt{N}})$ rate attained by the penalized TVD estimator is near minimax rate optimal. Thus we can say that the penalized TVD estimator is near minimax rate optimal over the parameter space $\{\theta \in \R^{n \times n}: \TV_{\mathrm{norm}}(\theta) \leq \textbf{V}\}$, \textit{simultaneously} over $\textbf{V}$ and $N$. However, this penalized TVD estimator needs to set a tuning parameter $\lambda$ which depends on the unknown $\sigma$ and an implicit constant $C$ which can be potentially difficult to set in practice. This naturally raises a question which is unresolved in the literature so far as we are aware:

\noindent\textit{Q2: Does there exist a completely data driven estimator which does not depend on any unknown parameters of the problem and yet achieves MSE scaling like $\tilde{O}(\tfrac{\textbf{V}}{\sqrt{N}})$, thus being simultaneously minimax rate optimal over 
$\textbf{V}$ and $N$?}

In Theorem~\ref{Thm:notuning} of this paper we answer this question in the affirmative by constructing such a fully data driven estimator.


The rest of the paper is organised as follows. In Section~\ref{Sec:results}, we state 
our main results. Then in Section~\ref{sec:discuss}, we discuss connections of our results with some recent works and also present simulation studies which support and verify our main theorems. The proofs of our main results involve sharp bounds on the {\em Gaussian widths} (see \eqref{def:gauss_width} in Section~\ref{sec:gauss_width_bnd_results}) for some special classes of matrices. We obtain these bounds based on a generic approach which we detail in Section~\ref{sec:chaining}. The next five sections describe the proofs of our main theorems and intermediate results. Section~\ref{Sec:appendix} is the appendix which contains proofs of some auxiliary results. 
\medskip

\noindent \textbf{Instructions for the reader}

\smallskip

In all the proofs of our results from Section~\ref{Sec:5} onwards, we will use $\TV(\cdot)$ to denote the \emph{unnormalized} version of~\eqref{eq:TVdef}. More precisely, for a $n \times n$ matrix $\theta$ we denote
\begin{equation}\label{eq:unnorm}
\TV(\theta) \coloneqq \sum_{(u,v) \in E_n} |\theta_{u} - \theta_{v}|.
\end{equation}
We adopt this convention because we believe it is easier to read and interpret the proofs with the unnormalized definition while it is instructive to use the normalized version for our theorems to facilitate interpretation of the risk bounds as a function of the sample size $N = n^2$. Also we will generically use $V$ to denote the unnormalized total variation whereas in Sections~\ref{sec:intro}--\ref{sec:discuss} we use \emph{bold} $\mb V$ to denote the {\em corresponding} normalized total variation.
In all our theorems presented in the next section we use \emph{bold} $\mb V^* $ to denote $\TV_{\mathrm{norm}}(\theta^*)$ where $\theta^*$ is the underlying true matrix and in all our proofs we use $V^* = \TV(\theta^*)$ for the corresponding unnormalized version.

\textbf{Acknowledgements}: This research was supported by a NSF grant and an IDEX grant from Paris-Saclay.  We thank the anonymous referees for extremely detailed comments and suggestions about the paper. These comments and suggestions have helped us to a great extent to improve our article. The project started when both the authors were at the University of Chicago.

\section{Main Results}\label{Sec:results}
\subsection{Constrained TVD}
Our first result states a risk bound of $\hat{\theta}_{\textbf{V}}$ under the bounded variation constraint. 
\begin{theorem}\label{Thm:1}
	Let $\theta^*$ be an arbitrary $n \times n$ matrix and $N = n^2$. Suppose the tuning 
	parameter is chosen such that $\mb V \geq \mb V^*$. Then the following risk bound 
	is true for a universal constant $C > 0$:                                                                                                                                                                                  
	\begin{equation*}
	{\rm MSE}(\hat{\theta}_{\mb V},\theta^*) \leq C \big( \sigma \frac{\mb V}{\sqrt{N}}\, (\log \e N)^{5/2} +  \frac{\sigma^2}{N} \big).
	\end{equation*}
\end{theorem}

\begin{remark}\label{rem1}
	The above result is similar to the $\ell_1$ bound of~\cite{hutter2016optimal}, the difference being the above risk bound holds for the constrained TVD estimator while the existing result of~\cite{hutter2016optimal} holds for the penalized estimator. For any sequence of $\mb V > 0$ (possibly growing with $n$ although the canonical scaling is when $\mb V = O(1)$), the minimax lower bound results (mentioned earlier) of~\cite{sadhanala2016total} now imply the minimax rate optimality (up to log factors) of the constrained TVD estimator $\hat{\theta}_{\mb V}$ over the parameter space $\{\theta \in \R^{n \times n}: \TV_{\mathrm{norm}}(\theta) \leq \mb V\}$.
\end{remark}

\begin{remark}
	As is made clear in Section~\ref{sec:thm1}, our technique for proving Theorem~\ref{Thm:1} is completely different from the technique used to prove the result of~\cite{hutter2016optimal}. While they analyze the properties of the pseudo-inverse of the edge incidence matrix $D,$ our proof relies on computing relevant Gaussian widths by recursive partitioning. Moreover, ingredients and ideas from this proof are also used crucially in the proofs of our other results.
\end{remark}

\subsection{Adaptive risk bound}
\label{sec:adaptive}
Now we come to the main result of this paper which is about proving adaptive risk bounds for $\theta^*$ which are piecewise constant on at most $k$ axis aligned rectangles where $k$ is a positive integer much smaller than $N.$ 
We call a subset $R \subset L_n$ a (axis aligned) rectangle if it is a product of two intervals. For a generic rectangle $R = ([a, b] \cap \N) \times ([c, d] \cap \N)$, we define $\nr(R)$ and $\nc(R)$ to be the cardinalities of $[c, d] \cap \N$ and $[a, b] \cap \N$ 
respectively. In words, $\nr(R)$ and $\nc(R)$ are simply the numbers of rows and columns of $R$ respectively if one views $R$ as a two-dimensional array of points. Then we define its aspect ratio to be $A(R) \coloneqq\max\{\frac{\nr(R)}{\nc(R)},\frac{\nc(R)}{\nr(R)}\}$. For a given matrix $\theta \in \R^{n \times n}$ we define $k(\theta)$ to be the cardinality of the minimal partition of $L_n$ into rectangles $R_1,\dots,R_{k(\theta)}$ such that $\theta$ is constant on each of the rectangles. Next we state our main result for the 2D TVD estimator.

\begin{theorem}\label{thm:adap}
	Let $\theta^* \in \R^{n \times n}$ be the underlying true matrix with $\TV_{\mathrm{norm}}(\theta^*) > 0$ and  $R_1^*,\dots,R_{k(\theta^*)}^*$ be its rectangular level sets which form a partition of the 2D grid $L_n.$ In addition, suppose the rectangles $R_i^*$ have bounded aspect ratio, that is there exists a 
	constant $c > 0$ such that $\max_{i \in [k(\theta^*)]} A(R_i^*)$ $\leq c$. Then we have the following risk bound:
	\begin{equation*}
	{\rm MSE}(\hat{\theta}_{\mb V},\theta^*) \leq C \big[(\mb V - \mb V^*)^2 + 
	\sigma^2 (\log \e n)^{9} \frac{\:k(\theta^*)^{5/4}}{N^{3/4}}\big]\,. 
	\end{equation*}
	Here $C$ is a constant that only depends on $c.$
\end{theorem}

\begin{remark}
Theorem~\ref{thm:adap} is really a statement about an ideally tuned constrained TVD estimator. One way to interpret it is that if the tuning parameter $\mb V$ is chosen such that $(\mb V - \mb V^*)^2 \leq C \sigma^2 (\log \e n)^{9} \frac{\:k(\theta^*)^{5/4}}{N^{3/4}}$ then the $\tilde{O}(\frac{\:k(\theta^*)^{5/4}}{N^{3/4}})$ rate of convergence holds. 
\end{remark}

\begin{remark}
	One consequence of the above theorem is that when $k(\theta^*) = O(1)$ then the ideally tuned TVD estimator attains a $\tilde{O}(N^{-3/4})$ rate. This rate is faster than the $\tilde{O}(N^{-1/2})$ rate that is available in the literature. Our main focus here has 
	been to attain the right exponent for $N$. The exponent of $k(\theta^*)$ and $\log n$ may not be 
	optimal. Since the current proof of this theorem is fairly involved technically, 
	obtaining the best possible exponents of $k(\theta^*)$ and $\log n$ is left for future research 
	endeavors. See Section~\ref{sec:discuss} for more discussions about the proof of the above theorem and comparisons with existing results. 
\end{remark}

\begin{remark}
	We think a bounded aspect ratio condition would actually be necessary for the $O(N^{-3/4})$ rate to hold in the above theorem; see Section~\ref{sec:bddasp} for more on this issue.
\end{remark}


A natural question is whether our upper bound in Theorem~\ref{thm:adap} is tight. Our next theorem says that, in the low $\sigma$ limit, the $N^{-3/4}$ rate is not improvable even if $k(\theta^*) = 2.$ 

\begin{theorem}\label{thm:lbd}
	Let $\theta^*_{ij} =  1$ if $j > n/2$ and $0$ otherwise. Thus, $\theta^*$ is of the following form:
	\[
	\theta^{*} =
	\left[ {\begin{array}{cc}
		\mb 0_{n \times n/2} & \mb 1_{n \times n/2} \\
		\end{array} } \right]
	\]
	Clearly $k(\theta^*) = 2.$ In this case, we have a lower bound to the risk of the ideally constrained TVD estimator. 
	\begin{equation*}
	\lim_{\sigma \to 0}\frac{1}{\sigma^2}{\rm MSE}(\hat{\theta}_{\mb V^*},\theta^*) \geq  \frac{c} {N^{3/4}}\,.
	\end{equation*}
	Here $c > 0$ is a universal constant. 
\end{theorem}

\subsubsection{Gaussian width bounds}
\label{sec:gauss_width_bnd_results}
Proving Theorem~\ref{thm:adap} and Theorem~\ref{thm:lbd} requires upper and lower bounds on the \textit{Gaussian width} of a certain family of matrices as we now explain. The Gaussian width of a set $K \subset \R^n$ is defined as \begin{equation}\label{def:gauss_width}\gw(K) = \E \sup_{\theta \in K} \langle Z, \theta \rangle\end{equation} where $Z = Z_n\sim N(0_n, I)$ and $\langle \cdot\,, \cdot \rangle$ is the usual Euclidean inner product between two vectors. We use $B_{m,n}(t)$ to denote the usual Euclidean ball of radius $t$ in $\R^{m 
	\times n}$. For any $A \subset \R^{n \times n}$ we denote the smallest cone containing $A$ by ${\rm Cone}(A)$ and the closure of $A$ by ${\rm Closure}(A)$. The \textit{tangent cone} $T(\theta^*) \subset \R^{n \times n}$ {\em at} $\theta^*$ with respect to the closed convex set $K^* \coloneqq \{\theta \in \R^{n \times n}: \TV_{\mathrm{norm}}(\theta) \le \mb V^*\}$ is defined as follows:
\begin{equation}\label{eq:deftancone}
T_{K^*}(\theta^*) \coloneqq {\rm Closure}(\,{\rm Cone}(\{\theta \in \R^{n \times n}: \theta^* + \theta \in K^*\})\,)\,.
\end{equation}
By definition, $T_{K^*}(\theta^*)$ is a closed convex cone. Informally, $T_{K^*}(\theta^*)$ represents all directions in which one can move infinitesimally from $\theta^*$ while still remaining in $K^*$. 
Roughly speaking, the problem of bounding the MSE from both directions is equivalent to bounding the square of $ \GW\big(T_{{K}^*}(\theta^*) \cap B_{n \times n}(1)\big)$ when $\theta^*$ is a piecewise constant matrix on rectangles. The precise connection of MSE to Gaussian widths is detailed in Section~\ref{Sec:adap} where the proofs of Theorem~\ref{thm:adap} and Theorem~\ref{thm:lbd} are also given. This connection prompts us to investigate how these tangent cones look like in the first place. \textit{The major technical contribution of this paper is to give upper and lower bounds on the Gaussian width of the tangent cone at a piecewise constant matrix which we encapsulate in the following two results.}


\begin{proposition}\label{prop:gwupbd}
	Let $\theta \in \R^{n \times n}$ be a given matrix and  $R_1,\dots,R_{k(\theta)}$ be its rectangular level sets which form a partition of the 2D grid $L_n.$ In addition, let us assume that the rectangles $R_i$ have bounded 
	aspect ratio, that is there exists a constant $c > 0$ such that $\max_{i \in [k]} A(R_i) \leq 
	c$. Let $K \coloneqq \{v \in \R^{n \times n}: \TV_{\mathrm{norm}}(v) \leq \TV_{\mathrm{norm}}(\theta)\}$ and $T_{K}(\theta)$ be the tangent cone at $\theta$ with respect to $K.$ Then there is a universal constant $C > 0$ such that 
	\begin{equation*}
	\GW(T_{K}(\theta) \cap B_{n,n}(1)) \leq C (\log n)^{{ 4.5}} k(\theta)^{5/8} n^{1/4}.
	\end{equation*}
\end{proposition}

\begin{proposition}\label{prop:gwlbd}
	Consider $\theta^*$ which is piecewise constant on two rectangles and is of the following form:
	\[
	\theta^{*} =
	\left[ {\begin{array}{cc}
		\mb 0_{n \times n/2} & \mb 1_{n \times n/2} \\
		\end{array} } \right]
	\] 
	Then, there exists a universal constant $c > 0$ such that we have the following lower bound:
	\begin{equation*}
	\GW(T_{{K}^*} (\theta^*) \cap B_{n \times n}(1)) \geq c n^{1/4}\,.
	\end{equation*}
\end{proposition}

The proofs of Proposition~\ref{prop:gwupbd} and Proposition~\ref{prop:gwlbd} are given in Sections~\ref{Sec:gwup} and~\ref{Sec:gwlb} respectively.

It should be mentioned here that bounding the Gaussian width of the tangent cone is a fundamental task in a different but 
related problem of signal recovery from a given number of measurements; 
see~\cite{ChandraFOCS} and~\cite{amelunxen2014living}. Matrix recovery using 2D Total 
Variation has been studied in the signal processing literature; see for 
instance~\cite{cai2015guarantees},~\cite{genzel2020ell1} and \cite{kabanava2014robust}. Our bounds on the Gaussian widths given in Proposition~\ref{prop:gwupbd}, Proposition~\ref{prop:gwlbd} and 
Theorem~\ref{thm:impos} (see below) appear to be new and are potentially of independent interest as stand alone results. Especially our use of {\em optimized} partitioning schemes (see Section~\ref{sec:onebdry} for details) in the proof of Proposition~\ref{prop:gwupbd} can be a useful strategy to attack other problems of similar flavor. See also Section~\ref{sec:guntu} for further discussion on the novelty of our proof.



\subsubsection{Impossibility of adaptation to non rectangular level sets}
Theorem~\ref{thm:adap} shows that the $O(N^{-3/4})$ rate is achievable when $\theta^*$ is piecewise constant on a few rectangles. A question arises here as to what rate is achievable when $\theta^*$ is piecewise constant but the level sets are not rectangular. The following theorem says that for a simple matrix $\theta^*$ whose level sets are triangular, the 
$\tilde{O}(N^{-1/2})$ rate cannot be improved. Below and in the rest of the paper we use $\mathbb I\{\cdot \in S\}$ to denote the indicator function for the set $S$ (often stated as a condition defining its elements), i.e., it takes the value 1 when its argument lies in the set $S$ and is $0$ otherwise.
\begin{theorem}\label{thm:impos}
	Consider the signal matrix $\theta^* \coloneqq \mathbb I\{i + j > n\}$. Then, there exists a universal constant $c > 0$ such that we have the following lower bound:
	\begin{equation*}
	\GW(T_{{K}^*} (\theta^*) \cap B_{n \times n}(1)) \geq c n^{1/2}\,.
	\end{equation*}
	Further, this implies a lower bound to the risk of the ideally constrained TVD estimator as follows:
	\begin{equation*}
	\lim_{\sigma \to 0}\frac{1}{\sigma^2}{\rm MSE}(\hat{\theta}_{\mb V^*},\theta^*) \geq  \frac{c} {N^{1/2}}\,.
	\end{equation*}
	Here $c > 0$ is a universal constant. 
\end{theorem}

\begin{remark}
	The proof of the above theorem should be extendable when $\theta^*$ is the indicator of a 
	circle or a regular $n$ sided ($n > 4$) polygon or any other shape which is sufficiently 
	non rectangular.  See Remark~\ref{rem:impos} for more on this issue. Therefore, it seems 
	that the rectangular shape of the level sets is crucial for the faster 
	$\tilde{O}(N^{-3/4})$ rate to hold. 
\end{remark}

\subsection{Tuning free TVD}
We now state our final result which relates to the question we posed about removing the tuning parameter and still retaining a risk bound which is essentially the same as in Theorem~\ref{Thm:1}. Choosing the tuning parameter is an important issue in applying the TVD methodology for denoising. The usual way out is to do some form of cross validation. There are some proposals available in the literature; see~\cite{solo1999selection},~\cite{osadebey2014optimal},~\cite{langer2017automated}. Soon after we uploaded our paper, a different tuning parameter free method appeared in \cite{ortelli2019oracle} which also achieves the optimal worst case $\tilde{O}(\mb V/\sqrt{N})$ rate of convergence. See Section~\ref{sec:ortelli_geer_a} for a comparison of our method with the one proposed in \cite{ortelli2019oracle}.

Our goal here is to construct a tuning parameter free estimator of $\theta^*$ which adapts to the true value of $\TV_{\mathrm{norm}}(\theta^*).$ The inspiration for this task comes from~\cite{chatterjee2015high} where the author gives a general recipe to construct tuning parameter free estimators in Gaussian mean estimation problems when the truth is known to have small value of some known norm. Even though the total variation functional is not a norm but a seminorm, the general idea in~\cite{chatterjee2015high} can be extended as we will show. However, the estimator of~\cite{chatterjee2015high} is a randomized estimator whereas in our case we construct a non randomized version. The following is a description of our tuning free estimator.

Let $\textbf{1}$ denote the $n \times n$ matrix consisting solely of ones. For any matrix $\theta \in \R^{n \times n},$ we define $\overline{\theta} \coloneqq \frac{1}{n^2} \sum_{i = 1}^{n} \sum_{j = 1}^{n} \theta[i, j]$ to be the mean of $\theta$. Define the estimator
\begin{equation}\label{estnew}
\hat{\theta}_{\mathrm{notuning}} \coloneqq \overline{y} \,\mb 1 \:\:\:\:\: + \:\:\:\:\: \argmin_{\{v \in \R^{n \times n}:\, \overline{v} = 0,\, \|y - \overline{y} \,\mb{1} - v\|^2 \leq (n^2 - 1) \hat{\sigma}^2\}} \TV_{\mathrm{norm}}(v) 
\end{equation}

where $\hat{\sigma}$ is an estimator of $\sigma$ defined as follows:
\begin{equation}\label{def:sigma}
\hat{\sigma} \coloneqq \frac{\TV_{\mathrm{norm}}(y)}{\E\:\TV_{\mathrm{norm}}(Z)} = \frac{\sqrt{\pi}\:\TV_{\mathrm{norm}}(y)}{4\:(n - 1)}\,.
\end{equation}

The intuition behind the estimator defined above is as follows. The estimation of $\theta^*$ is done by estimating the two orthogonal parts $\overline{\theta^*}\:\textbf{1}$ and 
$\theta^* - \overline{\theta^*} \textbf{1}$ separately. The first part is 
estimated by $\overline{y}\:\textbf{1}$. To estimate $\theta^* - \overline{\theta^*} \textbf{1},$ we use a Dantzig Selector type (see~\cite{candes2007dantzig}) version of the TVD estimator, which computes a zero mean matrix with the least total variation subject to being within a Euclidean ball of a suitable radius around the centered data matrix $y - 
\overline{y} \,\mb 1$. A good choice of this radius actually depends on the true $\sigma$ and hence as an intermediate step, we have to estimate $\sigma$ in the process which is denoted by $\hat{\sigma}.$ The main idea behind our construction of $\hat{\sigma}$ here is the fact that $\TV_{\mathrm{norm}}(\theta^*)$ is small compared to $\TV_{\mathrm{norm}}(Z)$ and hence $\TV_{\mathrm{norm}}(y) = \TV_{\mathrm{norm}}(\theta^* + \sigma Z)$ approximately equals $\sigma \TV_{\mathrm{norm}}(Z).$ We can then use concentration properties of the $\TV_{\mathrm{norm}}(Z)$ statistic to show that $\tfrac{\TV_{\mathrm{norm}}(Z)}{\E \TV_{\mathrm{norm}}(Z)}$ is approximately equal to $1.$ The following theorem supplies a risk bound for $\hat{\theta}_{\mathrm{notuning}}.$

\begin{theorem}\label{Thm:notuning}
	We have the following risk bound for our tuning free estimator:
	\begin{equation*}
	{\rm MSE}(\hat{\theta}_{\mathrm{notuning}},\theta^*) \leq C \big(\sigma \frac{\mb V^*}{\sqrt{N}} \log (\e n)^{5/2} + \frac{(\mb V^*)^2}{N} + \frac{\sigma^2}{\sqrt{N}}\big)\,
	\end{equation*}
	where $C$ is a universal constant. 
\end{theorem}

\begin{remark}
	Note that the above bound is meaningful only when $\lim_{N 
		\to \infty}\tfrac{\mb V^*}{\sqrt{N}} = 0$. Therefore in this regime, $\tfrac{(\mb V^*)^2}{N}$ is a lower order term. Thus, Theorem~\ref{Thm:notuning} basically says that the MSE of $\hat{\theta}_{\mathrm{notuning}}$, up to multiplicative log factors and an additive factor $\tfrac{\sigma}{\sqrt{N}}$, scales like $\tfrac{\mb V^*}{\sqrt{N}}$. In light of Remark~\ref{rem1} we can say that $\hat{\theta}_{\mathrm{notuning}}$ is minimax rate optimal (up to log factors) over $\{\theta \in \R^{n \times n}: \TV_{\mathrm{norm}}(\theta) \leq \mb V\}$, simultaneously for any sequence of $\textbf{V}$ (depending on $n$) which is bounded below by a 
	constant and above by $\sqrt{N}$. 
	To the best of our knowledge, this is the first result demonstrating such an estimator which is completely tuning free. 
\end{remark}

\section{Comparison with existing results, simulation studies and discussions}\label{sec:discuss}
To place our theorems in context, it is worthwhile to compare and relate our results with a couple of recent papers. We also discuss some issues related to our results.

\subsection{Comparison with \cite{hutter2016optimal}}
Let us compare our risk bound in Theorem~\ref{thm:adap} to the adaptive risk bound (Theorem~\ref{thm:hrig}) of~\cite{hutter2016optimal} when the truth $\theta^*$ is piecewise constant on a few axis aligned rectangles. Both of these theorems prove statements about tuned TVD estimators. Considering the very simple case when $\theta^*$ is of the following form:
\[
\theta^{*} =
\left[ {\begin{array}{cc}
	\mb 0_{n \times n/2} & \mb 1_{n \times n/2} \\
	\end{array} } \right]
\]
we have already mentioned in Section~\ref{sec:intro} that $\|D \theta^*\|_0 = O(\sqrt{N}).$ 
Thus, Theorem~\ref{thm:hrig} gives us an upper bound on the MSE scaling like 
$\tilde{O}(1/\sqrt{N})$ whereas our Theorem~\ref{thm:adap} gives a faster rate of 
convergence scaling like $\tilde{O}(1/N^{3/4}).$ More generally, if $\theta^*$ is piecewise 
constant on $k$ axis aligned rectangles with bounded aspect ratio and roughly equal size, it 
can be checked that $\|D \theta^*\|_0 \approx \sqrt{k^* N}.$ This means that 
Theorem~\ref{thm:hrig} gives us an upper bound on the MSE scaling like 
$\tilde{O}(\sqrt{k^*/N}).$ Compare this to Theorem~\ref{thm:adap} which gives a rate of 
convergence scaling like $\tilde{O}((k^*)^{5/4}/N^{3/4}).$ Thus, in the small $k^*$ regime 
when $k^* < N^{1/3}$, Theorem~\ref{thm:adap} provides a faster rate of convergence. 
This is the main contribution of this paper and to the best of our knowledge is the first of its kind in the literature.

\subsection{Comparison with~\cite{guntuboyina2017adaptive}} \label{sec:guntu}
As mentioned in Section~\ref{sec:intro}, one of our motivating factors behind investigating adaptivity of the 2D TVD estimator was its success in optimally estimating piecewise constant vectors in the 1D setting. Theorem~$2.2$ in~\cite{guntuboyina2017adaptive} gives a 
$\tilde{O}(k^*/n)$ rate for the ideally tuned constrained 1D TVD estimator when the truth 
$\theta^*$ is piecewise constant with $k^*$ pieces or blocks and each block satisfies a 
certain minimum length condition. In a sense, our Theorem~\ref{thm:adap} is a natural 
successor, giving the corresponding result in the 2D setting. Our bounded aspect ratio 
condition is the 2D version of the minimum length condition. A consequence of 
Theorem~\ref{thm:adap} and Theorem~\ref{thm:lbd} is that, in contrast to the 1D setting, the 
ideally tuned constrained TVD estimator can no longer obtain the oracle rate of convergence 
$\tilde{O}(k^*/n)$ in the 2D setting.

The proof of Theorem $2.2$ in~\cite{guntuboyina2017adaptive} was done by bounding the Gaussian widths of certain tangent cones. Our proof of Theorem~\ref{thm:adap} also adopts 
the same strategy and precisely characterizes the tangent cone $T_{K( V^*)}(\theta^*)$ 
(defined in~\eqref{eq:deftancone}) for piecewise constant $\theta^*$ and then bounds its 
Gaussian width. The main idea in~\cite{guntuboyina2017adaptive} was to observe that any unit 
norm element of the tangent cone is nearly made up of two monotonic blocks in each constant 
block of $\theta^*.$ Then the available metric entropy bounds for monotone vectors were used 
to bound the Gaussian width. A crucial ingredient in this proof is the well-known fact that any univariate function of bounded variation has a canonical representation as a difference of two monotonic functions. However, it is not clear at all how to adapt 
such a strategy to the 2D setting. In particular, it is not nearly as natural and convenient to express a matrix of bounded variation as a difference of two 
bi-monotone matrices. Our computation of Gaussian width of the tangent cone is therefore essentially \textit{two dimensional} and involves judicious recursive partitioning 
in both dimensions. We believe that our Gaussian width computations, especially the proof of Proposition~\ref{prop:gwonebdry}, consist of new techniques and are potentially useful for problems of similar flavor.

\subsection{Comparison with~\cite{ortelli2019oracle}}\label{sec:ortelli_geer_a}
At the latter stages of preparation of this manuscript we became aware of an independent work by~\cite{ortelli2019oracle} which is related to our manuscript. In~\cite{ortelli2019oracle}, the authors give a general technique to derive slow ($\ell_1$) and fast ($\ell_0$) rates for penalized TVD estimators and its square root version on general graphs. Thus, there seems to be two routes for obtaining fast rates for TVD. One goes through the route of bounding Gaussian width of an appropriate tangent cone to derive fast rates for the constrained TVD estimator; as done here in this manuscript as well as in~\cite{guntuboyina2017adaptive}. The other route; followed by~\cite{hutter2016optimal} and generalized by~\cite{ortelli2019oracle} is based on bounding the so called \textit{compatibility factor}.~\cite{ortelli2019oracle} show how to bound this compatibility factor for specific graphs such as the $1d$ grid graph and the $1d $ cycle graph. To the best of our knowledge, bounding the compatibility factor for piecewise constant functions on axis aligned rectangles for a $2d$ grid remains an open problem. Thus, as far as we are aware, the work in this manuscript proving fast rates of convergence on $2d$ grid graph for piecewise constant functions on axis aligned rectangles is the first of its type in the literature.

The work in~\cite{ortelli2019oracle} also proposes a general technique to obtain slow rates for a square root version of the TVD estimator. Similar to our paper, \cite{ortelli2019oracle} also considers the case when the noise variables are i.i.d. Gaussian. The advantage of using this square root version is that the tuning parameter does not need to depend on the unknown parameter $\sigma$. While the theoretically recommended choice of the tuning parameter $\lambda$ in \cite[Corollary~4.13]{ortelli2019oracle} does not depend on the noise variance $\sigma^2$, there is however the presence of an unspecified 
large universal constant $C$. It is not clear to us whether this $C$ can be explicitly 
specified. On the other hand, our tuning free estimator is 
explicitly specified and involves no unknown constants. We think the analysis of our tuning 
free estimator is also reasonably clean with the sources of the various possible errors made transparent in the proof. This is why we believe that our tuning free estimator provides a theoretically valid and possibly useful alternative to the square root regularization approach. Just to be clear, we are not claiming any optimality of our tuning free method, our intention is to demonstrate one theoretically valid way to obtain a minimax rate optimal tuning free estimator.

\subsection{Necessity of bounded aspect ratio condition in theorem~\ref{thm:adap}}\label{sec:bddasp}
We think a bounded aspect ratio condition would actually be necessary for the $O(N^{-3/4})$ rate to hold in Theorem~\ref{thm:adap}. For instance, consider the sequence of matrices $\theta^*$ such that $\theta^*[i,j] = \mathbb{I}\{j = n\}.$ Clearly, the rectangular level sets of the sequence of matrices $\theta^*$ do not satisfy the bounded aspect ratio condition. By an argument similar to the one used to prove our lower bound results in Theorem~\ref{thm:lbd} and Theorem~\ref{thm:impos}, one can show that the ${\rm MSE}(\hat{\theta}_{\mb V^*},\theta^*) \geq c N^{-1/2}.$ We have also verified this scaling of the MSE in our numerical experiments. 

The bounded aspect ratio condition says that the rectangular level sets of $\theta^*$ should not be too skinny or too long. In our proof, the bounded aspect ratio is needed for similar reasons as a minimum length condition is needed for the length of the pieces in the 1D setting; see~\cite{guntuboyina2017adaptive},~\cite{dalalyan2017tvd}.

\subsection{On obtaining the oracle rate for piecewise constant signals}\label{sec:pc}
In light of Theorem~\ref{thm:adap} and Theorem~\ref{thm:lbd} we can say the following statement. When the truth $\theta^*$ is piecewise constant on $k^*$ axis aligned rectangles, the TVD estimator \textit{cannot} attain the oracle rate of convergence scaling like $O(k^*/N).$ The question that now arises is whether there exists \textit{any} estimator which attains the $\tilde{O}(k^*/N)$ rate of convergence for all piecewise constant truths {\em as well as} the minimax rate $\tilde{O}{(\mb V^* / \sqrt{N})}$? Furthermore, can this estimator be chosen so that it is computationally efficient? These questions led us to examine decision tree estimators which are different from the TVD type estimators. We would like to point out here that in the paper~\cite{chatterjeegoswamiadaptive} we have been able to demonstrate computationally efficient estimators which attain both the aforementioned goals.

Apart from~\cite{chatterjeegoswamiadaptive}, some other recent papers have also sprung up which target piecewise rectangular signals. The papers \cite{ortelli2019tvd} and \cite{fang2019multivariate} study a different version of the TVD estimator which is also 
termed as the {\em Hardy Krause estimator}. As far as we understand, this estimator is well 
suited for estimating piecewise rectangular signals as it actually fits piecewise 
rectangular estimates. For a general signal with $k^*$ rectangular pieces (with some regularity conditions on the rectangular pieces), the rate proved by 
\cite{ortelli2019tvd} is $\tilde O(({k^*})^{3/2}/N)$ which is better than $\tilde{O}\big(({k^*})^{5/4}/N^{3/4}\big)$. Notice that the $\tilde O(({k^*})^{3/2}/N)$ rate still does not match the near oracle rate $\tilde{O}({k^*}/N)$ which has been obtained in~\cite{chatterjeegoswamiadaptive}.

\subsection{Constrained vs Penalized}
In this paper, we have focussed on the constrained version of the 2D TVD estimator. As mentioned in the introduction, the penalized version is also quite popular. In the low $\sigma$ limit, it can be proved that the constrained estimator $\hat{\theta}_{\mb V}$ with $\mb V = \mb V^* = \TV_{\mathrm{norm}}(\theta^*)$ is better than the penalized estimator for every deterministic choice of $\lambda.$ More precisely, we have for all $\lambda \geq 0$,
\begin{equation*}
\lim_{\sigma \downarrow 0} \frac{1}{\sigma^2}\:{\rm MSE}(\hat{\theta}_{\mb V^*},\theta^*) < \lim_{\sigma \downarrow 0} \frac{1}{\sigma^2}\:{\rm MSE}(\hat{\theta}_{\lambda},\theta^*).
\end{equation*}
The above inequality follows from the results of~\cite{oymak2013sharp} as described in Section 5.2 in~\cite{guntuboyina2017adaptive}. Since our main question here is whether faster/adaptive rates are possible for piecewise constant matrices, it is therefore natural to first study the constrained version with ideal tuning. A possible next step is to investigate whether a similar $N^{-3/4}$ rate is atttained by the penalized TVD estimator and if so, for what range of the tuning parameter $\lambda$.


\subsection{Simulation studies}
We consider three distinct sequences of matrices to 
facilitate comparison. We consider the simplest piecewise constant matrix $\theta^{\mathrm{two}} \in \R^{n \times n}$ where $\theta^{\mathrm{two}} \coloneqq \I\{j >  n/2\}.$ Hence 
$\theta^{\mathrm{two}}$ just takes two distinct values. The next matrix $\theta^{\mathrm{four}}$ is a block matrix with four constant blocks.
\[
\theta^{\mathrm{four}}\coloneqq
\left[ {\begin{array}{cc}
	1_{n/2 \times n/2} & 2_{n/2 \times n/2} \\
	0_{n/2 \times n/2} & 1_{n/2 \times n/2} \\
	\end{array} } \right]
\]
Finally, we also consider a $n \times n$ matrix $\theta^{\mathrm{worst}} \coloneqq 
\I\{i + j > n\}$. Clearly, $\theta^{\mathrm{worst}}$ does not have a block constant structure. 
For the matrix $\theta^{\mathrm{worst}}$ we incur the worst case rate $\tilde{O}(N^{-1/2})$ as 
shown in Theorem~\ref{thm:impos}; hence the name. The noise variance has been set to 1 for all the numerical experiments reported in this section.

The dependence of the MSE with $N = n^2$ can be experimentally checked as follows. We can estimate the MSE for a fixed $n$ by Monte Carlo repetitions and then iterate this for a grid of $n$ values. We then plot log of the estimated MSE with log $N$ and fit a least squares line to the plot. The slope of the least squares line then gives an indication of the correct exponent of $N$ in the MSE. Figure~\ref{fig:fig1} is such a plot for the ideally tuned constrained TVD estimator.
\begin{figure}
	\begin{center}
		\includegraphics[scale = 0.25]{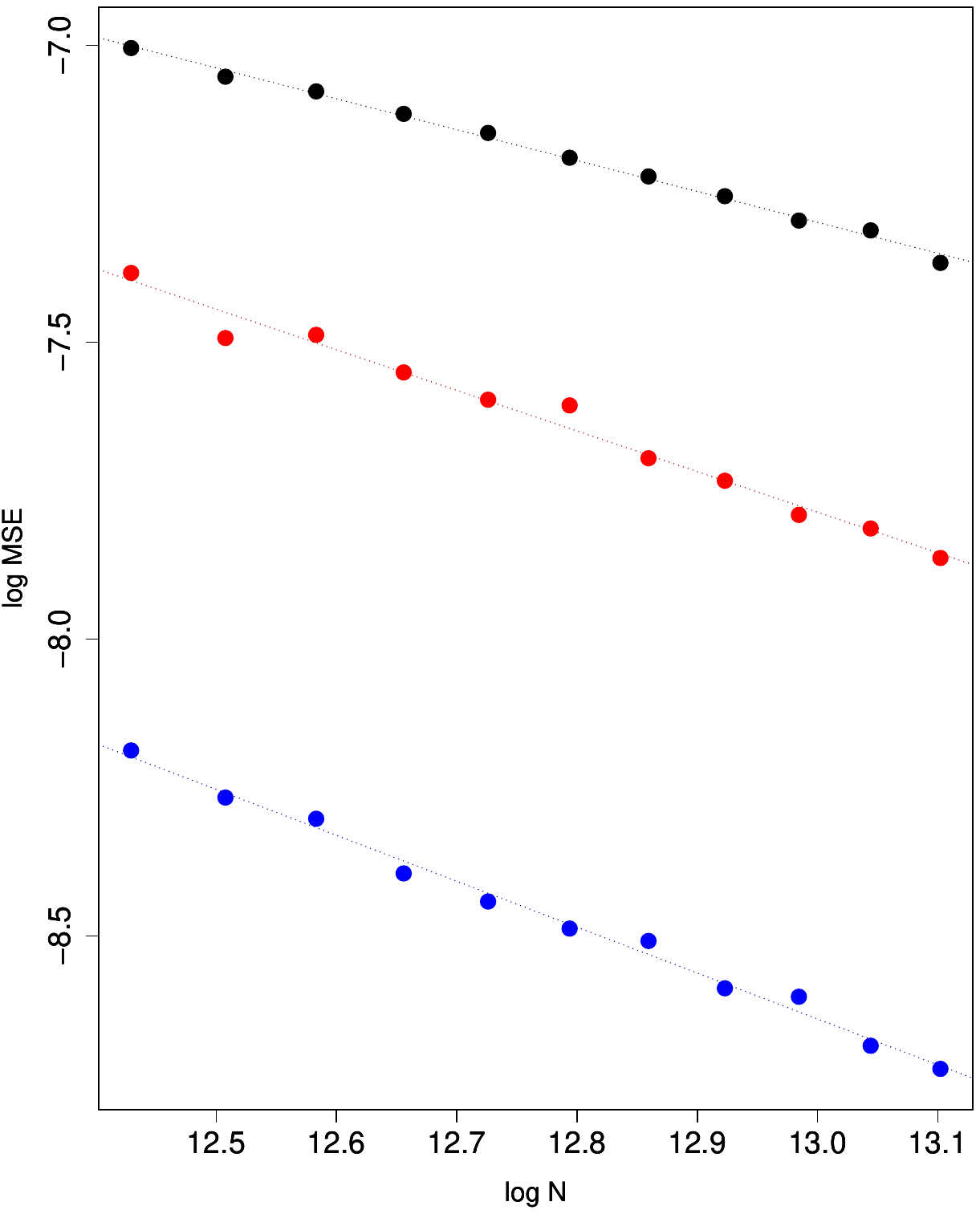}
		\caption{\textit{The MSE of the ideally tuned TVD estimator $\hat{\theta}_{\mb V^*}$ is estimated with $50$ Monte carlo repetitions for a grid of $n = \sqrt{N}$ ranging from $500$ to $700$ in increments of $20.$ The true matrices were taken to be $\theta^{\mathrm{two}}$ ({\color{blue} blue}), $\theta^{\mathrm{four}}$ ({\color{red} red}) and $\theta^{\mathrm{worst}}$ ({\color{black} black}). In each case, we have chosen the ideal tuning parameter to allow fair comparison. We plot log of estimated MSE versus log $N$ where log is taken in base $\e.$ The points are the estimated log MSE and the dotted lines are the least squares line fitted to the points. The least squares slope for $\theta^{\mathrm{two}}$ is $-0.73$ and for $\theta^{\mathrm{four}}$ is $- 0.68$ which is considerably lower than the slope for the matrix $\theta^{\mathrm{worst}}$ which is $-0.52.$}}
		\label{fig:fig1}
	\end{center} 
\end{figure}

\begin{figure}
	\begin{center}
		\includegraphics[scale = 0.25]{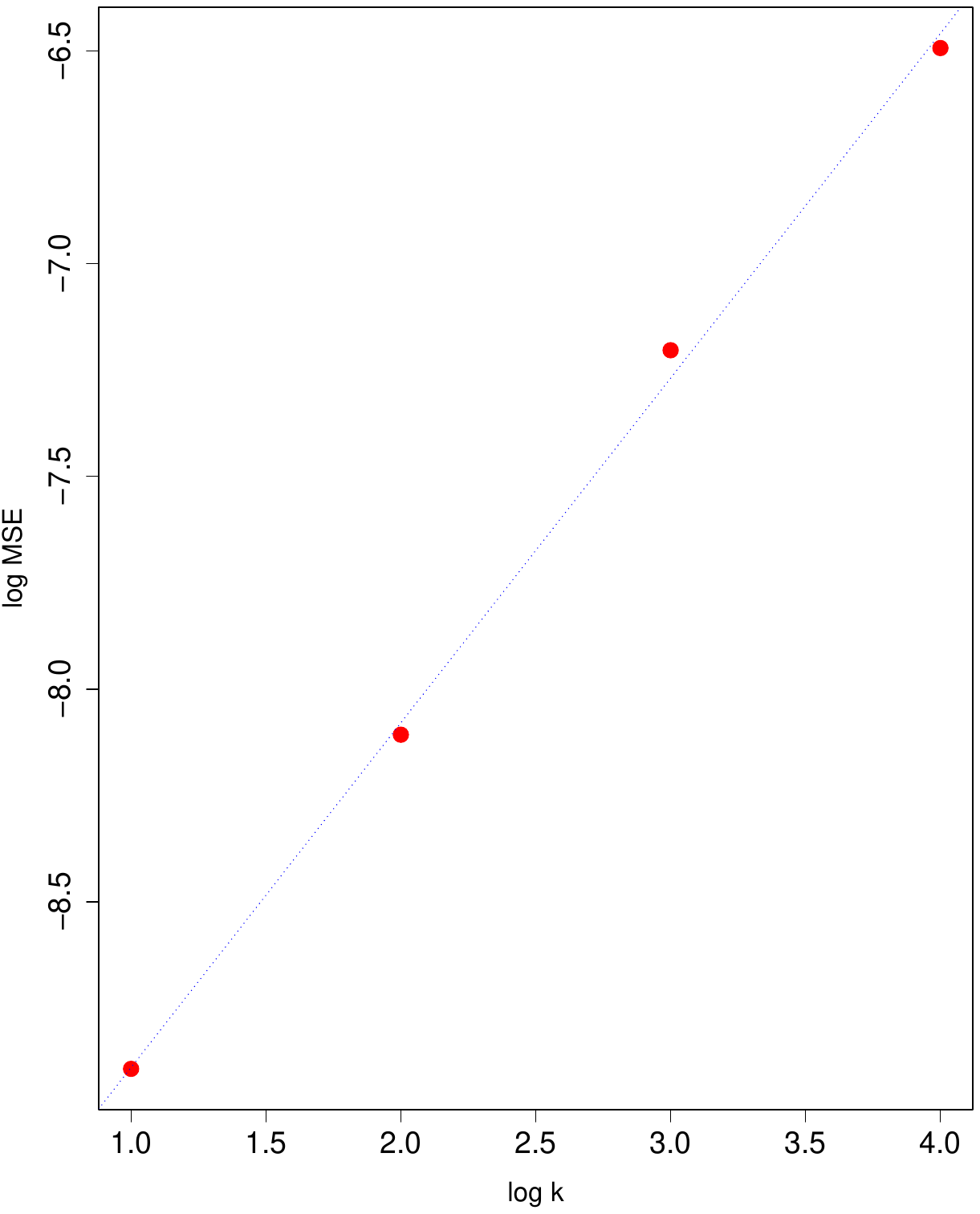}
		\caption{\textit{The MSE of the ideally tuned TVD estimator is estimated with $50$ Monte carlo repetitions when $n = 800.$  The true matrices were taken to be such that the number of rectangular level sets is $2,4,8,16.$ In each case, we have chosen the ideal tuning parameter to allow fair comparison. We have also normalized the matrices so that $\mb V^* = 1$. We plot log of estimated MSE versus $\log_2 k$ where $k = 2,4,8,16.$ The points are the estimated log MSE and the dotted lines are the least squares line fitted to the points. The least squares slope is $0.81.$}}
			\label{fig:fig2}
	\end{center} 
\end{figure}

In Figure~\ref{fig:fig1}, the risk is seen to be minimum for $\theta^{\mathrm{two}}$ followed by $\theta^{\mathrm{four}}$ and then $\theta^{\mathrm{worst}}.$ The slope for $\theta^{two}$ and $\theta^{\mathrm{four}}$ came out to be $-0.73$ and $-0.68.$ This agrees well with Theorem~\ref{thm:adap} and Theorem~\ref{thm:lbd} which says that the MSE decays at the rate $n^{-0.75}$ upto log factors. 
For the matrix $\theta^{\mathrm{worst}}$ the slope turned out be $-0.52$ which is in agreement with the worst case $\tilde{O}(N^{-1/2})$ rate given in Theorem~\ref{thm:impos}.

To investigate the dependence of MSE with the number of rectangular level sets $k(\theta^*)$, we took four matrices. The first two are $\theta^{\mathrm{two}},\theta^{\mathrm{four}}$ and the last two are obtained by further binary division so that the number of rectangular level sets is $8,16$ respectively. We normalized the matrices such that $\mb V^* = 1.$ We fixed $n = 800$ and did $50$ iterations of Monte Carlo simulations for each of the four matrices. We then plotted $\log {\rm MSE}$ versus $\log_2 k$ (see Figure~\ref{fig:fig2}) where $k = 2,4,8,16.$ The slope of the least squares line we got is $0.81.$ This suggests that our exponent of $k$ ($= 1.25)$ in the risk bound in Theorem~\ref{thm:adap} may not be optimal.

\begin{figure}
	\begin{center}
		\includegraphics[scale=0.25]{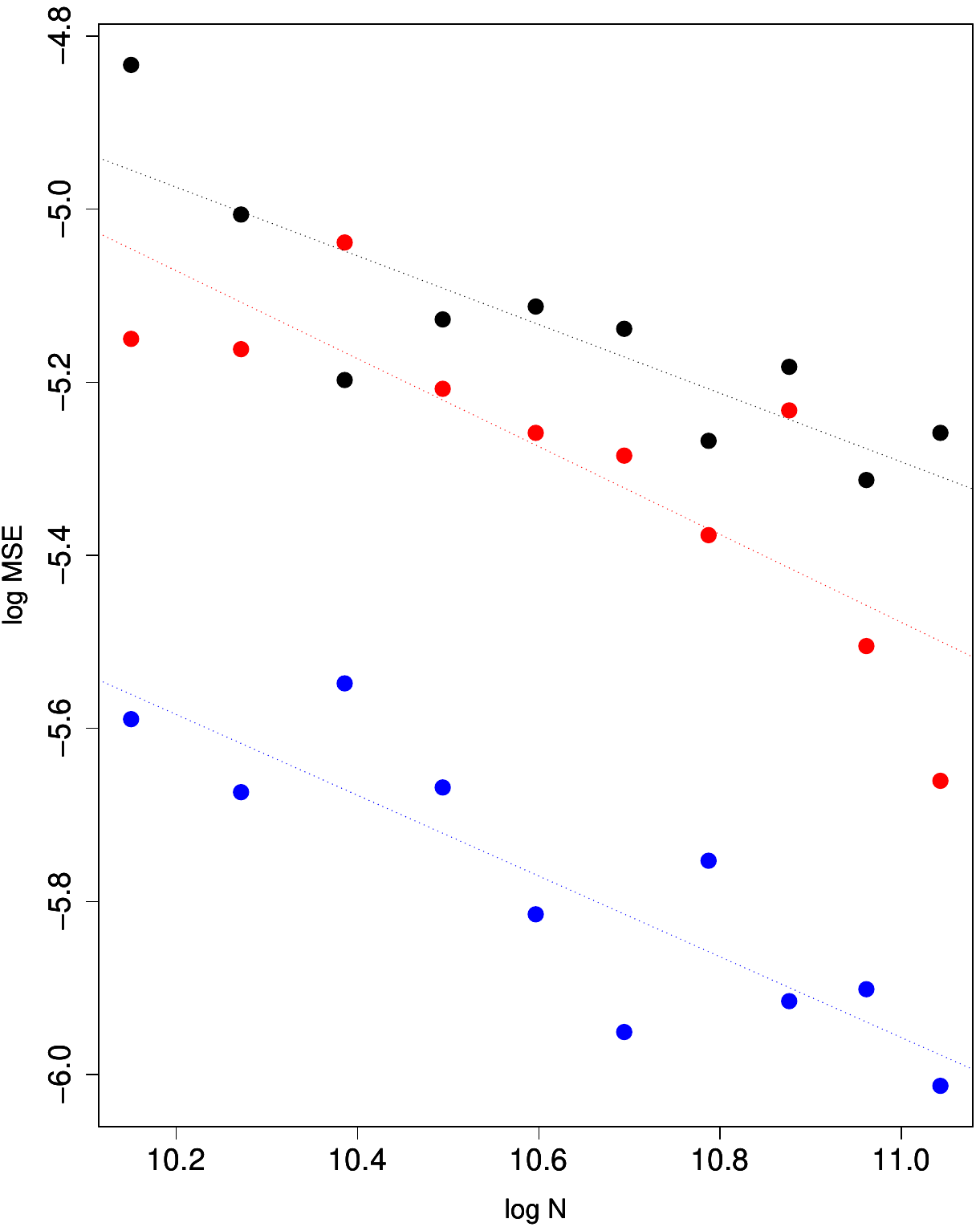}
		\caption{\textit{The MSE of our tuning free estimator is estimated with $50$ Monte carlo repetitions for a grid of $n = \sqrt{N}$ ranging from $160$ to $250$ in increments of $10.$ The true matrices were taken to be $\theta^{\mathrm{two}}$ ({\color{blue} blue}), $\theta^{\mathrm{four}}$ ({\color{red} red}), and $\theta^{\mathrm{worst}}$ ({\color{black} black}). We plot log of estimated MSE versus log $N$ where log is taken in base $e.$ The  circular points are the estimated log MSE and the dotted lines are the least squares line fitted to the points. The slopes of the least squares lines are $-0.47$,$-0.51$,$-0.40$ for $\theta^{\mathrm{two}}$, $\theta^{\mathrm{four}}$, $\theta^{\mathrm{worst}}$ respectively.}}
		\label{fig:fig3}
	\end{center}
\end{figure}
To assess the risk of our fully data driven estimator $\hat{\theta}_{\mathrm{notuning}}$, we again consider the three matrices $\theta^{\mathrm{two}},\theta^{\mathrm{four}}$ and $\theta^{\mathrm{worst}}$ respectively. Figure~\ref{fig:fig3} is a plot of log MSE versus $\log n.$

The simulations in Figure~\ref{fig:fig3} suggest that our estimator has MSE decaying at a $O(1/\sqrt{N})$ 
rate for all three matrices. The slope of all three least 
squares lines are reasonably close to $-0.5.$ This matches 
the rate given in Theorem~\ref{Thm:notuning}. 
 However, our tuning free estimator does not seem to be adaptive to piecewise constant structure like the constrained TVD estimator with 
ideal tuning. 

\begin{figure}
	\begin{center}
		\includegraphics[scale = 0.25]{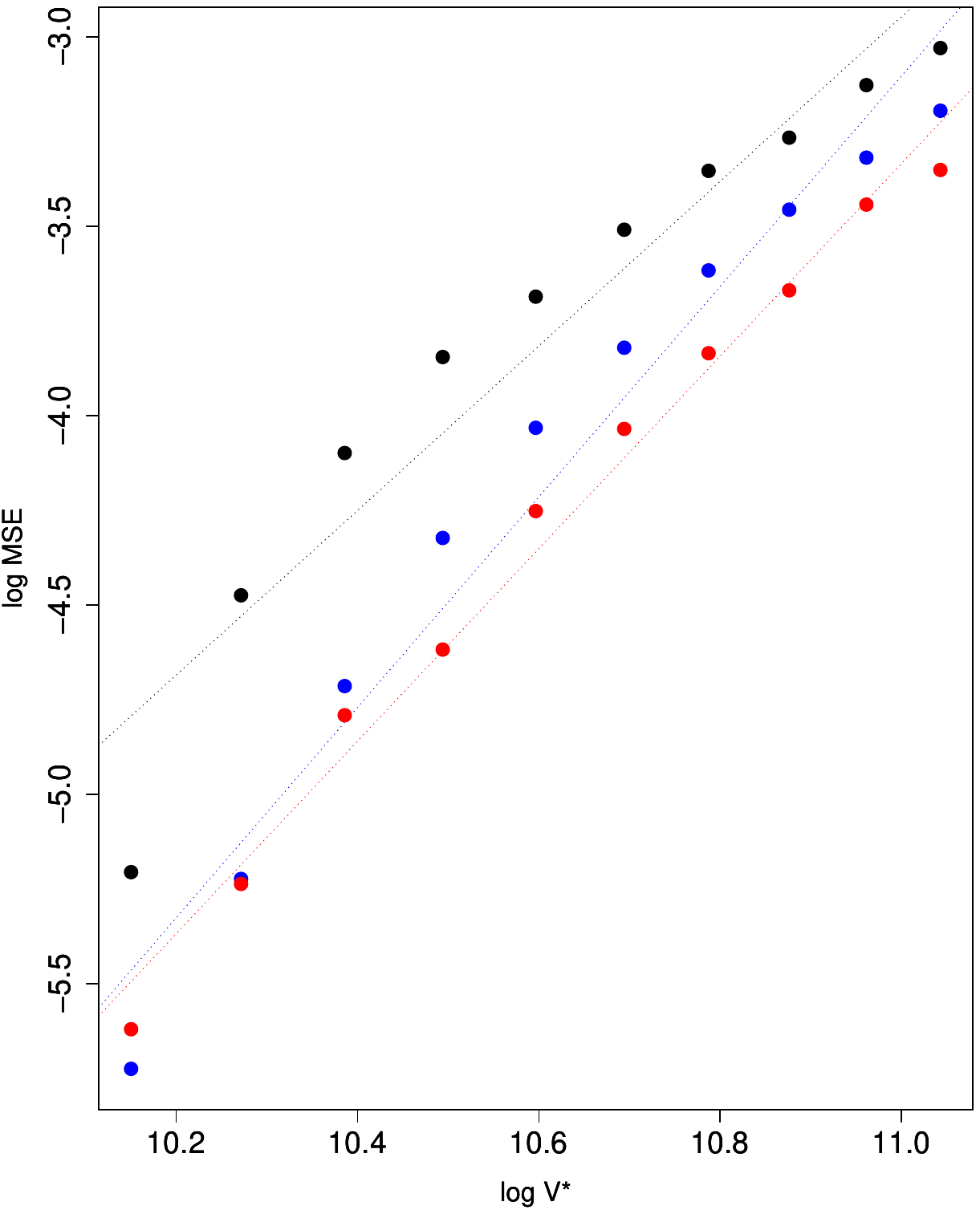}
		\caption{\textit{The MSE of the tuning free TVD estimator is estimated with $50$ Monte carlo repetitions for a grid of $\mb V^* \in [10]$ and $n = 200.$ The true matrices were taken to be $\theta^{\mathrm{two}}$ ({\color{blue} blue}), $\theta^{\mathrm{four}}$ ({\color{red} red}) and $\theta^{\mathrm{worst}}$ ({\color{black} black}) properly normalized. We plot log of estimated MSE versus log $N$ where log is taken in base $\e.$ The points are the estimated log MSE and the dotted lines are the least squares line fitted to the points. The least squares slope for $\theta^{\mathrm{two}}$ is $1.16$, for $\theta^{\mathrm{four}}$ is $1.07$ and for the matrix $\theta^{\mathrm{worst}}$ it is $0.94.$}}
		\label{fig:fig4}
	\end{center} 
\end{figure}
To investigate the dependence of the risk of our tuning free estimator on $\mb V^*$, for each of the three matrices $\theta^{\mathrm{two}}$, $\theta^{\mathrm{four}}$, $\theta^{\mathrm{worst}}$, we normalized the matrix such that $\mb V^* = 1,2,\dots,10$. We fixed $n = 200$ and did $50$ iterations for each $\mb V^*$ and each matrix. We then plotted log MSE versus log $\mb V^*$ (see Figure~\ref{fig:fig4}) and fitted a least squares line. The slopes for each of these three matrices came out to be $1.16,1.07,0.94$ respectively. This suggests that the right exponent of $V^*$ is $1$ and our risk bound has the right dependence on $\mb V^*.$ 

\section{A generic approach towards bounding Gaussian widths}
\label{sec:chaining}
Let us recall from Section~\ref{sec:gauss_width_bnd_results} that the Gaussian width of a set $K \subset \R^n$ is defined as $$\gw(K) = \E \sup_{\theta \in K} \langle Z, \theta \rangle$$ where $Z = Z_n\sim N(0_n, I)$ and $\langle \cdot\,, \cdot \rangle$ is the usual Euclidean inner product between two vectors. Our principal result in this section provides an upper bound on $\gw(K)$ in terms of the numbers and dimensions of its covering (linear) subspaces. This result executes and adapts the idea of chaining (see, e.g., \cite[Theorem~$5.24$]{van2014probability}) to the case when the covering sets are linear subspaces of $\R^n$. To this end let us define, for any $\ep > 0$, an $\epsilon$ {\em subspace cover} of $K$ to be any finite collection $\mathcal{S}$ of linear subspaces  of $\R^n$ such that 
$$\sup_{\theta \in K}\:{\rm dist}\Big(\theta, \: \bigcup_{S \in \mathcal S}S\Big) \leq \epsilon$$  where ${\rm dist}(A, B)$ denotes the Euclidean distance between the sets $A$ and $B$. We denote by $\diam(K)$ the diameter of $K$ which we assume to be finite. Also for any $t > 0$, we denote by $B_n(t)$ the $t$-Euclidean ball $\{\theta \in \R^n: \|\theta\| \leq t\}$ where $\|.\|$ is the Euclidean norm. We will often drop the subscript $n$ and just write $B(t)$ when the dimension is clear from the context.
\begin{proposition}[Gaussian width bound]\label{prop:chaining}
	For every $\epsilon \in (0, \diam(K))$, let $\mathcal{S}_{\epsilon}$ be an $\epsilon$ subspace cover of $K$. Also let $k_1 > k_0$ be integers with $k_0$ being the smallest integer satisfying $2^{-k_0} \geq \diam(K)$. Then we have
	\begin{equation*}
	\gw(K) \leq 3\sum_{k = k_0 + 1}^{k_1} 2^{-k} \big[\max_{S \in \mathcal{S}_{2^{-k}}} \sqrt{2\:\dim(S)} + 2\:\sqrt{\log |\mathcal{S}_{2^{-k}}|} + 1\big] + \sqrt{n}\, 2^{-k_1}.
	\end{equation*} 
\end{proposition}

\begin{proof}
	For any $\theta \in K$ and integer $k$ such that $2^{-k} < \diam(K)$, let $\theta_k$ denote a point in $\mathcal N_k \coloneqq \cup_{S \in \mathcal S_{2^{-k}}} S$ such that 
	\begin{equation*}
	{\rm dist}(\theta, \theta_k)   =  {\rm dist}(\theta, \mathcal N_k).
	\end{equation*}
	Such a point always exists since $\mathcal N_k$ is a finite union of linear subspaces. 
	When $\diam(K) \ge 2^{-k}$, on the other hand, we simply choose $\theta_k$ to be some fixed but arbitrary point $\theta_K$ in $K$. By definition, we thus have
	\begin{equation}
	\label{eq:projection}
	\|\theta - \theta_{k}\| \leq 2^{-k}
	\end{equation}
	for all $k \in \Z$. Let us now write for every $\theta \in K$,
	\begin{equation*}
	\theta = \theta_{K} + \sum_{k = k_0 + 1}^{k_1} (\theta_{k} - \theta_{k - 1}) + (\theta - \theta_{k_1})
	\end{equation*}
	so that
	\begin{equation*}
	\gw(K) = \E \sup_{\theta \in K} \langle Z, \theta \rangle \leq \E \langle Z, \theta_{K} \rangle +  \sum_{k = k_0 + 1}^{k_1} \E \sup_{\theta \in K} \langle Z,\theta_{k} - \theta_{k - 1} \rangle + \E \sup_{\theta \in K}  \langle Z,\theta - \theta_{k_1} \rangle.
	\end{equation*}	
	The first term on the right hand side above is 0, whereas the third time is bounded by 
	$\sqrt{n} 2^{-k_1}$ in view of the Cauchy-Schwarz inequality, display~\eqref{eq:projection} and the standard bound $\E \|Z\| \leq \sqrt{n}$. Therefore we can conclude the proof if we can show
	\begin{equation*}
	\E \sup_{\theta \in K} \langle Z,\theta_{k} - \theta_{k - 1} \rangle \le 3 \cdot 2^{-k} \big[\max_{S \in \mathcal{S}_{2^{-k}}} \sqrt{2\:\dim(S)} + 2\:\sqrt{\log |\mathcal{S}_{2^{-k}}|}\,\big]
	\end{equation*}
	for every integer $k$ satisfying $2^{-k} < \diam(K)$. To this end observe that
	\begin{equation*}
	\|\theta_k - \theta_{k - 1}\| \leq \|\theta_k - \theta\| + \|\theta_{k - 1} - \theta\| \leq 3 \cdot 2^{-k}
	\end{equation*}
	in view of \eqref{eq:projection} and $\theta_k - \theta_{k - 1} \in \mathcal{M}_k$ where $\mathcal{M}_k \coloneqq \{S_1 + S_2: S_1 \in \mathcal{N}_{2^{-k}}, S_2 \in \mathcal{N}_{2^{-(k - 1)}}\}$ is another finite collection of linear subspaces of $\R^n$. It is also clear from the definition that $|\mathcal{M}_k| \leq |\mathcal{N}_{2^{-k}}| |\mathcal{N}_{2^{-(k - 1)}}|$. All these observations bring us to the setting of:

	\begin{lemma}[Gaussian width for union of subspaces]
		\label{lem:gauss_width_union} 
		Let $\mathcal{S}$ be a finite collection of linear subspaces of $\R^n$ and $\Theta = \cup_{S \in \mathcal{S}} S \subset \R^n.$ In words, $\Theta$ is the union of subspaces in $\mathcal{S}$. Then we have
		\begin{equation*}
		\gw(\Theta \cap B(t)) \leq t\: \big[\max_{S \in \mathcal{S}} \sqrt{\dim(S)} + \sqrt{2\:\log  |\mathcal{S}|} + 1\big].
		\end{equation*}
	\end{lemma}
	Using Lemma~\ref{lem:gauss_width_union}, we can immediately deduce that
	\begin{equation*}
	\E \sup_{\theta \in K} \langle Z,\theta_{k} - \theta_{k - 1} \rangle \leq 3 \cdot 2^{-k} M
	\end{equation*}
	where $$M \coloneqq  \big[\max_{S_1 \in \mathcal{S}_{2^{-k}}, S_2 \in \mathcal{S}_{2^{-(k - 1)}}} \sqrt{\dim(S_1) + \dim(S_2)} + \:\sqrt{2\:\log |\mathcal{S}_{2^{-k}}| + 2\:\log |\mathcal{S}_{2^{-(k - 1)}}|} + 1\big].$$
	Now we can assume without any loss of generality that $|\mathcal{S}_{2^{-(k - 1)}}| \leq |\mathcal{S}_{2^{-k}}|$ as well as $$\max_{S \in \mathcal{S}_{2^{-(k - 1)}}} \dim(S) \leq \max _{S \in \mathcal{S}_{2^{-k}}} \dim(S),$$
	which finishes the proof of the proposition.
	
	\smallskip
	
	Let us now return to the proof of Lemma~\ref{lem:gauss_width_union}. Since $\Theta = t \, \Theta$, it follows from the definition of Gaussian widths that $\gw(\Theta \cap B(t)) = t\,\gw(\Theta \cap B(1))$ meaning we only need to work with $t = 1$.  We will use the following lemma involving only one linear subspace:
	\begin{lemma}\label{lem:1}
		For any linear subspace $S$ of $\R^n$ and $u \ge 0$, we have with probability at least $1 -  \exp(-\frac{u^2}{2}),$
		\begin{equation}\label{eq:lem0}
		\sup_{\theta \in S \cap B(1)}  \langle Z, \theta \rangle \leq \sqrt{\dim(S)} + u.
		\end{equation}
	\end{lemma}
	\begin{proof}
		We will use the well-known concentration inequality for Lipschitz functions of a Gaussian vector (see, e.g.~\cite[Theorem~7.1]{Ledoux01conc}). First of all notice that the random variable $f(Z) \coloneqq \sup_{\theta \in S \cap B(1)}  \langle Z, \theta \rangle$ is a Lipschitz function of $Z$ with Lipschitz constant $1$. It follows from the observation that, for any $z, z' \in \R^n$ and $\theta \in B(1)$,
		$$\langle z, \theta \rangle - \langle z', \theta \rangle = \langle z - z', \theta\rangle \le \|z - z'\|\|\theta\| \le \|z - z'\|$$
		where in the last but one step we used the Cauchy-Schwarz inequality. Therefore by the Gaussian concentration inequality mentioned in the beginning, we have for any $u \ge 0$
		$$\P (f(Z) - \E f(Z) \ge u) \le \exp(-\frac {u^2}{2}).$$
		Hence we can deduce the lemma upon showing that $\E f(Z) \le \sqrt{\dim(S)}$. To this end notice that $f(Z) = \|P_{S} Z\|$ where $P_S$ is the orthogonal projector onto the subspace $S$. Therefore, $f(Z)^2$ is a chi squared random variable whose degree of freedom equals $\dim(S)$ whence we get
		\begin{equation*}
		\E f(Z) \leq \sqrt{\E f(Z)^2} \leq \sqrt{\dim(S)}. \qedhere
		\end{equation*}
	\end{proof}

	Now, using a union bound followed by Lemma~\ref{lem:1} we get
	\begin{align*}
	\P\big(\sup_{\theta \in \Theta \cap B(1)} \langle Z, \theta \rangle \ge \max_{S \in \mathcal{S}} \sqrt{\dim(S)} + u\big) &\leq \sum_{S \in \mathcal{S}} \mathbb P\big(\sup_{\theta \in S \cap B(1)} \langle Z, \theta \rangle \ge \max_{S \in \mathcal{S}} \sqrt{\dim(S)} + u\big) \nonumber \\&\leq |\mathcal{S}| \exp(-\frac{u^2}{2}).
	\end{align*}
	Plugging in $u = \sqrt{2\:\log  |\mathcal{S}| + v^2}$ we obtain
	\begin{equation*}
	\P\big (\sup_{\theta \in \Theta \cap B(1)} \langle Z, \theta \rangle \ge \max_{S \in \mathcal{S}} \sqrt{\dim(S)} + \sqrt{2\:\log  |\mathcal{S}|} + v \big) \leq \exp(-\frac{v^2}{2}).
	\end{equation*}
	Integrating the above tail bound finishes the proof of Lemma~\ref{lem:gauss_width_union}.
\end{proof}

\begin{remark}
	A general and perhaps more standard way of bounding the Gaussian width of a set is through Dudley's entropy integral inequality (see \cite{Dudley67}). In this approach one first finds a ``good'' covering set corresponding to any given radius $r$ for the underlying set to obtain upper bounds on covering numbers which then enter an integral (after being transformed appropriately) bounding the Gaussian width. Proposition~\ref{prop:chaining} provides an alternative way when the covering sets are contained in finite unions of linear subspaces with comparable dimensions. For the purpose of the current article, this approach would save us some extraneous log factors in our bounds.
\end{remark}
\section{Proof of Theorem~\ref{Thm:1}}\label{Sec:5}
\label{sec:thm1}
We first set up some notations which would henceforth be used throughout the paper.
For a positive integer $n$, we will denote the subset of positive integers $\{1,\dots,n\}$ by $[n]$. Recall that in all the proofs of our results, we are going to use $\TV$ to denote the \emph{unnormalized} version of~\eqref{eq:TVdef} as defined in~\eqref{eq:unnorm}. 
Also we will use $V$ for the unnormalized total variation instead of the \emph{bold} $\mb V$ used for the {\em corresponding} normalized version.

Let us recall that the estimator $\hat{\theta}_{V}$ is the least squares estimator on the set 
\begin{equation}
\label{def:KnV}
K_n(V) \coloneqq \{\theta \in \R^{n \times n}: \TV(\theta) \leq V\}.
\end{equation}
We will often drop the subscript $n$ and just write $K(V)$ when the dimension is clear from the context. Below we adopt the standard approach of using the basic inequality defining least squares estimators to reduce our problem to controlling Gaussian widths.
\begin{lemma}\label{lem:basicineq}
	Under the same conditions as in the statement of Theorem~\ref{Thm:1} we have
	\begin{equation*}
	\E \|\hat{\theta}_{V} - \theta^*\|^2 \leq 2\:\sigma\: \E \sup_{\theta: \TV(\theta) \leq 2V, \,\overline{\theta} = 0} \langle Z,\theta \rangle + 2\:\sigma^2.
	\end{equation*}
\end{lemma} 
\begin{proof}
	Since $V \geq V^*$ we have the basic inequality $\|y - \hat{\theta}_{V}\|^2 \leq \|y - \theta^*\|^2$. Substituting $y = \theta^* + \sigma Z$ gives us
	\begin{align*}
	\|\theta^* - \hat{\theta}_{V}\|^2 &\leq 2 \langle \hat{\theta}_{V} - \theta^*, y - \theta^* \rangle = 2 \langle \hat{\theta}_{V} - \theta^*, \sigma\:Z \rangle\\  
	&=  2\:\sigma\: \langle \hat{\theta}_{V} - \overline{y} \textbf{1} -  (\theta^* - \overline{\theta^*} \textbf{1}), Z \rangle  + 2\:\sigma\: \langle \overline{y} \textbf{1} - \overline{\theta^*} \textbf{1}, Z \rangle\\
	&\leq 2\:\sigma\: \sup_{v: \TV(v) \leq 2V,\, \overline{v} = 0} \langle Z,v \rangle +2\:\sigma\ \langle \overline{y} \textbf{1} - \overline{\theta^*} \textbf{1}, Z \rangle. 
	\end{align*}
	where the last inequality follows because $\overline{\hat{\theta}_{V}} = \overline{y}$ and $\textbf{1}$ refers to the $n \times n$ matrix whose all elements equal 1. Now taking expectation on both sides of the above display and noting that 
	$$\E \langle \overline{y} \textbf{1} - \overline{\theta^*} \textbf{1}, Z \rangle = \sigma\:n^2 \E \overline{Z}^2 = \sigma$$ finishes the proof.
\end{proof} 
Let us define
\begin{equation*}
K^0(V) = K_n^0(V) \coloneqq \{\theta \in \R^{n \times n}: \TV(\theta) \leq V,\, \overline{\theta} = 0\}\,.
\end{equation*} 
In view of Lemma~\ref{lem:basicineq}, all we need is to evaluate the Gaussian width of the set $K^0(2V)$ to which end we will use Proposition~\ref{prop:chaining}. But for that we need to find ``efficient'' subspace covers of the set $K^0(V)$ corresponding to any distance $\epsilon$. Our next proposition will be crucial for this purpose. Below we denote, for any rectangular partition $P$ of $L_n$, the linear subspace of $\R^{n \times n}$ comprising only matrices that are constant on each (rectangular) block of $P$ by $S_P$. 

\begin{proposition}\label{prop:division}
	For every $\eta, V > 0$, there exist a set of rectangular partitions $\mathcal P(V,n,\eta)$ of $L_n$ (recall the definition from Section~\ref{sec:adaptive}) and a universal constant $C > 0$ such that
	\begin{itemize}
		\item For any $\theta \in K(V)$ (recall the definition from \eqref{def:KnV}), there exists a partition $P \in \mathcal P(V,n,\eta)$ satisfying $$({\rm dist}(\theta,S_p))^2 \leq V \eta \log n + \eta^2.$$ 

		\item Any partition $P \in \mathcal P(V,n,\eta)$ has number of (rectangular) blocks bounded by  $$ |P| \leq 1 + C\frac{V}{\eta} \log n.$$
		
		\item The cardinality of $\mathcal P(V,n,\eta)$ is bounded as
		$$\log |\mathcal P(V,n,\eta)| \leq  C\frac{V}{\eta} (\log n)^2.$$
	\end{itemize}
\end{proposition}

Before we prove this proposition, let us finish the proof of Theorem $2.1$ assuming it. 

\begin{proof}[Proof of Theorem~\ref{Thm:1}]
	Throughout this proof, we will use $C$ to denote an unspecified but universal positive constant whose exact value may change from one line to the next. For any $\epsilon, V > 0$, let $\eta_{\epsilon} \coloneqq \min(\frac{\epsilon^2}{2V \log n}, \frac{\epsilon}{\sqrt{2}})$ and consider the set of rectangular partitions $\mathcal{P}(V,n,\eta_{\epsilon})$ given by Proposition~\ref{prop:division}. Next define a collection $\mathcal{S}_{\epsilon}$ of linear subspaces of $\R^{n \times n}$ as follows:
	$$\mathcal{S}_{\epsilon} \coloneqq \{S_{P}: P \in \mathcal{P}(V,n,\eta_{\epsilon})\}.$$ By Proposition~\ref{prop:division} it can be seen that $\mathcal{S}_{\epsilon}$ forms an $\epsilon$ subspace cover of $K(V)$ and hence of $K^0(V)$ as well. Also, from the second and third properties of $\mathcal P(V, n , \eta)$ we get
	\begin{equation*}
	\max_{S \in \mathcal{S}_{\epsilon}} \dim(S) \leq 1 + C\max\big( \frac{V^2 (\log n)^2}{\epsilon^2}, \frac{V\log n}{\epsilon}\big) \text{ and } \log |\mathcal{S}_{\epsilon}| \leq C\max\big( \frac{V^2 (\log n)^3}{\epsilon^2}, \frac{V(\log n)^2}{\epsilon}\big).
	\end{equation*}
	We now have all the ingredients to apply Proposition~\ref{prop:chaining} except for an upper bound on the diameter of $K^0(V)$. To this end we use Proposition~\ref{prop:gagliardo} --- which we are going to state in the next subsection --- to deduce that $t \coloneqq {\rm diam}(K^0(V)) \le C \,V$. We thus obtain from Proposition~\ref{prop:chaining}, with $k_0 = \lfloor -\log_2 t\rfloor$ and $k_1 = \lceil \log_2(\frac nV)\rceil$,
	\begin{align}
	\GW(K^0(V)&) \leq C \sum_{k = k_0 + 1}^{k_1} 2^{-k} \big(\frac{V (\log n)^{3/2}}{2^{-k}} + (\log \e n)^{1/2} \,\big) + n\,2^{-k_1}  \nonumber \\ &\leq C\, V (k_1 - k_0)(\log n )^{3/2} + C\cdot2^{-k_0} (\log \e n)^{1/2} + V \nonumber \\
	&\leq C \big(\log \big(\frac{tn}{V} \vee \e\big) V (\log \e n)^{3/2} + t(\log \e n)^{1/2}\big) \le C V(\log \e n)^{5/2}.\label{eq:gwkv}
	\end{align}
	Theorem~\ref{Thm:1} now follows immediately from Lemma~\ref{lem:basicineq}
\end{proof}

\subsection{Proof of Proposition~\ref{prop:division}}
\label{sec:division_proof}
Given any partition $P$ of $L_n$ into rectangles, it is clear that the orthogonal projection of $\theta$ onto $S_P$, i.e., the unique matrix $\hat \theta_P \in S_P$ satisfying $\|\theta - \hat \theta_P\| = {\rm dist}(\theta, S_P)$, is constant on every rectangle $R$ of $P$ with the common value being the mean of $\theta_{|R}$ --- the restriction of $\theta$ to $R$. Therefore, with $\bar \theta_{|R}$ denoting the mean of $\theta_{|R}$,
\begin{equation}
\label{eq:tvl2distance}
{\rm dist}(\theta, S_P)^2 = \|\theta - \hat \theta_P\|^2 = \sum_{R \in P}\|\theta_{|R} - \bar \theta_{|R}1_R\|^2 \le |P|\,\max_{R \in P}\|\theta_{|R} - \bar \theta_{|R} 1_R\|^2
\end{equation}
where $1_{R} \in \R^R$ consists only of 1's. Our next result provides a way to bound the squared Frobenius distance between $\theta_{|R}$ and $\bar \theta_{|R}1_R$ in terms of the total variation of $\theta_{|R}$. This result, which is a discrete analogue of the Gagliardo-Nirenberg-Sobolev inequality for compactly supported smooth functions, will be crucial for deriving the first condition stipulated in Proposition~\ref{prop:division} for the particular partitioning scheme we are going to propose in this regard.
\begin{proposition}[Discrete Gagliardo-Nirenberg-Sobolev Inequality]
	\label{prop:gagliardo}
	Let $\theta \in \R^{m \times n}$ and $\overline \theta \coloneqq \sum_{i = 1}^m\sum_{j = 1}^n\theta[i, j] / mn$ be the average of the elements of $\theta$. Then we have, with $a \wedge b$ denoting the minimum of the (real) numbers $a$ and $b$,
	$$\sum_{i = 1}^m\sum_{j = 1}^n (\theta[i, j] - \overline \theta)^2 \leq (5 + \frac{4mn}{n^2\wedge m^2})\TV(\theta)^2\,.$$
	So in particular when $m = n$, we have
	$$\sum_{i = 1}^n\sum_{j = 1}^n(\theta[i, j] - \overline \theta)^2 \leq 9\TV(\theta)^2\,.$$
\end{proposition}
\begin{remark}
	Although the Gagliardo-Nirenberg-Sobolev inequality is classical for Sobolev spaces (see, e.g., Chapter~12 in \cite{leoni2017first}), we are not aware of any discrete version in the literature that applies to arbitrary matrices of finite size. Also it is not clear if the inequality in this exact form follows directly from the classical version. 
\end{remark}

Now we give a scheme for subdividing $\theta$ in multiple steps until the total variation of each of the resulting submatrices is bounded above by $\eta$. 

\smallskip

\noindent{\bf A greedy partitioning scheme: ~}For convenience of exposition we will assume that $n$ is an integer power of 2. The general $n$ can then be accommodated from the following observation. For any $t > 0$, let $B_{n, n}(t)$ denote the $t$-Euclidean (Frobenius) ball in $\R^{n \times n}$ and consider $\theta \in K_n(V) \cap B_{n, n}(t)$ (recall from our proof of Theorem~\ref{Thm:1} that we actually bound $\GW(K_n(V) \cap B_{n,n}(CV))$ for some universal constant $C$). Now let $n'$ denote the smallest integer power of 2 that is larger than or equal to $n$ and partition $\theta$ as \[
\theta =
\left[
\begin{array}{c c}
\theta_{11} & \theta_{12} \\
\theta_{21} & \theta_{22}
\end{array}
\right]
\]
where $\theta_{22} \in \R^{(n' - n) \times (n' - n)}$. Also define a $n' \times n'$ matrix $f(\theta)$ as
\[
f(\theta) =
\left[
\begin{array}{l l l}
\theta_{11} & \theta_{12} & \overleftarrow{\theta_{12}}\\
\theta_{21} & \theta_{22} & \overleftarrow{\theta_{22}}\\
\theta_{21}\uparrow & \theta_{22}\uparrow & \overleftarrow{\theta_{22}} \uparrow
\end{array}
\right]
\]
where $\overleftarrow{M}$, for any matrix $M$, denotes the matrix obtained by reversing the order of its columns whereas $M\uparrow$ is obtained by reversing the 
order of its rows. It is clear from the definition that $f(\theta) \in K_{n'}(4V) \cap B_{n',n'}(t)$ and also $$\GW(K_n(V) \cap B_{n,n}(t)) \le \E \sup_{\theta \in K_n(V): \|\theta\| \leq t} \langle Z_{n',n'}, f(\theta) \rangle \leq \GW(K_{n'}(4V) \cap B_{n',n'}(2t))$$
where $Z_{n', n'} \sim N(0_{n' \times n'}, I)$.

Let us now describe the scheme which is of the same flavor as the {\em breadth-first exploration} of a quaternary tree. The root node of the tree represents $L_n$ and the nodes at any level (or depth) $i \in [\log_2 n]$ represent (disjoint) rectangles of side-length $n2^{-i}$ with the property that the leaves of the tree truncated at level $i$ form a partition of $L_n$. Given level~$i-1$, the $i$-th level is constructed (or explored) as follows. For every leaf,  i.e., rectangle $R$ at level~$i-1$ satisfying $\TV(\theta_{|R}) > \eta$, we add four children of $R$, namely $R_{11}, R_{12}, R_{21}$ and $R_{22}$, to the tree where\[
R =
\left[
\begin{array}{c c}
R_{11} &  R_{12} \\
R_{21} &  R_{22}
\end{array}
\right]
\]
and $\nr(R_{11}) = \nc(R_{11}) = \nr(R)/2$. If the set of such leaves is empty or if $i - 1 = \log_2 n$, we stop.

\smallskip

Let us denote the final rectangular partition of $L_n$ obtained by applying the $(\TV , \eta)$ scheme to $\eta$ as $P_{\theta, \eta}$ and the set of partitions $\{P_{\theta,\eta}: \theta \in \mathcal \R^{n \times n}, \TV(\theta) \leq V\}$ as $\mathcal P(V, n, \eta)$. In our next result we verify that $\mathcal P(V, n, \eta)$ satisfy the last two properties stipulated in Proposition~\ref{prop:division}.
\begin{lemma}\label{lem:division}
	There exists a universal constant $C > 0$ such that for any $\theta \in \R^{n \times n}$ and $\eta > 0$, we have
	\begin{equation*}
	|P_{\theta, \eta}| \leq 1 + C\, \TV(\theta) \eta^{-1} \log n.
	\end{equation*}
	Furthermore, for any $V > 0$ we have 
	\begin{equation*}
	\log |\mathcal P(V,n,\eta)| \leq C  \:V \eta^{-1}(\log n)^2\,.
	\end{equation*}
\end{lemma}

\begin{proof}
	The basic idea of the proof hinges on super-additivity of the $\TV$ 
	functional over disjoint rectangles. Let $n_i$ denote the number of leaves in the tree formed by the scheme truncated at level $i$. In other words, $n_i$ is the cardinality of the partition $P_i$ formed by the rectangles corresponding to the leaves of the tree truncated at level $i$. Also let $s_i$ denote the number of leaves $R$ at level $i$ satisfying $\TV(\theta_{|R}) > \eta$. Clearly $n_0 = 
	1$ and $n_{i + 1} = n_i + 3 s_i$. Notice that, due to super-additivity of the $\TV$ functional, we must 
	have 
	\begin{equation}
	\label{eq:TV_big}
	s_i \leq \frac{\TV(\theta)}{\eta}.
	\end{equation}
	This implies in particular that
	\begin{equation}
	\label{eq:leaves_bnd}
	n_i \leq 1 + 3 i \frac{\TV(\theta)}{\eta}.
	\end{equation}
	Since $i \le \log_2 n$ by construction, it then follows
	\begin{equation*}
	|P_{\theta, \eta}| \leq 1 + 3\log_2 n \frac{\TV(\theta)}{\eta}.
	\end{equation*}
	Next we bound the number of possible partitions $P_{\theta,\eta}$ when $\TV(\theta) \leq V$. The number of distinct ways of adding leaves at level~$i + 1$ is at most $\big(1 + 3 i \frac{V}{\eta}\big)^{\frac{V}{\eta}}$
	in light of the displays \eqref{eq:TV_big} and \eqref{eq:leaves_bnd}. Therefore 
	\begin{align*}
	&\log |\mathcal P(V,n,\ep)| \leq C \frac V {\eta}(\log n)^2
	\end{align*}
	for some universal constant $C > 0$.
\end{proof}

With Lemma~\ref{lem:division} and Proposition~\ref{prop:gagliardo} in hand, we are now in a position to finish the proof of Proposition~\ref{prop:division}.

\begin{proof}[Proof of Proposition~\ref{prop:division}]
	For any given $\theta \in K(V)$, run the $(\TV,\eta)$ greedy scheme to obtain the partition $P_{\theta,\eta}.$ Within every rectangle of the partition $P_{\theta,\eta}$ the total variation of $\theta$ is at most $\eta$. Also, the number of rectangles in $P_{\theta,\eta}$ is at most $1 + C\frac{V}{\eta} \log n$. 
	Then by Proposition~\ref{prop:gagliardo} and \eqref{eq:tvl2distance} we can conclude
	\begin{equation*}
	\|\tilde{\theta} - \theta\|^2 \leq C V \eta \log n + \eta^2.
	\end{equation*}
	Also, by Lemma~\ref{lem:division}, as $\theta$ varies in $K(V)$, the number of distinct partitions $P_{\theta,\eta}$ that can be obtained is bounded by $\frac{V}{\eta} \log n$. This finishes the proof. 
\end{proof}

Finally it remains to give the proof of Proposition~\ref{prop:gagliardo}.
\begin{proof}[Proof of Proposition~\ref{prop:gagliardo}]
	For any $(i, j) \in [m] \times [n]$, 
	we have
	\begin{equation*}
	\label{eq:gagliardo1}
	(\theta[i, j] - \overline \theta )^2 \leq \sum_{j' \in [j]}|\theta[i, j'] - \theta[i, j' - 1]|\sum_{i' \in [i]}|\theta[i', j] - \theta[i' - 1, j]|\,,
	\end{equation*}
	where $\theta[i, 0] = \theta[0, j] = \overline \theta$ for all $(i, j) \in [m] \times [n]$. Summing this over all $i$ and $j$ we get
	\begin{align}
	\label{eq:gagliardo2}
	\sum_{i \in [m], j \in [n]}&(\theta[i, j] - \overline \theta )^2 \nonumber\\
	&\leq \sum_{i' \in [m], j' \in [n]}\sum_{i \geq i', j \geq j'}|\theta[i, j'] - \theta[i, j' - 1]|\times|\theta[i', j] - \theta[i' - 1, j]|\nonumber\\
	&\leq\sum_{i' \in [m], j' \in [n]}\sum_{i \in [m], j \in [n]}|\theta[i, j'] - \theta[i, j' - 1]|\times|\theta[i', j] - \theta[i' - 1, j]|\nonumber\\
	&=\sum_{i \in [m], j \in [n]}|\theta[i, j] - \theta[i, j - 1]|\sum_{i \in [m], j \in [n]}|\theta[i, j] - \theta[i-1, j]|\nonumber\\
	&= \big(\TVr(\theta) + \sum_{i \in [m]}|\theta[i, 1] - \overline \theta|\big)\big(\TVc(\theta) + \sum_{j \in [n]}|\theta[1, j] - \overline \theta|\big)\,.
	\end{align} 
Here the total variation $\TVr(\theta)$ along rows is defined as
$$\TVr(\theta) \coloneqq \sum_{i \in [m]}\sum_{j \in [n-1]} |\theta[i, j + 1] - \theta[i, j]|$$
and $\TVc(\theta) \coloneqq \TVr(\theta^T)$. Now let us try to bound $|\theta[i, 1] - \overline \theta|$.
	\begin{align*}
	\label{eq:gagliardo3}
	|\theta[i&, 1] - \overline \theta| \\
	&= \frac{1}{mn}\sum_{i' \in [m], j' \in [n]}|\theta[i, 1] - \theta[i', j']|\nonumber \\ 
	&\leq \frac{1}{mn}\sum_{i' \in [m], j' \in [n]}(|\theta[i, 1] - \theta[i, j']| + |\theta[i, j'] - \theta[i', j']|)\nonumber \\
	&\leq \frac{1}{n}\sum_{j' \in [n]}|\theta[i, 1] - \theta[i, j']| + \frac{1}{mn}\sum_{i' \in [m], j' \in [n]} |\theta[i, j'] - \theta[i', j']|\nonumber\\
	&\leq \TV(\theta[i, ]) + \frac{1}{mn}\sum_{j' \in [n]}\sum_{i' \in [m]}\TV(\theta[, j'])\nonumber\\
	&\leq \TV(\theta[i, ]) + \frac{1}{n}\sum_{j' \in [n]}\TV(\theta[, j']) = \TV(\theta[i, ]) + \frac{1}{n}\TV(\theta)\,.
	\end{align*}
	Hence 
	\begin{equation*}
	\label{eq:gigliardo4}
	\sum_{i \in [m]}|\theta[i, 1] - \overline \theta| \leq (1 + \frac{m}{n})\TV(\theta)\,.
	\end{equation*}
	Similarly
	\begin{equation*}
	\label{eq:gigliardo5}
	\sum_{j \in [n]}|\theta[1, j] - \overline \theta| \leq (1 + \frac{n}{m})\TV(\theta)\,.
	\end{equation*}
	Plugging these bounds into the last expression in \eqref{eq:gagliardo2}, we get
	$$\sum_{i \in [m], j \in [n]}(\theta[i, j] - \overline \theta)^2 \leq (2 + \frac{m}{n})(2 + \frac{n}{m})\TV(\theta)^2 \leq (5 + \frac{4mn}{n^2\wedge m^2})\TV(\theta)^2\,.\qedhere$$	
\end{proof}

\section{Proofs of Theorem~\ref{thm:adap} and Theorem~\ref{thm:lbd}}\label{Sec:adap}
We first describe the precise connection between MSE and Gaussian widths. Recall that use $B_{m,n}(t)$ to denote the usual Euclidean ball of radius $t$ in $\R^{m \times n}$. The \emph {statistical dimension} of a closed convex cone $K \subset \R^{N} = \R^{n \times n}$ is defined as 
\begin{equation*}
\delta(K) := \E \|\Pi_K(Z)\|^2 \qt{where $Z \sim N(0, I)$}
\end{equation*}
and $\Pi_K(Z) := \argmin_{u \in K} \|Z - u\|^2$ is the Euclidean projection of
$Z$ onto $K$. The terminology of statistical dimension is due to~\citet{amelunxen2014living} and we refer the reader to this paper for many properties of the statistical dimension. The statistical dimension $\delta(K)$ is closely related to the Gaussian width of $K \cap B_{n,n}(1).$ 
It has been shown in~\cite[Proposition~10.2]{amelunxen2014living} that 
\begin{equation}\label{gast}
\big[\GW(K \cap B_{n,n}(1))\big]^2 \leq \delta(K) \leq \big[\GW(K \cap B_{n,n}(1))\big]^2 + 1
\end{equation}
for every closed convex cone $K$.

The connection of the statistical dimension of tangent cones to 
the risk of $\hat{\theta}$ is the content of the following result
due to~\citet[Corollary 2.2]{bellec2018sharp}. 
\begin{theorem}[~\cite{bellec2018sharp}]\label{bellecthm} 
Suppose $Y \sim N(\theta^*, \sigma^2 I)$ for some $\theta^* \in \R^N$. Then 
\begin{equation*}\label{belma.eq}
 {\rm MSE}(\hat{\theta}_{V}, \theta^*) \leq \inf_{\theta \in K(V)} \left[\frac{1}{N}
\|\theta - \theta^*\|^2 + \frac{\sigma^2}{N}
\delta(T_{K(V)}(\theta))\right]. 
\end{equation*}
\end{theorem}

Another result that is of use to us is the following result of~\citet{oymak2013sharp} (Theorem $2.1$). It says that the upper bound provided in Theorem~\ref{bellecthm} is essentially tight. Recall from Section~\ref{sec:gauss_width_bnd_results} that $K^* = \{\theta \in \R^{n \times n}: \TV(\theta) \leq \TV(\theta^*)\}.$
\begin{theorem}[~\cite{oymak2013sharp}~]\label{hassibithm}
	\begin{equation*}
	\lim_{\sigma \rightarrow 0} \frac{1}{\sigma^2} {\rm MSE}(\hat{\theta}_{V^*},\theta^*) = \frac{1}{N} \delta(T_{{K}^*}(\theta^*)) \geq \frac{1}{N} \big[\GW(T_{{K}^*}(\theta^*) \cap B_{n,n}(1))\big]^2\,.
	\end{equation*}
\end{theorem}

\begin{remark}
	To clarify, Theorem $2.1$ in~\cite{oymak2013sharp} actually says that 
	\begin{align*}
	\lim_{\sigma \rightarrow 0} \frac{1}{\sigma^2} {\rm MSE}(\hat{\theta}_{\mb V^*},\theta^*) = \E \,\mathrm{dist}^2(Z,\mathrm{Polar}(T_{K^*}(\theta^*)))\,.
	\end{align*}
	Here $Z$, as usual, refers to a matrix of independent $N(0,1)$ entries, $\mathrm{Polar}(T_{K^*}(\theta^*))$ refers to the Polar Cone of ~$T_{{K}^*}(\theta^*)$ and $\mathrm{dist}$ refers to the Euclidean Distance between 
	two sets. Letting $K$ denote a general cone and $\Pi_K$ denote the Euclidean projection operator onto $K$, the standard Pythagorean Theorem for cones implies $$\mathrm{dist}^2(Z,\mathrm{Polar}(K)) = 
	\|\Pi_{K}(Z)\|^2\,.$$ Also, it holds that $\|\Pi_{K}(Z)\| = \sup_{\theta \in K: \|\theta\| \leq 1} \langle Z, \theta \rangle$. A proof of the above fact is available in Lemma $A.3$ in~\cite{chatterjee2019adaptive}. 
	Theorem~\ref{hassibithm} now follows from applying the above facts to Theorem~$2.1$ in~\cite{oymak2013sharp} and then using the elementary inequality $\E X^2 \geq (\E X)^2.$ 
\end{remark}

In light of the above facts and armed with Proposition~\ref{prop:gwupbd} and Proposition~\ref{prop:gwlbd} we are now ready to prove Theorem~\ref{thm:adap} and Theorem~\ref{thm:lbd} respectively.

\begin{proof}[Proof of Theorem~\ref{thm:adap}]
Theorem~\ref{bellecthm} along with~\eqref{gast} gives us
\begin{equation}\label{belma.eq}
{\rm MSE}(\hat{\theta}_{V}, \theta^*) \leq \inf_{\theta \in K(V)} \left[\frac{1}{N}
\|\theta - \theta^*\|^2 + \frac{\sigma^2}{N} + \frac{\sigma^2}{N}
\big[\GW(T_{K(V)}(\theta))\big]^2\right]. 
\end{equation}

With $V^* = \TV(\theta^*) > 0$, define $$\theta \coloneqq \overline{\theta^*} \textbf{1} + \frac{V}{V^*} \big(\theta^* - \overline{\theta^*} \textbf{1}).$$
By definition, $\TV(\theta) = V$ and $\theta$ is piecewise constant on the same partition of $L_n$ as is $\theta^*$. 
By Proposition~\ref{prop:gagliardo} we can assert that 
\begin{equation*}
\|\theta - \theta^*\|^2 = (V - V^*)^2 \frac{\|\theta^* - \overline{\theta^*} \textbf{1}\|^2}{(V^*)^2} \leq 9 (V - V^*)^2.
\end{equation*}
Therefore, in view of~\eqref{belma.eq}, we obtain
\begin{equation*}
{\rm MSE}(\hat{\theta}_{V}, \theta^*)  \leq \frac{9}{N} (V - V^*)^2 + \frac{\sigma^2}{N} + \frac{\sigma^2}{N} C (\log en)^9 k(\theta^*)^{5/4} N^{1/4}
\end{equation*}
where we have also used Proposition~\ref{prop:gwupbd} and the fact that $k(\theta) = k(\theta^*).$ 
\end{proof}

\begin{proof}[Proof of Theorem~\ref{thm:lbd}]
The proof of Theorem~\ref{thm:lbd} is immediate once we use Theorem~\ref{hassibithm} along with Proposition~\ref{prop:gwlbd}. 
\end{proof}

\section{Proofs of Proposition~\ref{prop:gwlbd} and Theorem~\ref{thm:impos}}\label{Sec:gwlb}

\subsection{Tangent Cone Characterization}
We fix a $\theta^* \in \R^{n \times n}$ and proceed to investigate the tangent cone  $T_{{K}^*}(\theta^*).$ Notice that $K^*$ is same as $K(V^*)$ defined in Section~\ref{sec:gauss_width_bnd_results} (see \eqref{def:KnV}). Let $\mathcal R^{*} = (R_1^*,R_2^*,\dots,$ $R_{k^*}^*)$ be a partition of $[n] \times [n]$ into $k^*$ rectangles where $k^* = k(\theta^*).$

Recall that the vertices in the grid graph $L_n$ correspond to the pairs $(i, j) \in [n] \times [n]$ and its edge set $E_n$ consists of:
$$\mbox{all }((i, j), (k, \ell)) \in L_n \times L_n \mbox{ such that } |i - j| + |k-\ell| = 1\,.$$
For any edge $e \in E_n$, we denote by $e^{+}$ and $e^{-}$ the vertices associated with $e$ with respect to the 
natural partial order. For any $\theta \in \R ^{L_n}$, we will use $\Delta_e\theta$ as a shorthand notation for 
the (discrete) \emph{edge gradient} $\theta(e^+) - \theta(e^-)$. Thus $\TV(\theta) = \sum_{e \in E_n} 
|\Delta_e \theta|$. For a general rectangle $R:=  ([a_1,a_2] \times [b_1,b_2]) \cap \Z^2 \subset L_n$, we define its right boundary as follows:
\begin{equation*}
\rb (R) := \{(i,j) \in R: j = b_2\}\,.
\end{equation*}
While defining the above set, we are using the matrix convention for indexing the vertices of $L_n$. Thus, the top-left vertex in the two-dimensional array $L_n$ is indexed by $(1,1)$ and the bottom-right vertex by $(n,n)$. Similarly we define the left, top and bottom boundaries of $R$ and denote them by $\lb(R)$, $\tb(R)$ and $\bb(R)$ respectively.  The boundary of $R$, denoted by $\partial R$, is defined as 
\begin{equation*}
\partial R := \rb (R) \cup \lb(R) \cup \tb(R) \cup \bb(R)\,.
\end{equation*}
\subsubsection{Starting from the definition}
The tangent cone $T_{{K}(V^*)}(\theta^*)$ is the smallest closed, convex cone containing all the elements $\theta$ in $\R^{n \times n}$ such that $\theta^* + \theta \in 
{K}(V^*)$ for $V^* = \TV(\theta^*)$. Let $A^* \coloneqq\{e \in E_n: 
|\Delta_e\theta^*| > 0\}$ and $(A^*)^{c} = E_n \setminus A.$ Observe that $|\Delta_e (\theta^* + \theta)| -  |\Delta_e 
(\theta^*)|= |\Delta_e \theta| - 0 = |\Delta_e\theta|$ for every edge $e$ in 
$(A^*)^{c}$. Thus in order for $\theta^* + \theta \in {K}(V^*)$, the increments in the absolute edge gradients of $\theta^* + \theta$ from the edges in $(A^*)^{c}$ must be compensated by an equal or greater amount of decrease in the absolute edge gradients for the 
edges in $A^*$. The precise statement is the content of

\begin{lemma}\label{lem:tanconechar}
	We have the following set equality:
	\begin{equation}\label{eq:tanconechar}
	T_{{K}(V^*)}(\theta^*) =  \big\{\theta \in \R^{n \times n}: \sum_{e \in (A^*)^c} |\Delta_e\theta| \leq - \sum_{e \in A^*} sgn( \Delta_e \theta^*)\Delta_e\theta \big\}\,.
	\end{equation}
	Here, $sgn(x) \coloneqq \I\{x > 0\} - \I\{x < 0\}$ is the usual sign function.  
\end{lemma}

\begin{proof}
	Let $T$ be the set on the right side of~\eqref{eq:tanconechar}. Let us first prove that $T_{{K}(V^*)}(\theta^*) \subset T$. An important feature of $T$ is that it is a closed convex cone. Hence it suffices to show that $\theta \in T$ whenever $\theta^* + \theta 
	\in {K}(V^*)$. To this end let $\theta$ be such that $\TV(\theta^* + \theta) \leq \TV(\theta^*)$. Since ${K}(V^*)$ is a convex set, we have 
	$$\TV(\theta^* + c \theta) = \TV \big(c(\theta^* + \theta) + (1 - c) \theta^*\big) \leq \TV(\theta^*)$$
	for any $0 \leq c \leq 1$. Now observing that
	\begin{equation*}
	\TV(\theta^* + c \theta) = \sum_{e \in A^*} |\Delta_e\theta^* + c \Delta_e\theta|  + c \sum_{e \in (A^*)^c} |\Delta_e\theta|\,,
	\end{equation*} 
	we can write
	\begin{align}
	\label{eq:tanconechar_expand}
	\TV(\theta^* + c \theta) = \sum_{e \in A^*} \big[sgn( \Delta_e \theta^*)\Delta_e\theta^* + c \, sgn( \Delta_e \theta^*) \Delta_e\theta\big]  + c \sum_{e \in (A^*)^c} |\Delta_e\theta| \leq \TV(\theta^*)\,
	\end{align} 
	whenever $c$ is small enough satisfying $sgn (\Delta_e \theta^* + c\Delta_e\theta) =  sgn (\Delta_e \theta^*)$ for all $e \in A^*$. By definition, 
	$$\TV(\theta^*) = \sum_{e \in A^*} sgn( \Delta_e \theta^*)\Delta_e\theta^*$$
	which together with \eqref{eq:tanconechar_expand} gives us $\theta \in T$. 

	It remains to show that $T \subset T_{{K}(V^*)}(\theta^*).$ It suffices to show that for any $\theta \in T$ there exists a small enough $c > 0$ such that $\TV(\theta^* + c \theta) \leq V^*$. This can be shown using the same reasoning given after~\eqref{eq:tanconechar_expand}.
\end{proof}

With the above characterization of the tangent cone, we are now ready to prove our lower bound to the risk given in Theorem~\ref{thm:lbd}.

\subsection{Proof of Proposition~\ref{prop:gwlbd}}
Recall that here we consider $\theta^*$ which is piecewise constant on two rectangles and is of the following form:
\[
\theta^{*} =
\left[ {\begin{array}{cc}
	\mb 0_{n \times n/2} & \mb 1_{n \times n/2} \\
	\end{array} } \right]
\]

\begin{proof}
	Consider $n$ to be even and a perfect square (i.e., $\sqrt{n}$ is an integer) for simplicity 
	of exposition. Also for a generic $n \times n$ matrix $\theta$ we will denote $\theta^{(1)}$ to be the submatrix formed by the first $n/2 - 1$ columns, $v^{(\theta)}$ to be the $n/2$-th 
	column and $\theta^{(2)}$ to be the submatrix formed by the last $n/2$ columns. Also, for two matrices $\theta$ and $\theta^{'}$ with the same number of rows, we will denote $[\theta : 
	\theta^{'}]$ to be the matrix obtained by concatenating the columns of $\theta$ and $\theta^{'}$.

	We can  now use Lemma~\ref{lem:tanconechar} to characterize the tangent cone $T_{{K}(V^*)}(\theta^*)$.
	\begin{equation*}
	T_{{K}(V^*)}(\theta^*) = \big\{\theta \in \R^{n \times n}: \TV([\theta^{(1)}: v^{(\theta)}]) + \TV(\theta^{(2)}) \leq \sum_{i = 1}^{n} \theta[i,n/2] - \theta[i,n/2 + 1]\big\}
	\end{equation*}
	In this proof, we will actually lower bound the Gaussian width of a convenient subset of $T_{{K}(V^*)}(\theta^*)$. To this end, for constants $c_1, c_2 \in (0, 1)$ to be specified later, let us define 
	$$S \coloneqq \big\{\theta \in T_{{K}(V^*)}(\theta^*): \theta^{(1)} = \frac{c_1}{n} \textbf{1}_{n \times (n/2 - 1)},\:\:\theta^{(2)} = \textbf{0}_{n \times n/2}, \:\:v^{(\theta)} \in \{c_1 / n, c_2 / \sqrt{n}\}^{n}\big\}\,.$$ 
	In words, for $\theta \in S$, the first $n/2 - 1$ columns are all equal to $c_1/n$, the last $n/2$ columns of $\theta$ are $0$ and the entries in the $n/2$-th column can take 
	two values; either $c_2/\sqrt{n}$ or $c_1/n$. Also, for any {\em such} matrix $\theta$, 
		\begin{equation}\label{eq:reqd}
	\TV([\theta^{(1)}:v^{(\theta)}]) \leq \sum_{i = 1}^{n} v^{(\theta)}_i \iff \theta \in S.
	\end{equation}

	Before going further, let us define the set of indices $B_j \coloneqq \{(j - 1)\sqrt{n} + 1, \,(j - 1)\sqrt{n} + 2,\dots,j\:\sqrt{n}\}$ for $j \in [\sqrt{n}].$ In words, we divide $[n]$ into $\sqrt{n}$ many equal contiguous blocks and $B_j$ refers to the $j$th 
	block. Now, for any realization of a random Gaussian matrix $Z$, let us define the 
	matrix $\nu$ so that $\nu^{(1)} \coloneqq \frac{c_1}{n} \textbf{1}_{n \times (n/2 - 1)}$ 
	and $\nu^{(2)} \coloneqq \mb 0_{n \times n/2}$. Moreover, we define $v^{(\nu)}$ as follows: 
		\begin{equation*}
	v^{(\nu)}_{i} \coloneqq \sum_{j \in [\sqrt{n}]: B_j \ni i}  \big(\,\I\{\mbox{$\sum_{k \in B_j} Z[k,n/2] > 0$}\} \frac{c_2}{\sqrt{n}} + \I\{\mbox{$\sum_{k \in B_j} Z[k,n/2] < 0$}\} \frac{c_1}{n}\,\big).
	\end{equation*}

	In words, the vector $v^{(\nu)}$ is defined so that it is constant on each of the blocks $B_j$. If $\sum_{i \in B_j} Z[i,n/2] > 0$, the value on $B_j$ is $\frac{c_2}{\sqrt{n}}$, otherwise the value is $\frac{c_1}{n}$. Now we claim that the following are true for some appropriate choice of $c_1$ and $c_2$:
	
	a) $\nu \in S$ for any $Z.$ 
	
	b) $\|\nu\| \leq 1.$

	Taking the above claims to be true we can write
	\begin{align*}
	&\GW \big(T_{{K}(V^*)}(\theta^*) \cap B_{n,n}(1)\big) \geq \GW \big(S \cap B_{n,n}(1)\big) = \E \sup_{\theta \in S \cap B_{n,n}(1)} \langle \theta, Z \rangle \geq \E \langle \nu, Z \rangle  \\=&\, \E \langle \nu^{(1)}, Z^{(1)} \rangle + \E \langle \nu^{(2)}, Z^{(2)} \rangle + \E \sum_{i = 1}^{n} v^{(\nu)}_i Z[i,n/2] = \E \sum_{i = 1}^{n} v^{(\nu)}_i Z[i,n/2]
	\end{align*}
	where we used the fact that $\nu^{(1)},\nu^{(2)}$ are constant matrices and $Z$ has mean zero entries.

	Now let us denote $\mathcal{Z}_j \coloneqq \sum_{i \in B_j} Z[i,n/2]$. Note that $(\mathcal{Z}_1,\dots,\mathcal{Z}_{\sqrt{n}})$ are independent mean zero Gaussians with standard deviation $n^{1/4}$. Therefore
	\begin{align*}
	&\E \sum_{i = 1}^{n} v^{(\nu)}_i Z[i,n/2] = \sum_{j \in [\sqrt{n}]} \E \big(\I\{\mathcal{Z}_j > 0\} \mathcal{Z}_j \frac{c_2}{\sqrt{n}} + \I\{\mathcal{Z}_j < 0\} \mathcal{Z}_j \frac{c_1}{n}\big) \\
	=&\, \sum_{j \in [\sqrt{n}]} (\frac{c_2}{\sqrt{n}} - \frac{c_1}{n}) n^{1/4}\phi = (c_2 - \frac{c_1}{\sqrt{n}})\phi n^{1/4}\,
	\end{align*}
	where for a standard Gaussian random variable $z$, we denote $\phi = \E\, z \I\{z > 0\}.$ 

	It remains to choose $c_1, c_2$ so that the two claims hold as well as $c_2 - \tfrac {c_1}{\sqrt{n}}$ is positive. To this end notice that for validating the first claim it suffices to show, in view of the definition of $\nu$, that the first inequality in \eqref{eq:reqd} holds for $\nu$, i.e., the following is true
	\begin{equation}\label{eq:show}
	\TV([\nu^{(1)}:v^{\nu}]) \leq \sum_{i = 1}^{n} v^{\nu}_i\,.
	\end{equation} 
	Now entries of $v^{\nu}$ can take two values, either $\frac{c_2}{\sqrt{n}}$ or $\frac{c_1}{n}.$ In either case it can be checked that when $c_2 \ge \tfrac{c_1}{\sqrt{n}}$ we have for each row index $i \in [n]$ 
	\begin{equation}
	\label{eq:vinu_minus_tv}
	v^{\nu}_i - \TV(\nu^{(1)}[i,1],\dots,\nu^{(1)}[i,n - 1],v^{\nu}_i) = \frac{c_1}{n}.
	\end{equation}
	Along with the fact that $$\TV([\nu^{(1)}:v^{\nu}]) = \sum_{i = 1}^{n} \TV(\nu^{(1)}[i,1],\dots,\nu^{(1)}[i,n - 1],v^{\nu}_i) + \TV(v^{\nu})\,,$$ \eqref{eq:vinu_minus_tv} implies that in order to verify~\eqref{eq:show} it suffices to show 
	$\TV(v^{\nu}) \leq c_1.$ But $v^{\nu}$ is a piecewise constant vector with at most $\sqrt{n}$ jumps of size 
	$\frac{c_2}{\sqrt{n}}$. Thus we have $\TV(v^{\nu}) \leq c_2$. Hence ensuring $c_2 \leq c_1$ is sufficient to obtain the first claim. The second claim is trivially satisfied if $c_2 \leq \sqrt{1 - c_1^2}.$ Thus, choosing $c_1 = c_2 = 1/\sqrt{2}$ we can satisfy both claims as well as $c_2 - \tfrac{c_1}{\sqrt{n}} = \tfrac{1}{\sqrt{2}}(1 - 1 / \sqrt{n}) > 0$ for all $n \ge 2$.
\end{proof}



The task now is to obtain a ``matching'' upper bound on the gaussian width, which would eventually lead to the proof of Theorem~\ref{thm:adap} in view of 
Theorem~\ref{bellecthm}. Since the proof is lengthy and somewhat technical, for the benefit of the reader we first provide an informal roadmap of the proof before starting it formally.

\subsection{Proof of Theorem~\ref{thm:impos}}

\begin{proof}
	Consider the signal matrix $\theta^* \coloneqq \mathbb I\{i + j > n\}$. From the characterization of the tangent cone given by Lemma~\ref{lem:tanconechar}, we have
	\begin{align*}
	T_{K(V^*)} (\theta^*) = \big\{\theta \in \R^{n \times n}: \sum_{e \in (A^*)^c} |\Delta_e \theta| \le - \sum_{e \in A^*}\Delta_e \theta\big\} \end{align*}
	where every edge $e$ in $A^*$ is either of the form $((i, n-i), (i, 
	n-i+1))$ or $((i, n-i), (i+1, n-i))$ for some $i \in [n-1]$.
	
	Now consider the family $T^*$ of matrices defined below:
	\begin{align*}
	T^* \coloneqq \{\theta \in \R_{+}^{n \times n} : \theta[i, j] = 0 \,\,\forall \,(i, j) \mbox{ satisfying } i + j \neq n\}.
	\end{align*}
	It is not difficult to check that $T^* \subseteq T_{K(V^*)}(\theta^*)$. It is also clear that $T^*$ is (linearly) 
	isomorphic to $\R_{+}^{n-1}$. Therefore 
	\begin{align*}
	\gw(T_{K(V^*)}(\theta^*) \cap B_{n \times n}(1)) \ge \gw(T^* \cap B_{n \times n}(1)) = \gw(\R_{+}^{n - 1} \cap B_{n-1}(1)) \geq c \sqrt{n}.
	\end{align*}
	where $B_{m}(r)$ denotes the usual Euclidean ball of radius $r$ in $\R^m$ and $c > 0$ is a universal constant. Now an application of Theorem~\ref{hassibithm} along with the above Gaussian width lower bound also furnishes a lower bound to the limiting MSE. 
\end{proof}

\begin{remark}\label{rem:impos}
	The vertex boundary of a set $A \subset L_n$ is defined to be the set of vertices 
	which share an edge with $A^c$. Consider the level sets of $\theta^*$ which are 
	the sets $A = \{(i,j) \in L_n: i + j > n\}$ and $A^{c}.$ The simple argument 
	presented in the proof of Theorem~\ref{thm:impos} relies crucially on the fact 
	that the vertex boundary of the level sets $A$ and $A^c$ are not connected in the 
	graph $L_n$. One can now consider other signals of the form $\theta^* = \mathbb 
	I\{A\}$ for a general subset $A \subset L_n$. One can check that if $A$ is of the 
	shape of a circle or a square rotated by $45$ degrees then also the vertex 
	boundary of the level sets will contain $O(n)$ connected components which are 
	singletons. Therefore, a similar argument will give a $O(n)$ lower bound to 
	$\gw(T_{K(V^*)}(\theta^*).$ We believe that it might be possible to formalize the 
	intuition that whenever $A$ is sufficiently far from being a rectangle, 
	$\gw(T_{K(V^*)}(\theta^*)$ is lower bounded by $O(n).$ 
	
\end{remark}

\section{Proof of Proposition~\ref{prop:gwupbd}}\label{Sec:gwup}

\subsection{Informal Roadmap}
The proof of Proposition~\ref{prop:gwupbd} can be divided into three major steps which we now describe. Recall that $\theta^*$ is the true signal which is piecewise constant on axis aligned rectangles $R_1^*,\dots,R_{k(\theta^*)}^*$ which partition $L_n.$

\textbf{Step 1}: We have to bound $\GW(T_{{K}(V^*)} (\theta^*) \cap B_{n \times n}(1)).$ To do this, we show that if a matrix $\theta$ is in $T_{{K}(V^*)} (\theta^*) \cap B_{n \times 
n}(1)$ then each rectangular submatrix  $\theta_{R_i^*}$ satisfies the property that 
$\TV(\theta_{R_i^*})$ is at most the $\ell_1$ norm of its four boundaries plus a small wiggle 
room $\delta > 0.$ Such matrices are denoted later in~\eqref{def:fbr} as $\M^{4}.$ This 
fact then reduces our problem to bounding the Gaussian width for the class of matrices 
$\M^{4}.$ Corollary~\ref{cor:twoways}, Lemma~\ref{lem:imp} and Lemma~\ref{lem:imp2} are part of this step. 

\smallskip

\textbf{Step 2}: Before starting the Gaussian width calculations, we found it convenient to further simplify the class of matrices $\M^{4}.$ In this step, we show that if a matrix $\theta$ lies in $\M^{4}$ then we can subdivide it further into several submatrices which now satisfy a simpler property. The property is that the total variation of these submatrices are at most the $\ell_1$ norm of only one or none of its boundaries (instead of four) plus an appropriately small ``wiggle room'' $\delta > 0.$ These sets of matrices are denoted by $\M^{1}$ and $\M^{0}$ respectively and are defined just before Lemma~\ref{lem:gw4to1}. Along with Lemma~\ref{lem:gw4to1},
Lemmata~\ref{lem:finalparti}--\ref{lem:2to1} are also parts of this step. 

\smallskip

\textbf{Step 3}: This is the step where we actually compute the metric entropies of the classes of matrices $\M^{1}$ and $\M^{0}$ and finally 
bring all the pieces together. Proposition~\ref{prop:gwonebdry} and Lemmata~\ref{lem:gwnobdry}--\ref{lem:division2} are all parts of this step. 


\subsection{Towards simplifying the tangent cone} \label{subsec:tanconesimplify1}

We first want to split $\theta$ into submatrices each of which satisfies a separate constraint. This and the next subsection are devoted to this goal.
Let us revisit Lemma~\ref{lem:tanconechar}. Since $\theta^*$ is constant on each rectangle $R_i^* \in \mathcal R^*$ it follows that
$$A^* = \{e \in E_n: e^+ \in R_i^* \mbox{ and }e^- \in R_j^* \mbox{ for some }i \neq j\in [k^*]\}\,.$$
As a consequence we get the following corollary:
\begin{corollary}
	\label{cor:twoways}
	Fix $\theta^* \in \R^{n \times n}$. We have
	\begin{align*}
	T_{ K(V^*)}(\theta^*) \subset \Big\{&\theta \in \R^{n \times n}: \sum_{i \in [k^*]}\TV(\theta_{R_i^*}) \leq \sum_{i \in [k^*]} \sum_{u \in \partial R_i^*} |\theta(u)|\Big\}\,.
	\end{align*}
\end{corollary}
The first step towards obtaining a decomposition where each submatrix satisfies some constraint is to separate the constraints for $R_i^*$'s. More precisely we would like 
\begin{equation}\label{eq:tancone_separation}
\TV(\theta_{R_i^*}) \leq \sum_{u \in \partial R_i^*} |\theta(u)|
\end{equation}
for each $i \in [k^*]$. As we will see below that this is ``almost'' the truth when we consider matrices in the tangent cone which are of unit norm.

Let us make precise the notion of an ``almost'' version of 
\eqref{eq:tancone_separation}. To this end we introduce for any $\delta, \, t > 0$:
\begin{align}\label{def:fbr}
\M^{4}(m',n',\delta,t) := \{\theta \in \R^{m' \times n'}: \TV(\theta) \leq &\norm{\theta_\lt}_1 + \norm{\theta_\rt}_1 + \norm{\theta_\tp}_1 \nonumber \\ &+ \norm{\theta_\bt}_1  + \delta, \,\norm{\theta} \leq t\}\,,
\end{align}
where $\theta_\lt \coloneqq \theta[\:,\:1]$, $\theta_\rt \coloneqq \theta[\:,\:n']$, $\theta_\tp \coloneqq \theta[1,\:]$ and 
$\theta_\bt \coloneqq \theta[m',\:]$. In plain words, $\M^{4}(m',n',\delta,t)$ consists of matrices of norm at most $t$ whose total variation is bounded by the total $\ell_1$ norm of its four boundaries plus an extra 
\textit{wiggle room} $\delta > 0$. In our next result we show that for any $\theta$ in $T_{ K(V^*)}(\theta^*)$ intersected with the unit Euclidean ball $B_{n \times n}(1)$, the restriction $\theta_{|R_i^*}$ of $\theta$ to $R_i^*$ lies in $\M^{4}(m_i, n_i,\delta_i,t_i)$ for each $i \in [k]$ with $m_i := \nr(R_i^*)$, $n_i := \nc(R_i^*)$ and $t_i$'s and $\delta_i$'s satisfying some upper bounds on their $\ell_2$ and $\ell_1$-norms respectively.
\begin{lemma}
	\label{lem:imp}
	We have the set inclusion
	\begin{align*}
	T_{ K(V^*)}(\theta^*) \cap B_{n \times n}(1) \subset \bigcup_{\bm \delta \in S_{k^*,\Delta(\theta^*)}} \bigcup_{\bm {t^2} \in S_{k^*,1}} \{\theta \in \R^{n \times n}: \theta_{|R_i^*} \in \M^{4}(m_i,n_i,\delta_i,t_i), \:\:\forall i \in [k^*]\}
	\end{align*}
	where $S_{k^*,r} := \{a \in \R^{k^*}_{+}: \sum_{i \in [k^*]} a_i \leq r\}$ is the non 
	negative simplex with radius $r > 0$, $\bm {t^2}$ is the vector $(t_1^2, \dots, 
	t_{k^*}^2) \in \R_+^{k^*}$ and
	\begin{equation*}
	\Delta(\theta^*) = \sqrt{2 \sum_{i \in [k^*]} \big(\frac{m_i}{n_i} + \frac{n_i}{m_i}\big)}.
	\end{equation*}
\end{lemma}

\begin{remark}
	By virtue of Lemma~\ref{lem:imp}, we achieve our objective of obtaining a characterization of $T_{ K(V^*)}(\theta^*)$ where we have separate constraints for each $R_i^* \in \mathcal R^{*}.$ The constraints are now coupled together by the wiggle room vector $\bm \delta \in  S_{k^*,\Delta(\theta^*)}$ and the (squared) $\ell_2$-norm vector $\bm {t^2}$.
\end{remark}
\begin{proof}
	We will start with a claim. 
	
\begin{claim}\label{claim:lightrow}
Let $\theta \in T_{ K(V^*)}(\theta^*)\, \cap \,B_{n \times n}(1).$ 
		Then for each $i \in [k^*]$ and any fixed choice of rows and columns $r_i,\,c_i$ in 
		$R_i^*$, we have $\theta_{|R_i^*} \in \M^{4}(m_i,n_i,\delta_i,t_i)$ where $\sum_{i \in 
			[k^*]}t_i^2 \leq 1$ and $(\delta_1, \dots, \delta_{k^*}) =: \bm \delta \in \R^{k^*}_{+}$ satisfies  
		$$\norm{\bm \delta}_1 \:\leq\: 2\: \sum_{i \in [k^*]} 
		\Big(\norm{\theta_{|c_i}}_{1} + \norm{\theta_{|r_i}}_{1}\Big)\,,$$
		where $\theta_{|c_i}$ (or $\theta_{|r_i}$) is the vector obtained by restricting $\theta$ to the row $c_i$ (respectively the column $r_i$).
	\end{claim}
	
	Let us first deduce the lemma assuming our claim. Consider a $\theta 
	\in T_{ K(V^*)}(\theta^*)$ such that $\|\theta\|_2 \leq 1$ and for each $i \in [k^*]$, 
	let $r_i$ and $c_i$ denote the rows and columns such that the $\ell_1$ norms of 
	$\theta_{|r_i}$ and $\theta_{|c_i}$ are minimum. Then by 
	Claim~\ref{claim:lightrow}, each $\theta_{|R_i^*} \in \M^{4}(m_i,n_i,\delta_i,t_i)$ with $\bm {t^2} \in S_{k^*, 1}$ and $\bm \delta$ satisfying
	\begin{equation}
	\label{eq:deltaell1bnd}
	\norm{\bm \delta}_1 \:\leq\: 2\:\sum_{i \in [k^*]}	\min_{\substack{c : \, c \:\text{is a column of}\: R_i^*,\\
			r : \, r \:\text{is a row of}\: R_i^*}}  \Big(\norm{\theta_{|c}}_1 + \norm{\theta_{|r}}_1\Big).
	\end{equation}
Now for each $i \in [k^*]$, we have
	\begin{align*}
	\min_{c :\,c \:\text{is a column of}\: R_i^*} \norm{\theta_{|c}}_1 \leq \sqrt{m_i} \min_{c : c \:\text{is a column of}\: R_i^*} \norm{\theta_{|c}}_2 \leq \sqrt{\frac{m_i}{n_i}} \|\theta_{|R_i^*}\|_2\,.
	\end{align*} 
The first inequality is an application of the Cauchy-Schwarz inequality and the second inequality follows from the ``minimum is less than the average'' principle. Similarly, one can obtain the row version of these inequalities and together they give us
	\begin{align*}
	\min_{\substack{c : \, c \:\text{is a column of}\: R_i^*,\\
			r : \, r \:\text{is a row of}\: R_i^*}}  \Big(\norm{\theta_{|c}}_1 + \norm{\theta_{|r}}_1\Big) \le \Big(\sqrt{\frac{m_i}{n_i}} + \sqrt{\frac{n_i}{m_i}}\Big)\|\theta_{|R_i^*}\|_2. 
	\end{align*} 
	Summing the above inequality over all $i \in [k^*]$ and subsequently using the Cauchy-Schwarz inequality as well as the fact that 
	$\|\theta\|_2 \leq 1$, we get in view of \eqref{eq:deftancone}
	\begin{align*}
\norm{\bm \delta}_1\leq \sqrt{2 \sum_{i \in [k^*]} \big(\frac{m_i}{n_i} + \frac{n_i}{m_i}\big)} = \Delta(\theta^*)\,,
	\end{align*}
thus yielding the lemma.
	
	\smallskip
	
{\em Proof of Claim~\ref{claim:lightrow}.}	The constraint on $t_i$'s is clear and therefore all we need to show is the constraint on $\delta_i$'s. Recall from the definition in \eqref{def:fbr} that $\delta_i$ can be chosen, for any $i \in [k^*]$, as
\begin{equation}
\label{eq:delta_bnd}
\delta_i \coloneqq \big(\TV(\theta_{|R_i^*}) - \sum_{u \in \partial R_i^*}|\theta(u)|\big)_+
\end{equation}
where $a_+ \coloneqq \max\{a, 0\}$ for any $a \in \R$. Now fix $i \in [k^*]$ and consider a generic row $r_i$ of $R_i^*$. Treating $r_i$ as a horizontal path in the graph $L_n$, let us denote 
	its two end-vertices by $u$ and $w$ with $u \in \lb(R_i^*)$ and $w \in \rb(R_i^*)$. Now 
	denoting the vertex in $r_i \,\cap\, c_i$ by $v$, we see that $v$ occurs between the 
	vertices $u$ and $w$ in the row $r_i$. 
	Therefore we can write
	\begin{align*}
	\TV(\theta_{r}) \geq |\theta(u)| + |\theta(w)| - 2 |\theta(v)|\,.
	\end{align*} 
	Summing the above inequality for every row in the rectangle $R_i^*$ gives us
	\begin{align*}
\TVr(\theta_{|R_i^*}) \geq \sum_{u \in \lb(R_i^*)}|\theta(u)| + \sum_{w \in \rb(R_i^*)}|\theta(w)| - 2 \norm{\theta_{c_i}}_1.
	\end{align*}
	By a similar argument applied to the columns of $R$ we obtain
	\begin{align*}
	\TVc(\theta_{|R_i^*}) \geq \sum_{u \in \tb(R_i^*)}|\theta(u)| + \sum_{w \in \bb(R_i^*)} |\theta(w)| - 2 \norm{\theta_{r_i}}_1.
	\end{align*}
	Summing the previous two displays we get the following inequality:
	\begin{align*}
	\TV(\theta_{|R_i^*}) \geq \sum_{u \in \partial R_i^*}|\theta(u)| - 2 \norm{\theta_{r_i}}_1 -  2 \norm{\theta_{c_i}}_1.
	\end{align*}
	Now if $\theta \in T_{K(V^*)}(\theta^*)$, then as a consequence of Corollary~\ref{cor:twoways} we also have
	\begin{equation*}
	\sum_{i \in [k^*]} \TV(\theta_{|R_i^*}) \leq \sum_{i \in [k^*]}\sum_{u \in \partial R_i^*}|\theta(u)|\,.
	\end{equation*}
	Hence an application of Lemma~\ref{lem:simple} (stated and proved in the appendix) to $f_i = \TV(\theta_{|R_i^*})$, $g_i = \sum_{u \in \partial R_i^*}|\theta(u)|$, $h_i = 2 (\norm{\theta_{r_i}}_1 + \norm{\theta_{c_i}}_1)$ and $w_i = \delta = 0$, would give us the claim in view of \eqref{eq:delta_bnd}.
\end{proof}
With the help of Lemma~\ref{lem:imp} we can now deduce the following lemma.

\begin{lemma}\label{lem:imp2}
	With the notation described in this section, we have the following upper bound:
	\begin{align*}
	\GW&(T_{K(V^*)}(\theta^{*}) \cap B_{n,n}(1))\\ &\leq\max_{\Delta(\theta^*) \bm \delta: \bm \delta \in S_{k^*,2} \cap H_{k^*}}\: \max_{\bm {t^2} \in S_{k^*,2} \cap H_{k^*}} \:\sum_{i \in [k^*]} \GW(\M^{4}(m_i, n_i,\Delta(\theta^*) \delta_i ,t_i)) + \:\: C \sqrt{k^*}.
	\end{align*}
	where $H_{k^*} \coloneqq \{\frac{1}{k^*},\frac{2}{k^*},\dots,1\}^{k^*}$ and $C > 0$ is a universal constant.
\end{lemma}
\begin{proof}
	Using Lemma~\ref{lem:imp} we can write
	\begin{align}\label{eq:step1}
	&\E \sup_{\theta \in T_{K(V^*)}(\theta^*): \|\theta\| \leq 1} \langle Z, \theta \rangle \leq 
	\:\:\E \sup_{\bm \delta \in S_{k^*,\Delta(\theta^*)}} \sup_{\bm {t^2} \in S_{k^*,1}} \sum_{i 
		\in [k^*]} \sup_{\theta \in T_{K(V^*)}(\theta^*): \|\theta\| \leq 1} \langle 
	Z_{|R_i^*}, \theta_{|R_i^*} \rangle \nonumber \\\leq&\, \E \sup_{\bm \delta \in 
		S_{k^*,\Delta(\theta^*)}} \sup_{\bm {t^2} \in S_{k^*,1}} \sum_{i \in [k^*]} 
	\sup_{\theta_i \in \M^{4}(m_i, n_i,\delta_i,t_i)} \langle Z, \theta_i \rangle 
	\end{align}
	where, by a slight abuse of notation, $Z$ always refers to a matrix of independent standard normals with appropriate number of rows and columns.
	
	At this point, we would like to convert the supremum over $\bm\delta, \bm {t^2}$ (or, equivalently $\bm t$) in the non negative simplex to a maximum over a finite net of $\bm \delta, \bm t$ We can accomplish 
	this by the following trick. Fix any $\bm \delta \in 
	S_{k^*,\Delta(\theta^*)}.$ Then we can define a vector $\bm q = \bm q(\bm \delta) \in \R^{k^*}$ such that $$q_i \coloneqq \frac{1}{k^*} \big\lceil 
	\frac{k^* \delta_i}{\Delta(\theta^*)} \big\rceil.$$ It is clear that $\mb q \in H_{k^*} \cap S_{k^*,2}$. It is also clear that $\bm \delta \leq \bm q\:\Delta(\theta^*)$ element-wise. Due to 
	similar reason, for any $\bm {t^2} \in S_{k^*,1}$ there exists $\bm w = \bm w(\bm t) \in H_{k^*} \cap 
	S_{k^*,2}$ such that $\bm {t^2} \leq \bm w$ element-wise. Since the collections $\M^{4}(m_i, n_i,\delta_i,t_i)$ are increasing in $(\delta_i, t_i)$ (with respect to set inclusion), it follows from the previous discussion that
	\begin{align}\label{eq:step2}
	&\E \sup_{\bm \delta \in S_{k^*,\Delta(\theta^*)}} \sup_{\bm {t^2} \in S_{k^*,1}} \sum_{i \in [k^*]} \sup_{\theta_i \in \M^{4}(m_i, n_i,\delta_i,t_i)} \langle Z, \theta_i \rangle  \nonumber \\\leq&\, \E\: \max_{\Delta(\theta^*) \bm \delta: \bm \delta \in S_{k^*,2} \cap 
		H_{k^*}}\: \max_{\bm {t^2} \in S_{k^*,2} \cap H_{k^*}} \:\sum_{i \in [k^*]} \sup_{\theta_i \in \M^{4}(m_i, n_i,\Delta(\theta^*) \delta_i ,t_i)} \langle Z, \theta_i \rangle 
	\end{align}

	Since $Z$ is a matrix with i.i.d $N(0, 1)$ entries, the first two maximums in the right hand side of the above display can actually be taken outside the expectation upto 
	an additive term. This follows from the well known concentration properties of suprema of gaussian random 
	variables. In particular, we now apply Lemma~\ref{lem:Guntu} (stated in the appendix), true for suprema of gaussians, to obtain for a universal constant $C,$
	\begin{align}\label{eq:step3}
	&\E\: \max_{\Delta(\theta^*) \bm \delta: \bm \delta \in S_{k^*,2} \cap H_{k^*}}\: \max_{\bm {t^2} 
		\in S_{k^*,2} \cap H_{k^*}} \:\sum_{i \in [k^*]} \sup_{\theta_i \in \M^{4}(m_i, 
		n_i,\Delta(\theta^*) \delta_i ,t_i)} \langle Z, \theta_i \rangle \nonumber \\\leq&\, 
	\max_{\Delta(\theta^*) \bm\delta: \bm\delta \in S_{k^*,2} \cap H_{k^*}}\: \max_{\bm {t^2} \in 
		S_{k^*,2} \cap H_{k^*}} \:\sum_{i \in [k^*]} \E \sup_{\theta_i \in \M^{4}(m_i, 
		n_i,\Delta(\theta^*) \delta_i ,t_i)} \langle Z, \theta_i \rangle \:\: + \:\: C 
	\sqrt{\log |H_{k^*} \cap S_{k^*, 2}|}.
	\end{align}
	
	To bound the log cardinality $\log |H_{k^*} \cap S_{k^*,2}|$, 
	note that for any positive integer $k^*$, the cardinality $|H_{k^*} \cap S_{k^*,2}|$ is the same as the number of $k^*$ tuples of positive integers summing up to at most $2k^*$. By standard combinatorics, we have
	\begin{equation*}
	|H_{k^*} \cap S_{k^*,2}| = \sum_{s = k^*}^{2k^*} {s - 1 \choose k^* - 1}\,.
	\end{equation*}
	Since 
	$$\frac{\binom{s}{k^*-1}}{\binom{s-1}{k^*-1}} = \frac{s}{s - k^* + 1} \geq \frac{2k^* - 1}{k^*}\,$$
	for all $s \in \{k^*, \ldots, 2k^*-1\}$, it follows that
	\begin{equation*}
	|H_{k^*} \cap S_{k^*,2}| \leq 3\binom{2k^* - 1}{k^* - 1} \leq C\e^{Ck^*}\,
	\end{equation*}
	for some positive absolute constant $C$.
	
	Using~\eqref{eq:step1},~\eqref{eq:step2},~\eqref{eq:step3} and the above cardinality bound, we can finally finish the proof by writing
	\begin{align*}
	& \GW(T_{K(V^*)}(\theta^*) \cap B_{n,n}(1)) = \E \sup_{\theta \in T_{K(V^*)}(\theta^*): \|\theta\| \leq 1} \langle Z, \theta \rangle \nonumber \\\leq&\, \max_{\Delta(\theta^*) \bm\delta: \bm\delta \in S_{k^*,2} \cap H_{k^*}}\: \max_{\bm {t^2} \in S_{k^*,2} \cap H_{k^*}} \:\sum_{i \in [k^*]} \E \sup_{\theta_i \in \M^{4}(m_i, n_i,\Delta(\theta^*) \delta_i ,t_i)} \langle Z, \theta_i \rangle \:\: + \:\: C \sqrt{k^*}\,. &&\qedhere
	\end{align*}
\end{proof}
Operationally, the above lemma reduces the task of upper bounding the Gaussian width of $T_{K(V^*)(\theta^{*})} \cap B_{n,n}(1)$ to upper bounding the Gaussian width of $\M^{4}$ with appropriate parameters. However, it would be convenient for us to bound the Gaussian width when the number of boundaries involved in the constraint is at most one instead of four. The results in the next subsection makes this possible.

\subsection{Further simplification: from four boundaries to one}
We now proceed to the second step, i.e., reducing the number of boundaries involved in the constraints from four to one (or zero).
Thus, we will keep on subdividing each $\theta_{|R_i}$ until we obtain submatrices satisfying constraints similar to \eqref{def:fbr}, albeit with the $\ell_1$-norm of at most one boundary vector appearing on the right hand side of the bound on total variation. This is the content of this subsection.
\newline
\newline
Taking the cue from the the previous subsection, let us define
\begin{equation*}
\M^{\tp}(m',n',\delta,t) \coloneqq \{\theta \in \R^{m' \times n'}: \TV(\theta) \leq \norm{\theta_\tp}_1 + \delta,\, \|\theta\| \leq t\}.
\end{equation*} 
We can define $\M^{\bt}(m',n',\delta,t)$, $\M^{\lt}(m',n',\delta,t)$ and $\M^{\rt}(m',n',\delta,t)$ 
in a similar fashion. Notice that the constraint satisfied by the total variation of the members of $\M^{\rt}(m',n',\delta,t)$ is ``almost'' identical to 
\eqref{eq:reqd}. \textit{By abuse of notation we will refer to any of the four families of matrices described above by a generic notation which is 
	$\M^{1}(m',n',$ $\delta,t)$.} The reason behind this is that our ultimate concerns would be the Gaussian widths of these families which, for $m'$ and $n'$ close enough to each other, are expected to be of similar order by symmetry. 
Using a single notation for them would thus minimize the notational clutter. In a similar vein we define
\begin{equation*}
\M^{0}(m',n',\delta,t) \coloneqq \{\theta \in \R^{m' \times n'}: \TV(\theta) \leq \delta,\, \|\theta\| \leq t\}\,.
\end{equation*} 

Having defined the relevant families of matrices, we can now state our main result for this subsection. 
\begin{lemma}\label{lem:gw4to1}
	Fix positive integers $m,n$ and positive numbers $\delta,t.$ Define for each integer $j \geq 1$,
	\begin{equation}
	\label{eq:finalpartideltaij}
	\delta^{(j)} \coloneqq \delta + 16\:(j+1)\: \:t\:\big(\sqrt{\frac{{m}}{{n}}} + \sqrt{\frac{{n}}{{m}}}\big)\,.
	\end{equation}
	Then we have the following bound for a universal constant $C > 0$,
	\begin{align*}
	\GW&(\M^{4}(m,n,\delta,t))\\
	&\leq C\, \Big(\sum_{j = 1}^{K}\big(\GW(\M^{1}(\frac{m}{2^j},\frac{n}{2^j},\delta^{(j)},t)) + \GW(\M^{0}(\frac{m}{2^j},\frac{n}{2^j},\delta^{(j)},t)\big)\Big).
	\end{align*}
	Here, to simplify notations, we use $m/2^j$, for $m,j \in \N$, to denote any (but fixed in any given context) integer $m'$ 
	between $m2^{-(j + 1)}$ and $m2^{-j}$. The similar definition for $n$ instead of $m$ is denoted by $n/2^j.$ $K$ equals the number of binary divisions of $[m] \times [n]$ on both axes that are possible and equals $\min\{\log_2 m, \log_2 n\}$ up to a universal constant. 
\end{lemma}

The above lemma bounds the Gaussian width of $\M^{4}$ in terms of Gaussian widths of simpler classes of matrices $\M^{1}$ and $\M^{0}$. We devote the next subsection to its proof. 

\subsection{Proof of Lemma~\ref{lem:gw4to1}}\label{subsec:lem4to1}
We need some intermediate lemmas. We start with the following lemma. The notation convention is same as in Lemma~\ref{lem:gw4to1}.
\begin{lemma}\label{lem:finalparti}
There exists a rectangular partition $\mathcal R$ of $[m] \times [n]$ with the following property. For any $\theta \in \M^{4}(m, n, \delta, t)$,
there exists  non negative real numbers $t_R$ for every rectangle $R \in \mathcal R$ such that: \begin{itemize}
		\item $\mathcal R = \bigcup_{j \in [K], \,k \in [2]}\mathcal R_{j, k}$ where $\mathcal R_{j, k}$'s are disjoint sets of rectangles and all the rectangles in $\mathcal R_{j, k}$ are of size $m_i/2^j \times n_i/2^j$.

		\item $|\mathcal R_{j, 1}| \leq 8$ and for any $R \in \mathcal R_{j, 1}$ we have $\theta_{|R}\in\M^{1}(m/2^{j},n/2^{j},\delta^{(j)},t_{R})$.
		
		\item $|\mathcal R_{j, 2}| \leq 4$ and for any $R \in \mathcal R_{j, 2}$ we have $\theta_{|R}\in\M^{0}(m/2^{j},n/2^{j},\delta^{(j)},t_{R})$.
		
		\item $\sum_{R \in \mathcal R} t_{R}^2 = t^2$
	\end{itemize}
\end{lemma}

\begin{proof}[Proof of Lemma~\ref{lem:gw4to1}]
	The proof of Lemma~\ref{lem:gw4to1} follows directly from Lemma~\ref{lem:finalparti} and the sub-additivity of the Gaussian width functional. 
\end{proof}


The task now is to prove Lemma~\ref{lem:finalparti}. The proof of Lemma~\ref{lem:finalparti} is divided into two steps where we state and prove two intermediate lemmas. In the first step we reduce the number of ``active'' boundaries, i.e., the number of boundary vectors involved in the bound on total variation, from four to two and in the second step we reduce them from two to one or zero. The main idea of the proofs is essentially same as that of Lemma~\ref{lem:imp}. 

\begin{remark}
	While lemma~\ref{lem:finalparti} is true for any integers $m,n$, the reader can safely read on as if $m,n$ are 
	powers of $2$. The essential aspects of the proof of Lemma~\ref{lem:finalparti} all go through in this case. 
	Writing the general case would make the notations messy. 
	For the sake of clean exposition, we thus write the entire 
	proof when $m$ and $n$ are powers of $2$. At the end, we mention the modifications needed when $m,n$ are not powers of $2$.
\end{remark}

\medskip
\noindent{\bf Four to two boundaries. }

In order to state this result let us define for any $\delta > 0,$
\begin{equation*}
\M^\tr(m,n,\delta,t) \coloneqq \{\theta \in \R^{m \times n}: \TV(\theta) \leq 
\norm{\theta_\rt}_1 + \norm{\theta_\tp}_1 + \delta,\, \|\theta\| \leq t\}.
\end{equation*}
Similarly we can define the families $\M^{\tl}(\cdots),\,\M^{\bl}(\cdots)$ and $\M^\br(\cdots)$. Likewise $\M^{1}(m,n,\delta,t)$, we will refer generically to any of these 
four families of matrices by $\M^{2}(m,n,$ $\delta,t)$. Below we call a partitioning of a 
matrix $\theta \in \R^{m \times n}$ as an {\em equal dyadic partitioning} if each submatrix 
lies in $\R^{m/2 \times n/2}$ and is formed by adjacent rows and columns of $\theta$ as 
$\theta^{\tl},\,\theta^{\tr},$ $\theta^{\bl}$ and $\theta^{\br}$ in the obvious order.
\begin{lemma}\label{lem:4to2}
	Take any $\theta \in \M^{4}(m,n,\delta,t)$. Let us denote the four submatrices obtained by an equal dyadic partitioning of $\theta$. Then the submatrix $\theta^{ab}$, where $a \in \{\tp, \bt\}$ and $b \in \{\lt, \rt\}$, itself satisfies
	\begin{align*}
	\TV(\theta^{ab}) \leq \|\theta^{ab}_a\|_1 + \|\theta^{ab}_b\|_1 + \delta + 16 t\big(\sqrt{\frac{{m}}{{n}}} + \sqrt{\frac{{n}}{{m}}}\big)\,.
	\end{align*}
	In words, if a matrix $\theta \in \M^{4}(m,n,\delta,t)$ is dyadically partitioned into four equal sized submatrices, each of these four submatrices lies in $\M^{2}(m/2,n/2,\delta',t)$ where $\delta' \coloneqq \delta + 16 t 
	(\sqrt{\frac{{m}}{{n}}} + \sqrt{\frac{{n}}{{m}}})$; {furthermore the boundaries that are active for these submatrices are the ones that they share with $\theta$.}
\end{lemma}

\begin{proof}
	Since $\norm{\theta}_2 \leq t$, there exists $1 \leq i \leq m/2 < i' \leq m$ and $1 \leq j \leq n/2 < j' \leq n$ such that 
	\begin{align*}
	&\max\{\norm{\theta[i,\,]},\|\theta[i',\,]\|\} \leq \frac{2t}{\sqrt{m}}\\&
	\max\{\|\theta[,\,j]\|,\|\theta[,j']\|\} \leq  \frac{2t}{\sqrt{n}}\,.
	\end{align*}
	The previous display and the Cauchy-Schwarz inequality together imply
	\begin{align}\label{eq:l1bnd}
	&\max\{\|\theta[i,\,]\|_1,\|\theta[i',\,]\|_1\} \leq 2t\sqrt{\frac{n}{m}}\nonumber\\&
	\max\{\|\theta[\,,\,j]\|_1,\|\theta[\,,\,j']\|_1\} \leq 2t\sqrt{\frac{m}{n}}\,.
	\end{align}
	Now consider the submatrix $\theta^{\tl}$ for which we have
	$$\TVr(\theta^{\tl}) \geq \|\theta^{\tl}[\,,\,1]\|_1 -  \|\theta^{\tl}[\,,\,j]\|_1 = \|\theta^{\tl}_{\lt}\|_1 -  \|\theta[1:m/2\,,\,j]\|_1\,,$$
	where in the last step we used the fact that $\theta^{\tl}[\,,\,j] = \theta[1:m/2,\, j]$. A similar argument gives us
	\begin{equation*} 
	\TVc(\theta^{\tl}) \geq \|\theta^{\tl}_\tp\|_1 - \|\theta[i\,,1:n/2\,]\|_1\,.
	\end{equation*}
	Analogous lower bounds for $\TVr$ and $\TVc$ of the other three submatrices can be derived involving the $\ell_1$ norms of appropriate boundaries and (partial) rows or 
	columns of $\theta$. Adding all these together and using \eqref{eq:l1bnd}, we obtain
	\begin{equation*} 
	\sum_{\substack{a \in \{\tp, \bt\},\\ b \in \{\lt, \rt\}}} (\TVr(\theta^{ab}) + \TVc(\theta^{ab})) \geq \sum_{\substack{a \in \{\tp, \bt\},\\ b \in \{\lt, \rt\}}} (\|\theta^{ab}_a\|_1 + \|\theta^{ab}_b\|_1) - 16 t 
	\big(\sqrt{\frac{{m}}{{n}}} + \sqrt{\frac{{n}}{{m}}}\big).
	\end{equation*}
	On the other hand, since $\theta \in \M^{4}(m,n,\delta,t)$ we have
	\begin{align*}
	\sum_{\substack{a \in \{\tp, \bt\},\\ b \in \{\lt, \rt\}}} (\TVr(\theta^{ab}) + \TVc(\theta^{ab})) &\leq \TV(\theta) \leq \sum_{c \in \{\tp, \bt, \lt, \rt\}} \|\theta_c\|_1 + \delta\\
	& = \sum_{\substack{a \in \{\tp, \bt\},\\ b \in \{\lt, \rt\}}} (\|\theta^{ab}_a\|_1 + \|\theta^{ab}_b\|_1) + \delta\,.
	\end{align*}
	An application of Lemma~\ref{lem:simple} now finishes the proof of the lemma from the previous two displays. 
\end{proof}

\medskip
\noindent{\bf Two to one or zero boundary. }

Let us start by stating the following lemma which one can think of as a version of Lemma~\ref{lem:4to2} applied to an element of 
$\M^{2}(m, n, \delta, t)$. The proof is very similar and we leave it to the reader to verify.
\begin{lemma}\label{lem:2to1}
	Let $\theta \in \M^{ab}(m,n,\delta,t)$ for some $a \in \{\tp, \bt\}$ and $b \in \{\lt, 
	\rt\}.$ We can partition $\theta$ into equal sized four submatrices $\theta^{\tl}, \theta^{\tr},$ $\theta^{\bl}$ and $\theta^{\br}$ in the obvious manner such that the submatrix $\theta^{cd}$, where $c \in \{\tp,$ $ \bt\}$ and $d \in \{\lt, \rt\}$, satisfies
	\begin{align*}
	\TV(\theta^{cd}) \leq \|\theta^{cd}_c\|_1 \I\{a \, = \,c\} + \|\theta^{cd}_d\|_1 \I\{b \, = \,d\} + \delta + 16\:t\big(\sqrt{\frac{{m}}{{n}}} + \sqrt{\frac{{n}}{{m}}}\big)\,.
	\end{align*}
	In words, if a matrix $\theta \in \M^{2}(m,n,\delta,t)$ is dyadically partitioned into four equal sized submatrices, then each of these four submatrices has at most two active boundaries and a wiggle room of at most $\delta + 16t(\sqrt{\frac{{m}}{{n}}} + \sqrt{\frac{{n}}{{m}}})$; furthermore the active boundaries are the ones that they share with one of the active boundaries of $\theta$.
\end{lemma}

We are now ready to conclude the proof of Lemma~\ref{lem:finalparti}.

\begin{proof}[Proof of Lemma~\ref{lem:finalparti}]
	Recall that we are assuming $m,n$ are powers of $2$ for simplicity of exposition.

	\textbf{Step 0}: Partition $[m] \times [n]$ dyadically into four equal rectangles so that for any such rectangle $S$,  $\theta_{|S} \in \M^{2}(m/2,n/2, \delta^{(0)}, \norm{\theta_{|S}})$ by Lemma~\ref{lem:4to2} where 
	$$\delta^{(0)} = \delta + 16\:t\big(\sqrt{\frac{m}{n}} + \sqrt{\frac{n}{m}}\big)\,.$$

	\textbf{Step 1}: Let $S$ (there are four of them) be a generic rectangle obtained from the previous step. Using Lemma~\ref{lem:2to1}, we now partition $\theta_{|S}$ into four equal parts (rectangles).   
	We then get two matrices in $\M^1(m/4,n/4,\delta^{(1)}, t)$, one matrix in $\M^0(m/4,n/4,\delta^{(1)}, t)$ and \textit{the remaining one from $\M^2(m/4,n/4,\delta^{(1)}, t)$}. Here, $$\delta^{(1)} =  \delta^{(0)} + 16\:t(\sqrt{\frac{m}{n}} + \sqrt{\frac{n}{m}}) = \delta + 32\:t(\sqrt{\frac{m}{n}} + \sqrt{\frac{n}{m}})\,.$$

	\textbf{Steps~$\bm{j} \mb{\geq} \mb{2}$}: From the last step we get exactly one matrix in $\M^2(m/4,n/4,\delta^{(1)}, t)$, for each of the $4$ rectangles $S.$ For each $S,$ we now recursively use Lemma~\ref{lem:2to1} by partitioning this matrix again into four exactly equal parts in a dyadic fashion and continue the same procedure with the matrix obtained in each step with two active boundaries until we end up with matrices only with 0 or 1 active boundary. Observe that in the very last step we arrive at a submatrix with exactly one row or column in place of the one with two active boundaries.
	\newline

	For each $j \geq 1$, define $\mathcal R_{j,1}$ as the collection of rectangles $R$ obtained in step $j$ such that $\theta_{|R}$ has exactly $1$ active boundary. From Lemma~\ref{lem:2to1}, we know that there are exactly two such 
	rectangles for any given $S$ (from step $0$) 
	and therefore $|\mathcal R_{j,1}| \leq 8$. 
	For any $j \geq 1,$ and any rectangle $R \in \mathcal R_{j,1}$, repeated application of Lemma~\ref{lem:2to1} yields that $\theta_{|R} \in \mathcal M^{1}(m/2^j, n/2^j, \delta^{(j)}, \norm{\theta_{|R}})$ where
	$$\delta^{(j)} = \delta + 16(j + 1)\:t(\sqrt{\frac{m}{n}} + \sqrt{\frac{n}{m}}).$$
	Now defining $\mathcal R_{j,2}$ as the collection of rectangles $R$ obtained in step $j$ 
	such that $\theta_{|R}$ has no active boundary, we can deduce in a similar way that 
	$|\mathcal R_{j, 2}| \leq 4.$ Also for such rectangles $R$ and $j \geq 1$ we have 
	$\theta_{|R} \in \mathcal M^{0}(m/2^j, n/2^j,\delta^{(j)},\norm{\theta_{|R}})$. Finally, notice that 
	$$\sum_{j \geq 1}\sum_{R \in \mathcal R_{j, 1} \cup \mathcal R_{j, 2}} \norm{\theta_{|R}}^2 = \norm{\theta}^2 \leq t^2\,.$$
	Thus the collection of rectangles $\{\mathcal R_{j,k}: j\geq 1, k \in [2]\}$ satisfies all the conditions of Lemma~\ref{lem:finalparti}.


\end{proof}

\begin{remark}
	For the statement of Lemma~\ref{lem:gw4to1} to hold, the important thing in the proof of 
	Lemma~\ref{lem:finalparti} is that in every step $1 \leq j \leq K$, the aspect ratio of 
	the submatrices does not change significantly. The reader can check that at every step, 
	both the number of rows and columns halve, thus keeping the aspect ratio constant. At 
	every step, the dimensions of the submatrices halve and thus decrease geometrically, 
	while the allowable wiggle room increases additively by the factor (does not change with 
	$j$) $16 \:t(\sqrt{\frac{m}{n}} + \sqrt{\frac{n}{m}})$. 
\end{remark}

\begin{remark}
	Let us discuss the case when $m,n$ are not necessarily 
	powers of $2$ in the proof of Lemma~\ref{lem:finalparti}. 
	The first step of reducing the number of active boundaries 
	from four to two, by applying Lemma~\ref{lem:4to2}, can be carried out in the same way by splitting at the point 
	$\floor{m/2}$ and $\floor{n/2}$. Next, we come to the 
	stage when we are applying Lemma~\ref{lem:2to1} to reduce 
	the number of active boundaries from two to one, on the 
	four submatrices obtained from the previous step.  Let us 
	denote the dimensions of these $4$ submatrices generically 
	by $m', n'$. Recall, in the first step of subdivision, we 
	get exactly one submatrix with $2$ active boundaries. The others have $1$ or $0$ active boundaries. \textbf{At this step, we can subdivide such that the submatrix with two active boundaries has dimensions which are exactly powers 
		of $2$.} For instance, we can split at the unique power of $2$ between $m'/4$ and $m'/2$ on one dimension and do the 
	exact same thing for the other dimension. Once we have this submatrix with two active boundaries to have 
	dimensions which are exactly powers of $2$, we can carry out the rest of the steps as in the proof of 
	Lemma~\ref{lem:finalparti}. It can be checked that, in this case, all the inequalities we deduce while proving Lemma~\ref{lem:gw4to1} goes through with the possible 
	mutiplication of a universal constant. 
\end{remark}


\subsection{Upper bounds on Gaussian Widths and the proof of Proposition~\ref{prop:gwupbd}}
Now that we have reduced the problem of bounding the gaussian width of $T_{K(V^*)(\theta^*)} \cap B_{n, n}(1)$ to that of $\M^{0}(m,n,\delta,t)$ and $\M^{1}(m,n,\delta,t)$, we need to obtain upper bounds on these quantities in order to 
conclude the proof of Theorem~\ref{thm:adap}. Our next lemma provides an upper bound on the 
gaussian width of $\mathcal M^{0}(m, n, \delta, t)$ which we henceforth denote as $\gw^{0}(m, n, \delta, t)$.
\begin{lemma}\label{lem:gwnobdry}
	Fix $\delta > 0$ and $t \in (0, 1]$. For positive integers $m$ and $n$ such that $\max\{m/n,n/m\} \leq c$ for some $c > 0$, we have the following upper bound on the Gaussian width: 
	
	\begin{equation*}
	\gw^{0}(m,n,\delta,t) \leq C \big(\log \big(\frac{tn}{\delta} \vee \e\big) \delta (\log \e n)^{3/2} + t(\log \e n)^{1/2}\big)
	\end{equation*}
	where $C$ is a constant depending only on $c$.
\end{lemma}
\begin{proof}
	Since $m$ and $n$ are of the same order, the bound computed in \eqref{eq:gwkv} from Section~\ref{sec:thm1} remains valid in this case.
\end{proof}

In our next proposition, we provide an upper bound on $\gw^{1}(m,n,\delta,t)$, i.e., the gaussian width of $\mathcal M^{1}(m, n, \delta, 
t)$. This is the main result in this subsection and one of the main technical contributions of this paper. 
\begin{proposition}\label{prop:gwonebdry}
	Fix $\delta \in (0, n]$ and $t \in (0, 1]$. Then for positive integers $m,n$ satisfying the conditions of the previous lemma, we have the following upper bound on the Gaussian width: 
	\begin{equation}
	\gw^{1}(m,n,\delta,t) \leq C (\log \e n)^{9/2} n^{1/4}\sqrt{(t + \delta)^{2\downarrow}} + C\big( (\log \e n)^4 t + n^{-9}\big)\,.
	\end{equation}
	Here $x^{2\downarrow} \coloneqq x + x^2$ and 
	$C > 0$ is a constant depending solely on $c$.
\end{proposition}

We will prove the above proposition slightly later. Lemma~\ref{lem:gwnobdry} and Proposition~\ref{prop:gwonebdry} together with Lemma~\ref{lem:gw4to1} now imply (with $\gw^4(m, n, \delta, t)$ denoting the gaussian width of $\M^{4}(m,n,\delta,t)$)
\begin{lemma}\label{lem:4bdrycalc}
	Under the same condition as in the previous proposition, we have 
	\begin{equation*}
	\GW^4(m,n,\delta,t) \leq C (\log \e n)^{9/2} n^{1/4} (\sqrt{t} + \sqrt{\delta} + \delta) + C\big( (\log \e n)^{5} t + n^{-9} \log \e n \big)\,
	\end{equation*}
	where $C > 0$ is a universal constant.
\end{lemma}

The proof just involves collecting all the relevant terms and adding them up. The reader can safely skip the proof in the first reading.

\begin{proof}
	In this proof, we write $a \lesssim b$ to mean $a \leq C\, b$ for some positive constant $C$ --- depending at most on the aspect ratio $c$ --- whose exact value can change from line to line. 
	Recall that Lemma~\ref{lem:gw4to1} implies for $K \lesssim \log n$,
	\begin{align*}
	\GW^4(m,n,\delta,t) \lesssim \sum_{j = 1}^{K} 
	\big[\GW^1(\frac{m}{2^j},\frac{n}{2^j},\delta^{(j)},t) + \GW^0(\frac{m}{2^j},\frac{n}{2^j},\delta^{(j)},t)\big]\,.
	\end{align*}
	First we compute, in view of Proposition~\ref{prop:gwonebdry},
	\begin{align*}
	\sum_{j = 1}^{K}\GW^1(\frac{m}{2^j},\frac{n}{2^j},&\delta^{(j)},t)\lesssim  (\log \e n)^{9/2} \sum_{j = 1}^{K} (\frac{n}{2^{j}})^{1/4} \big(\sqrt{t + \delta^{(j)}} + t 
	+ \delta^{(j)}\big) + (\log \e n)^5t + n^{-9} \log \e n\\ &\lesssim  (\log \e n)^{9/2} \sum_{j 
	= 1}^{K} (\frac{n}{2^{j}})^{1/4} \big(\sqrt{t} + \sqrt{\delta^{(j)}} + t + \delta^{(j)}\big)  + (\log \e n)^5t + n^{-9} \log n\\&\lesssim  (\log \e n)^{9/2} \sum_{j = 1}^{K} (\frac{n}{2^{j}})^{1/4} \big(\sqrt{t} + \sqrt{\delta + jt} + t + \delta + jt\big) + (\log n)^5t + n^{-9}\log \e n\\&\lesssim (\log \e n)^{9/2} n^{1/4} \big(\sqrt{t} + \sqrt{\delta} + \delta\big) + (\log \e n)^5t + n^{-9} \log \e n
	\end{align*}
	where we have repeatedly used $\sqrt{a + b} \leq \sqrt{a} + \sqrt{b}$ and in the last inequality we have summed up the geometric series. On the other hand, Lemma~\ref{lem:gwnobdry} implies
	\begin{equation*}
	\sum_{j = 1}^{K} \GW^0(\frac{m}{2^j},\frac{n}{2^j},\delta^{(j)},t) \lesssim (\log n)^{5/2}\sum_{j = 1}^{K} (\delta + jt)\lesssim \delta (\log \e n)^{7/2} + (\log \e n)^{9/2} t\,. 
	\end{equation*}
	We can now deduce the lemma from the last two displays.
\end{proof}
With the help of the above lemma we can now conclude the proof of Proposition~\ref{prop:gwupbd}.
\begin{proof}[Proof of Proposition~\ref{prop:gwupbd}]
	
	Throughout this proof we will use the notation $C$ to denote some positive constant --- depending at most on the aspect ratio $c$ like in the previous proof --- whose exact value 
	may change from one line to the next. Also we will use ``$a \lesssim b$'' to mean ``$a \lesssim C b$''. Recall that by Lemma~\ref{lem:imp2}, $\GW(T_{K(V^*)}(\theta^{*}) \cap B_{n,n}(1))$ is at most
	\begin{align}
	\label{eq:GW_bnd}
	\max_{\Delta(\theta^*) \delta: \delta \in S_{k^*,2} \cap H_{k^*}}\: \max_{\bm {t^2} \in S_{k^*,2} \cap H_{k^*}} \:\sum_{i \in [k^*]} \GW^4(m_i, n_i,\Delta(\theta^*) \delta_i ,t_i) + \:\: C \sqrt{k^*}\,.
	\end{align}
	Now we plug in the bound from Lemma~\ref{lem:4bdrycalc} to obtain a bound on the sum inside the two maximums in the above display:
	\begin{align*}
	\sum_{i \in [k^*]} &\GW^4(m_i, n_i,\Delta(\theta^*) \delta_i ,t_i) \\ 
	&\lesssim (\log n)^{{9/2}}\sum_{i \in [k^*]}  n_i^{1/4} \big(\sqrt{t_i} + (\Delta(\theta^*) \delta_i)^{1/2} + \Delta(\theta^*) \delta_i\big) + (\log n)^5 \sum_{i \in [k^*]}t_i + k^* n^{-9}\log n\,.
	\end{align*}
	Since the aspect ratios of each of the rectangular level sets of $\theta^*$ are bounded by a constant, we have 
	$\sum_{i = 1}^{k^*} {n_i}^2 \lesssim n^2$. This can be seen as follows:
	$$\sum_{i = 1}^{k^*} n_i^2 \lesssim \sum_{i = 1}^{k^*} n_i m_i \lesssim mn \lesssim n^2.$$
	
	Therefore, we can repeatedly apply the Cauchy-Schwarz inequality to deduce for $\delta, \,\bm {t^2} \in S_{k^*,2}$,  
	\begin{align*}
	&\sum_{i = 1}^{k^*} n_i^{1/4} \sqrt{t_i} \lesssim (k^*)^{5/8} n^{1/4}\,,\:\:\: \sum_{i = 1}^{k^*} n_i^{1/4} \delta_i^{1/2} \lesssim (k^*)^{3/8} n^{1/4}\,,  \\&
	\sum_{i = 1}^{k^*} n_i^{1/4} \delta_i \lesssim n^{1/4} \:\:\text{and}\:\: \sum_{i = 1}^{k^*} t_i \lesssim \sqrt{k^*}\,.
	\end{align*} 
	Also because of constant aspect ratio, we have $$\Delta(\theta^{*}) = \sqrt{\sum_{i = 1}^{k^*} 2 \big(\frac{m_i}{n_i} + \frac{n_i}{m_i}\big)} \lesssim 
	\sqrt{k^*}\,.$$ Combining the last two displays we notice that 
	$(\log  n)^{9/2}(k^*)^{5/8}n^{1/4}$ emerges as the dominant term and hence
	\begin{align*}
	&\:\sum_{i \in [k^*]} \GW^4(m_i, n_i,\Delta(\theta^*) \delta_i ,t_i) \lesssim (\log n)^{9/2} (k^*)^{5/8} n^{1/4}\,.
	\end{align*}
	Together with \eqref{eq:GW_bnd} this finishes the proof.
\end{proof}

All that remains towards the proof of Proposition~\ref{prop:gwupbd} is Proposition~\ref{prop:gwonebdry}. The proof of this proposition is fairly involved. The rest of this section is devoted to its proof.

\subsection{Proof of Proposition~\ref{prop:gwonebdry}} \label{sec:onebdry}
By symmetry, it is enough to bound $\gw(\mathcal M^{\rt}(m,n,\delta,t))$. To this end, let us introduce a new class of matrices as follows:
\begin{equation*}
\mathcal{A}(m,n,u,v,t) \coloneqq \{\theta \in \R^{m \times n}: \TVr(\theta) \leq u, \TVc(\theta) \leq v, \|\theta\| \leq t\}\,
\end{equation*}
where, let us recall, that the total variation $\TVr(\theta)$ along rows is defined as
$$\TVr(\theta) \coloneqq \sum_{i \in [m]}\sum_{j \in [n-1]} |\theta[i, j + 1] - \theta[i, j]|$$
and $\TVc(\theta) \coloneqq \TVr(\theta^T)$.



The following lemma gives an upper bound of $\gw(\mathcal M^{\rt})$ in terms of the Gaussian widths of $\mathcal{A}$ with appropriate parameters.  
\begin{lemma}\label{lem:A}
	Let $k$ denote the smallest integer satisfying $(1 + 2 + \dots 2^{k}) \geq n$. Then we have the following inequality:
	$$\gw^\rt(m,n,\delta,t) \leq \sum_{j \in [k]} \gw(\mathcal A(m,\: n_j,\: 2t\sqrt{m / n_{j}} + \:\delta, \:t\sqrt{m / n} + \delta,\:t))\,,$$
	where $n_j = 2^j$ for $j \in [k-1]$ and $n = \sum_{j \in [k]}n_j$.
\end{lemma}


\begin{proof}
	The proof proceeds by dividing the $n$ columns into blocks of geometrically increasing length and showing that for any $\theta \in \mathcal M^{\rt}(m,n,\delta,t)$ the submatrices defined by the blocks live in $\mathcal{A}$ with appropriate parameters. Let $\theta \in \mathcal M^\rt(m,n,\delta,t)$ and subdivide $\theta$ into submatrices $[\theta^{(k)} \rvert \theta^{(k-1)}\rvert\cdots \rvert \theta^{(1)}]$ where 
	$\theta^{(j)}$ has $n_j$ many columns. Therefore it suffices to prove that $$\theta^{(j)} \in \mathcal A(m,\: n_j,\: t\sqrt{m / n_{j-1}} + \:\delta, \:t\sqrt{m/n} + 
	\delta,\:t)$$ for all $j \in [k]$ as $\sqrt{n_j / n_{j-1}} < 2$. Since $\norm{\theta^{(j)}} \leq \norm{\theta} \leq t$, we only need to verify the required bounds on $\TVc(\theta^{(j)})$ and $\TVr(\theta^{(j)})$.
	
	\medskip
	
	\noindent{\em Verifying the bound on $\TVc(\theta^{(j)})$\,.} We will prove the stronger statement $\TVc(\theta) \leq t\sqrt{m/n} + \delta$. Since $\norm{\theta} \leq t$ and 
	$\norm{\theta}^2 = \sum_{\ell \in [n]}\norm{\theta[\,, 
		\:\ell]}^2$, it follows that $\norm{\theta[\,, 
		\:\ell^*]} \leq t / \sqrt{n}$ for some $\ell^* \in [n]$ and hence $\norm{\theta[\,, \:\ell^*]}_1 \leq t\sqrt{m/n}$ by 
	the Cauchy-Schwartz inequality. Now using the condition that  $\TV(\theta) \leq \norm{\theta[\,,\:n]}_1 + \delta$ (from the definition of $\mathcal M^\rt(m,n,\delta,t)$), we get
	\begin{align}
	\label{eq:lem:main1}
	\norm{\theta[\,,\:n]}_1 + \delta - \TVc(\theta) &\geq \TV(\theta) - \TVc(\theta) = \TVr(\theta) \geq \norm{\theta[\,,\:n] -\theta[\,,\:\ell^*]}_1 \nonumber\\
	&\geq \norm{\theta[\,,\:n]}_1  - \norm{\theta[\,,\:\ell^*]}_1 \geq \norm{\theta[\,,\:n]}_1  - t\sqrt{m/n}\,.
	\end{align}
	Thus $\TVc(\theta) \leq \delta + t\sqrt{m/n}$. 
	
	\medskip
	
	\noindent{\em Verifying the bound on $\TVr(\theta^{(j)})$\,.} Let us start with $\theta^{(1)}$. By the Cauchy-Schwartz inequality, $\norm{\theta[\,,\:n]}_1 \leq \sqrt{m} \norm{\theta[\,,\:n]}_2 \leq t\sqrt{m}$ and thus 
	$$\TVr(\theta^{(1)}) \leq \TV(\theta) \leq \norm{\theta[\,,\:n]}_1 \leq t\sqrt{m}\,.$$
	Next consider $\theta^{(j)}$ for some $j \geq 2$. Since $\norm{\theta^{(j-1)}}_2^2 \leq t$ and it has $n_{j-1}$ columns, there is a column of $\theta^{(j-1)}$ whose 
	$\ell_2$-norm is at most $t / \sqrt{n_{j-1}}$. Suppose this column is $\theta[\,,\:a]$. Then a calculation similar to \eqref{eq:lem:main1} yields,
	\begin{align*}
	&\norm{\theta[\,,\:n]}_1 + \delta \geq \TV(\theta) \geq \TVr(\theta) \geq \TVr([\theta^{(j-1)} \rvert \theta^{(j-2)}\rvert\cdots \rvert \theta^{(1)}]) + \TVr(\theta^{(j)}) \\\geq & \:\norm{\theta[\,,\:n] - \theta [\,,\:a]}_1 + \TVr(\theta^{(j)}) \geq \norm{\theta[\,,\:n]}_1 - \norm{\theta[\,,\:a]}_1 + \TVr(\theta^{(j)}).
	\end{align*}
	But this implies, along with the Cauchy-Schwartz inequality, that
	\begin{equation*}
	\TVr(\theta^{(j)}) \leq \norm{\theta[\,,\:a]}_1 + \delta \leq \sqrt{m} \norm{\theta[\,,\:a]}_2 + \delta \leq t\sqrt{m/n_j} + \delta\,. \qedhere
	\end{equation*}
\end{proof}

It therefore suffices, in view of the previous lemma, to bound the gaussian width of each $\mathcal A(m, n_j,$ 
$2t\sqrt{m/n_j} + \:\delta, \:t\sqrt{m/n} + \delta,t)$ from above in order to bound $\gw^\rt(m, n, \delta, t)$.
Defining $a = \sqrt{m / n_{j}}$, we can write 
$$\mathcal A(m, n_j, 2t\sqrt{m/n_j} + \:\delta, \:t\sqrt{m/n} + \delta,t) = \mathcal A(m, m/a^2, 2ta + \:\delta, \:t\sqrt{m / n} + \delta,t) \eqqcolon \mathcal A_{a}\,.$$
Notice that we suppressed the dependence on $m, n, \delta$ and $t$ which henceforth refer to the corresponding parameters in Proposition~\ref{prop:gwonebdry}.

In our next result, which is crucial for the proof of Proposition~\ref{prop:gwonebdry}, we give a subspace cover for the set $\mathcal A_a$ corresponding to any distance $\tau$ between {$1 / m$} and 1.
\begin{lemma}\label{lem:main}
	Let $t \leq 1$, $\tau \in [1/m, 1]$ and {{$a \geq c$}} be such that $m /a^2$ is a positive integer between $1$ and $n$. Here $c$ is from the statement of Proposition~\ref{prop:gwonebdry}. Then there exists a $\tau$ subspace cover $\mathcal S_{\tau}$ of $\mathcal A$, depending on $m, n, a, \delta$ and $t$ in addition to $\tau$, and a constant $C > 0$ depending solely on $c$ such that
	$$\max(\log |S_\tau|,\, \max_{S \in S_\tau} \dim(S)) \le C(\log(\e m))^{3}\mathcal L_m\big ( 1 + \frac{\sqrt{m}C_{m, n, \delta, t}}{\tau^2}\big)$$
	where $\mathcal L(x) \coloneqq x \log (\e \log(\e m)^2 x)$ and 
	$$C_{m, n, \delta, t} \coloneqq \log(\e m)\big(t\sqrt{\frac{m}{n}} + t + \delta\big)^{2\downarrow} \,$$
	(recall that $x^{2\downarrow} \coloneqq x + x^2$). 
	%
	%
	%
\end{lemma}

\begin{remark}\label{remark:main_lead}
	Notice that $\mathcal L(x)$ is linear in $x$ ignoring the $\log$ factors. Thus it is helpful to read the above bound as scaling like $\tfrac{\sqrt{m}}{\tau^2}$ up to log 
	factors and the lower order terms. This $\sqrt{m}$-scaling is crucial for us in order to derive the $1/4$ exponent of $n$ in Proposition~\ref{prop:gwonebdry} and subsequently the correct exponent of $n$ in Theorem~\ref{thm:adap}.
\end{remark}
\begin{remark}\label{remark:main}
	The reason for assuming a polynomial lower bound (in $m$) on $\tau$ is that we want $\log (1 /\tau)$ to be at most $O(\log m)$. Hence the bounds of Lemma~\ref{lem:main} remain valid, with appropriate changes in $C$, as long as $\tau \geq 1 / m^c$ for some universal constant $c > 0$. 
\end{remark}

With Lemma~\ref{lem:main} we can now finish the proof of Proposition~\ref{prop:gwonebdry}. 
\begin{proof}[Proof of Proposition~\ref{prop:gwonebdry}]
	An important feature of the bounds in Lemma~\ref{lem:main} is that it does not depend on $a$. Hence an application of Proposition~\ref{prop:chaining} would yield the same bound on each Gaussian width appearing inside the summation in the 
	statement of Lemma~\ref{lem:A}. From this we can deduce Proposition~\ref{prop:gwonebdry} in a straightforward manner. The detailed computation is given below. In the remainder of the proof we will use $C$ to denote any positive constant depending {\em at most} on $c$ whose exact value may change from one line to the next.
	
	Applying Proposition~\ref{prop:chaining} with $k_0 = \lfloor - \log_2 2t\rfloor$ and $k_1 = -\lceil \log_2 \nu \rceil$ where $\nu = t/m \vee m^{-10}$ and using Lemma~\ref{lem:main} subsequently to bound the relevant terms (see Remark~\ref{remark:main_lead}), we get
	\begin{align*}
	\GW(\mathcal A_a) \leq C\sum_{k = k_0 + 1}^{k_1} 2^{-k}(\log(\e m))^{1.5}\sqrt{\mathcal L_m\big ( 1 + 2^{2k}\,\sqrt{m}\,C_{m, n, \delta, t}\big)} \, + \, \sqrt{m n}\,\nu\,.
	\end{align*}
	Now recalling the definition of $\mathcal L_m(\cdot)$, we can write
	\begin{align*}
	\mathcal L_m&\big( 1 + 2^{2k}\sqrt{m}C_{m, n, \delta, t}\big) = \big(1 + 2^{2k}\sqrt{m}C_{m, n, \delta, t}\big) \:\big(1 + \log \log (em)^2 + \log 2^{2k}\sqrt{m}C_{m, n, \delta, t}\big) \\& \leq\big(1 + 2^{2k}\sqrt{m}C_{m, n, \delta, t}\big) \:\big(1 + \log \log (em)^2 + \log (m^{21} C_{m, n, \delta, t})\big) \\& \leq C \log (\e m (1 + \delta))\,(1 + 2^{2k}\sqrt{m}C_{m, n, \delta, t})
	\end{align*}
	where in the last inequality we used the fact that $C_{m,n,\delta,t} \leq C (1 + \delta) \log (\e m) $ since $t \leq 1$ and $m/n$ is assumed to be bounded by a constant. The last two displays therefore imply
	\begin{align*}
	\GW(\mathcal A_a) &\leq C (\log (\e m))^{1.5}\sqrt{\log (\e m (1 + \delta))} \,\, (t + m^{1/4} \log (\e m) \sqrt{C_{m, n, \delta, t}}\,)  + \sqrt{tn/m} + \sqrt{n} / m^{20}\\& 
	\leq C (\log \e n)^{3.5} n^{1/4}\sqrt{(t + \delta)^{2\downarrow}} + C (\log \e n)^3 t + C n^{-9.5}\,
	\end{align*}
	where in the final step we used the fact that $\delta \in (0, n]$ as well as 
	$\max\{m/n, n/m\} \le c$. 
	The proposition now follows from summing this bound over $k$ as in Lemma~\ref{lem:A}.
\end{proof}

The thing that remains to be done is the proof of Lemma~\ref{lem:main}. 
An important ingredient is the following weaker analogue for the general case.
\begin{lemma}\label{lem:prep}
	Let $k, m , n$ be positive integers with $1 \le k \le m$ (not to be confused with the parameters in Lemmata~\ref{lem:A} -- \ref{lem:main}). Also let $t \leq 1$ and $u, v, \tau > 0$. Then there exists a $\tau$ subspace cover $S_\tau$ of $\mathcal A(m, n, u, v, t)$, depending on $m, n, k, u, v$ and $t$ in addition to $\tau$, and a universal constant $C > 0$ such that
	$$\max(\log |S_\tau|,\, \max_{S \in S_\tau} \dim(S)) \le C\big(J_k + \sqrt{J_k} \frac{v\sqrt{m}}{\tau\sqrt{k}}\big)\log\big(\e mJ_k + \e m\sqrt{J_k} \frac{v\sqrt{m}}{\tau\sqrt{k}}\big)
	$$
	when $k < m$, whereas for $k = m$
	$$\max(\log |S_\tau|,\, \max_{S \in S_\tau} \dim(S)) \le CJ_k\log(\e mJ_k). 
	$$
	Here
	$$J_k \coloneqq C \log (\e n)\big(k + \frac{u\sqrt{nk}}{\tau}\big).$$
	%
	%
	%
\end{lemma}

\begin{remark}
	\label{remark:covering_no}
	Lemma~\ref{lem:prep}, by itself, is not sufficient to prove Lemma~\ref{lem:main}. To see this, let us plug in $n = m / 
	a^2$ and $u = 2ta$ in the expression for $J_k$. One can easily check that while this makes $J_k$ free from $a$, the principal terms in the bounds on the dimension and cardinality do not attain the required $\sqrt{m}$-scaling for any choice of $k$.
\end{remark}

In the course of proving Lemma~\ref{lem:prep}, we will repeatedly use a subdivision scheme 
based on the value of either $\TVr$ or $\TVc$. We will also use it in the proof of 
Lemma~\ref{lem:main} and therefore describe it here in a general setting. Let us point out 
that a very similar scheme was described in Section~\ref{sec:division_proof} in the context of proving Theorem~\ref{Thm:1}.

\smallskip

{\bf A greedy partitioning scheme: ~}Consider a set $\mathcal S$ and a function $T: \cup_{n \in \N}\mathcal S^n \mapsto \R_{\geq 0}$ satisfying $T(AB) \geq T(A) + T(B)$ for all $A, B \in \cup_{n \in \N} \mathcal S^n$ where $AB$ denotes the concatenation of $A$ and $B.$ Also suppose for any singleton $s \in S,$ the function 
$T$ satisfies $T(s) = 0$. To relate this to a concrete example, the reader may consider the case where $\mathcal S = \R^m$ so that $\mathcal S^n \equiv \R^{m \times n}$ and $T$ is the function $\TVr$. Now for any $\ep > 0$, the \emph{$(T, \ep)$ scheme} subdivides an element $U$ of $\cup_{n \in \N} \mathcal S^n$ as $U_1U_2\cdots U_K$ such that $T(U_i) \leq \ep$ for 
all $i \in [K]$. This is achieved in several steps of binary division as follows. In the first step, we check whether $T(U) \leq \epsilon.$ If so, then stop and output $U.$ Else, divide $U$ as $U'_1U'_2$ into two almost equal parts. This means $|U'_1| = \lfloor |U|/2 \rfloor$ and $|U'_2| = |U'| - |U'_1|.$ In each step, we have a representation of $U$ of the form $U'_1U'_2\cdots U'_{K'}.$ We consider each $i \in 
[K']$ such that $T(U'_i) > \ep$ and subdivide $U'_i$ into two almost equal parts. We repeat this procedure until each part $U'$ in the current representation satisfies $T(U') \leq \ep$.

Suppose that $|U| = n$. The subdivision of $U$ produced by the $(T, \ep)$ scheme 
corresponds to a partition of $[n]$ into contiguous blocks, say, $P_{U; T, \ep}$. Let $|P_{U; T, \ep}|$ denote the number of blocks of the partition $P_{U; T, \ep}.$ Now for $t > 0$, let $\mathcal P(t, n, \ep, T)$ denote the set of partitions $\{P_{U; T, \ep}: U \in \mathcal S^n, T(C) \leq t\}$. A key ingredient in the proof of Lemma~\ref{lem:prep} (and subsequently Lemma~\ref{lem:main}) is the following universal upper bound on the cardinality of $P_{U; T, \ep}$.
\begin{lemma}\label{lem:division2}
	Then for the $(T, \ep)$ division scheme we have 
	\begin{equation*}
	\max_{P \in \mathcal P(t, n, \ep, T)} |P_{U; T, \ep}| \leq \log_2 (4n) (1 + \frac{t}{\epsilon})\,.
	\end{equation*}
\end{lemma}

The proof of Lemma~\ref{lem:division2} is very similar to that of Lemma~\ref{lem:division}. Nevertheless, for the sake of completeness, we provide its proof in the appendix (see Section~\ref{subsec:division2}). We also defer the proof of Lemma~\ref{lem:prep} to the end of this subsection and finish the proof of Lemma~\ref{lem:main} assuming it.
\begin{proof}[Proof of Lemma~\ref{lem:main}]
	Take any 
	$\theta \in \mathcal{A}_{a}$ 
	and fix $\ep \in (0, 1)$ whose precise value based on $\tau$ would be chosen later. Let us denote the $m' \times n'$ two dimensional grid (graph) by $L_{m', n'}$ 
	and subdivide $L_{m, m/a^2}$ as
	\begin{equation}
	\label{eq:K}
	L_{m, m/a^2} = \begin{bmatrix}
	R_1 \\
	R_2 \\
	\vdots \\
	R_K
	\end{bmatrix}
	\end{equation}
	where $\TVc(\theta_{|R_i}) \leq \ep$ for all $i \in [K]$ and $K \leq \log_2 (4m) (1 + \TVc(\theta)\ep^{-1})$. We achieve this by applying the $(\TVc,\epsilon)$ division scheme to the rows of $\theta$ (see Lemma~\ref{lem:division2}). Denoting the set of all possible partitions of $L_{m, m/a^2}$ obtained in this manner by $\mathcal P$, we deduce
	\begin{equation}
	\label{eq:mathcalP1}
	|\mathcal P| \leq m^{\log_2 (4m) (1 + \,(t\sqrt{\tfrac{m}{n}} + \delta)\ep^{-1})}\,.
	\end{equation}
	Corresponding to the partition $P = P(\theta)$ in \eqref{eq:K}, let $S_{P, \rm row}$  denote the linear subspace of $\R^{m \times m/a^2}$ comprising only matrices having identical rows in each $R_i$. It is clear that the orthogonal projection of $\theta$ onto $S_{P, \rm row}$ is given by
	\begin{equation*}
	\hat \theta_{S_{P, \rm row}} = \tilde \theta = \begin{bmatrix}
	\tilde \theta_1 \\
	\tilde \theta_2 \\
	\vdots \\
	\tilde \theta_K
	\end{bmatrix}\,
	\end{equation*}
	where each row of $\tilde \theta_i \coloneqq \tilde \theta_{|R_i}$ is equal to the average row of $\theta_{R_i}$. By repeated application of Lemma~\ref{prop:1dtvapprox} (stated and proved in the appendix), we obtain
	\begin{equation}
	\label{eq:approx1}
	{\rm dist}(\theta, S_{P, \rm row}) = \|\theta - \tilde \theta\|_2 \leq  
	\sqrt{m}\ep\,. 
	\end{equation}
	Also by standard properties of orthogonal projections, it follows that $\|\tilde \theta\|_2 \leq \|\theta\|_2 \leq t$. We further claim that 
	$\tilde \theta \in \mathcal A_{a} \equiv \mathcal A(m, \frac{m}{a^2}, 2 t a + 
	\:\delta, \:t\sqrt{\frac{m}{n}} + \delta,t)$. Hence to establish this claim we only need to show that $\TVr(\tilde \theta) \leq \TVr(\theta)$ and $\TVc(\tilde \theta) 
	\leq \TVc(\theta)$. We can obtain the first inequality as follows:
	\begin{align}
	\label{eq:preserve_tvr}
	\TVr(\tilde \theta) &= \sum_{i \in [K]}\nr(R_{i})\sum_{\ell \in [\nc(R_{i}) - 1]}|\tilde \theta_{i}[1, \ell + 1] - \tilde \theta_{i}[1, \ell]| \nonumber \\
	&=\sum_{i \in [K]}\nr(R_{i})\sum_{\ell \in [\nc(R_{i}) - 1]}\Big \vert \nr(R_{i})^{-1}\sum_{i' \in [\nr(R_{i})]}\big(\tilde \theta_{i}[i', \ell + 1] - \tilde \theta_{i}[i', \ell]\big)\Big \vert \nonumber \\
	&\leq \sum_{i \in [K]}\sum_{\ell \in [\nc(R_{i}) - 1]} \sum_{i' \in [\nr(R_{i})]}\vert\tilde \theta_{i}[i', \ell + 1] - \tilde \theta_{i}[i', \ell]\vert = \TVr(\tilde \theta)\,.
	\end{align}
	For the second inequality we just apply Lemma~\ref{approxdecreasestv} (stated and proved in the appendix section) to each column of $\theta$. 

	In the rest of the article we call a subset of $L_{m' \times n'}$ a {\em subgrid} if it is a product of subsets (as opposed to only subintervals) of $[1, m'] \cap \N$ and $[1, n'] \cap \N$ respectively. We will now regroup $R_i$'s into several subgrids. For any positive integer $\ell$ such that $2^\ell \leq {2}m$, define {the set} $S_\ell \coloneqq \big\{i \in [K]: 2^{\ell - 1} \leq \nr(R_i) < 2^\ell\big\}$ and let $B_\ell$ be 
	the vector which is the sorted version of $S_\ell$. Now consider the {\em subgrid} of $L_{m, m / a^2}$
	$$R^\ell \coloneqq \begin{bmatrix}
	R_{B_\ell(1)} \\
	R_{B_\ell(2)} \\
	\vdots \\
	R_{B_\ell(K_\ell)}
	\end{bmatrix}\,$$
	where $K_\ell \coloneqq |B_\ell|$. In words, $R^\ell$ comprises the rectangles $R_i$, in order, whose number of rows lies between $2^{\ell - 1}$ and $2^\ell$. It is clear 
	that $R^1, R^2, \dots, R^L$ are disjoint subgrids of $L_{m, n}$ where $L \leq 
	\log_2(2 m)$. 
	Let us also denote $\tilde \theta_{|R^\ell}$ by $\tilde \theta^\ell$. Notice that if the matrices $\hat \theta^1, \hat \theta^2, \dots, \hat \theta^L$ satisfy $\|\tilde \theta^\ell - \hat \theta^\ell \| \leq \sqrt{m}\ep$ for all $\ell \in [L]$ and $\hat \theta \in \R^{m \times m/a^2}$ is such that $\hat \theta_{|R^\ell} = \hat \theta^\ell$ for all $\ell \in [L]$, then we have
	\begin{equation}
	\label{eq:tau_ep}
	\|\theta - \hat \theta\| \leq \|\theta - \tilde \theta\| + \|\tilde \theta - \hat \theta\|\overset{\eqref{eq:approx1}}{\leq} \sqrt{m\epsilon^2} + \sqrt{m \log_2(2 m) \epsilon^2} \leq \:\sqrt{2\:m \log_2(4m)}\:\epsilon\,.
	\end{equation}
	We now choose $\ep$ by requiring this approximation error to be $\tau$, i.e., by setting $\ep = \tau / \sqrt{2\:m \log_2(4m)}$ (notice that $1 /4m^{2} \leq \ep \leq 1 / 
	\sqrt{m}$ when $\tau \in [1 / m, 1]$). Therefore if $\mathcal S_{\tau, P}^\ell$ is a $\sqrt{m}\epsilon$ subspace cover for the family $\mathcal A_{\ell, P}^*$ (say) of matrices $\tilde \theta^\ell$ corresponding to $P \in \mathcal P$ and $\ell \in [L]$, we can immediately obtain a $\tau$ subspace cover $S_\tau$ for $\mathcal A_{a}$ satisfying:
	\begin{equation}
	\label{eq:covering_number_main0}
	\max_{S \in S_\tau} \dim(S) \leq \max_{P \in \mathcal P} \sum_{\ell \in [L]} \max_{S \in \mathcal S_{\tau, P}^\ell} \dim(S)
	\end{equation}
	and
	\begin{equation}
	\label{eq:covering_number_main}
	|\mathcal S_\tau|\leq |\mathcal P|\,.\max_{P \in \mathcal P} \prod_{\ell \in [L]}|\mathcal S_{\tau, P}^\ell|.
	\end{equation}
	
	Now fix a $P \in \mathcal P$ and let $\Theta^\ell$ denote the matrix formed by the first (or any) rows of $\tilde \theta_{B_\ell(1)}, \tilde \theta_{B_\ell(2)}, \dots,$ $\tilde \theta_{B_\ell(k_\ell)}$ in order, i.e., the rows 
	of $\tilde \theta^\ell$ that are \emph{potentially} 
	distinct. We claim that
	\begin{equation}
	\label{eq:Thetaell}
	\Theta^\ell \in \mathcal A\big (K_\ell,\: \frac{m}{a^2}, \: \frac{2 t a + \delta}{2^{\ell - 1}}, \:t\sqrt{\frac{m}{n}} + \delta,\:t\big) \eqqcolon \mathcal A_{a, \ell} \,(= \mathcal A_{\ell, P})\,.
	\end{equation}
	The constraints on the number of rows and columns of $\Theta^\ell$ as well as $\|\Theta^\ell\|$ are clear. For the remaining constraints first observe that $\tilde \theta^\ell \in \mathcal A(\nr(R^\ell), \frac{m}{a^2}, 2 t a + \:\delta, \:t\sqrt{\frac{m}{n}} + \delta,t)$ (the only non-obvious part is the bound on $\TVc(\tilde \theta^\ell)$ which follows from the 
	triangle inequality). From the definition of $\Theta^\ell$ it is immediate that 
	$$\TVc(\Theta^\ell) = 
	\TVc(\tilde \theta^\ell)\,\, \mbox{and}\,\,\TVr(\Theta^\ell) \leq
	\frac{\TVr(\tilde \theta^\ell)}{\min_{i \in [K_\ell]}\nr(R_{B_\ell(i)})}\,.$$
	Therefore the bounds on $\TVc(\Theta^\ell)$ and $\TVr(\Theta^\ell)$ follow from the similar bounds for $\tilde \theta^\ell$ and the fact that $\nr(R_{B_\ell(i)}) \geq 2^{\ell - 1}$ for each $i \in [K_\ell]$.

	Further notice that since $\nr(R_{B_\ell(i)}) < 2^{\ell}$ for each $i \in [K_\ell]$, we have $\|\tilde\theta^\ell - \hat\theta^\ell\| 
	\leq 2^{\ell / 2}\|\Theta^\ell - \hat\Theta^\ell\|$ where $\hat \theta^\ell$ comprises repetitions of the rows of $\hat \Theta^\ell$ in the same way as $\tilde{\theta}^\ell$ comprises repetitions of the rows of $\tilde{\Theta}^\ell.$ Therefore any $2^{-\ell / 2}\sqrt{m}\ep$ subspace cover $\mathcal S_{\epsilon}^\ell$ for $\mathcal A_{a, \ell}$ induces a $\sqrt{m}\ep$ subspace cover $\mathcal S_{\tau, P}^\ell$ for $\mathcal A^*_{\ell, P}$. Our next claim is about a \emph{uniform} upper bound on $\max_{S \in \mathcal S_{\epsilon}^\ell} \dim(S)$ and $|\mathcal S_{\epsilon}^\ell|$ for some particular choice of $\mathcal S_{\epsilon}^\ell$ and hence that of $\max_{S \in \mathcal S_{\tau, P}^\ell} \dim(S)$ and $|\mathcal S_{\tau, P}^\ell|$ as well.
	\begin{claim}
		\label{claim:optimize}
		There is a choice of $\mathcal S_{\epsilon}^\ell$ for any $\ell \in \N_{>0}$ and $\ep \in [1/m^2, 1 /\sqrt{m}]$ such that for some universal constant $C > 0$,
		\begin{align*}
		\max(\log |\mathcal S_{\epsilon}^\ell|,\, \max_{S \in \mathcal S_{\epsilon}^\ell} \dim(S)) \leq C \log(\e m)^2\mathcal L_m\big( 1 + \frac{1}{\sqrt{m}\epsilon^2}\big(t\sqrt{\frac{m}{n}} + t + \delta\big)^{2\downarrow} \big)
		\end{align*}
		where we recall from the statement of Lemma~\ref{lem:main} that $\mathcal L(x) = x \log (\e \log(\e m)^2 x)$ and $x^{2\downarrow} = x + x^2$.
	\end{claim}
	Claim~\ref{claim:optimize} follows directly from Lemma~\ref{lem:prep} when we choose $k$ in an \emph{appropriate} manner. The complete proof is given after the current proof.

	\noindent{\bf Concluding the proof.} In the remainder of the proof we will use $C$ to denote any positive, universal constant whose exact value may change from one line to the next. Using Claim~\ref{claim:optimize} let us first bound
	\begin{align*}
	\max_{P \in \mathcal P} \sum_{\ell \leq \log_2{2 m}}\max(\log |\mathcal S_{\tau, P}^\ell|,\, \max_{S \in \mathcal S_{\tau, P}^\ell} &\dim(S)) \leq C (\log(\e m))^3\mathcal L_m\big( 1 + \frac{1}{\sqrt{m}\epsilon^2}\big(t\sqrt{\frac{m}{n}} + t + \delta\big)^{2\downarrow} \big)\nonumber\\
	&\leq C(\log(\e m))^3\mathcal L_m\big( 1 + \frac{\sqrt{m}\log(\e m)}{\tau^2}\big(t\sqrt{\frac{m}{n}} + t + \delta\big)^{2\downarrow} \big)\nonumber \\
	&= C(\log(\e m))^3\mathcal L_m\big ( 1 + \frac{\sqrt{m}C_{m, n, \delta, t}}{\tau^2}\big)
	\end{align*}
	where we used the fact that $\tfrac{1}{\sqrt{m}\epsilon^2} = \tfrac{\sqrt{m}\log_2(4m)}{\tau^2}$ (recall the choice of $\ep$ after \eqref{eq:tau_ep} and also the definition of 
	$C_{m, n, \delta, t}$ from the statement of Lemma~\ref{lem:main}). On the other hand, since $\ep \leq 1 / \sqrt {m}$, we can bound $\log |\mathcal P|$ in view of \eqref{eq:mathcalP1} as
	\begin{align*}
	\log_2 (4m) &(1 + \,(t\sqrt{\tfrac{m}{n}} + \delta)\frac{1}{\ep})\log m \leq \log_2 (4m) (1 + \,(t\sqrt{\tfrac{m}{n}} + \delta)\frac{1}{\sqrt{m}\ep^2})\log m \nonumber\\
	&\leq C \log (\e m)^2  \big ( 1 + \frac{\sqrt{m}C_{m, n, \delta, t}}{\tau^2}\big )\,.
	\end{align*}
	Since $\mathcal L(x) \geq x$ for all $x \geq 1$, we deduce by combining the previous two displays and subsequently plugging them into \eqref{eq:covering_number_main0}--\eqref{eq:covering_number_main}:
	\begin{equation*}
	\max(\log |\mathcal S_{\tau}|,\, \max_{S \in \mathcal S_{\tau}} \dim(S)) \leq C(\log(\e m))^3\mathcal L_m\big ( 1 + \frac{\sqrt{m}C_{m, n, \delta, t}}{\tau^2}\big)\,. \qedhere
	\end{equation*}
\end{proof}

\smallskip

\begin{proof}[Proof of Claim~\ref{claim:optimize}]
	The ``main'' contribution in the bounds on $\log |\mathcal S_{\epsilon}^\ell|$ and $\max_{S \in \mathcal S_{\epsilon}^\ell} \dim(S)$ given by Lemma ~\ref{lem:prep} comes from
	\begin{equation*}
	J_{k, \ell}^* \coloneqq \big(J_{k, \ell} + \sqrt{J_{k, \ell}}\frac{(t\sqrt{\frac{m}{n}} + \delta)\sqrt{K_\ell}}{2^{-\ell/2}\sqrt{m}\ep\sqrt{k}}\I\{k < K_\ell\}\big)
	\end{equation*}
	where
	\begin{equation}
	\label{eq:Jkl}
	J_{k, \ell} = C \log (\e m / a^2)\big (k + \frac{2^{-\ell}(2ta + \delta)\sqrt{mk}}{a2^{-\ell/2}\sqrt{m}\ep}\big) \overset{{{a \geq c}}}{\leq} C \log (\e m)\big(k + \frac{2^{-\ell}(2t + \delta)\sqrt{mk}}{2^{-\ell/2}\sqrt{m}\ep}\big)\,
	\end{equation}
	(recall the statement of Lemma~\ref{lem:prep} and \eqref{eq:Thetaell}). Therefore, as already mentioned in the proof of Lemma~\ref{lem:main}, we will apply Lemma~\ref{lem:prep} for some $k \in [K_\ell]$ so 
	that $J_{k, \ell}^*$ has a small value. In the rest of the proof we will use $C$ to denote an unspecified but universal positive constant whose value may change from one 
	instant to the next. Using the simple fact $\sqrt{x + y} \leq \sqrt{x} + \sqrt{y}$, we can bound $J_{k, \ell}^*$ as follows:
	\begin{equation}
	\label{eq:choose_k2}
	J_{k, \ell}^* \leq J_{k, \ell} + \I\{k < K_\ell\}\,C_{\ep, k, \ell}\,,
	\end{equation}
	where
	\begin{equation}
	\label{eq:C_epkl}
	C_{\ep, k, \ell} \coloneqq C\sqrt{\log (\e m)}(t\sqrt{\frac{m}{n}} + \delta)\,\sqrt{K_\ell}\,\,\big(\frac{1}{2^{-\ell/2}\sqrt{m}\ep} + \frac{2^{-\ell/2}\sqrt{(2t + \delta)}m^{1/4}}{(2^{-\ell/2}\sqrt{m}\ep)^{3/2}k^{1/4}}\big)\,.
	\end{equation}
	Now let us consider two cases separately based on whether $\sqrt{K_\ell} 2^{-\ell/ 2}\sqrt{m}\ep$ is smaller or larger than 1. Recall that $2^{-\ell/ 2}\sqrt{m}\ep$ is the covering radius in question, and the condition above is equivalent to $K_\ell$ being smaller or larger than the inverse of the covering radius squared. 
	
	\smallskip

	{\em Case~1: $K_\ell \leq \frac{2^{\ell}}{m \epsilon^2}$.} In this case we choose $k = K_\ell$ so that Lemma~\ref{lem:prep} and \eqref{eq:Jkl} together give us
	\begin{equation}
	\label{eq:covering1}
	\max(\log |\mathcal S_{\epsilon}^\ell|,\, \max_{S \in \mathcal S_{\epsilon}^\ell} \dim(S)) \leq CJ_{{k}, \ell}\log (\e K_\ell J_{k, \ell})
	\end{equation}
	where
	\begin{equation}
	\label{eq:bnd_Jkl}
	J_{{k}, \ell} \leq C \log (\e m)\big(K_\ell + \frac{2^{-\ell}(2t + \delta)\sqrt{mK_\ell}}{2^{-\ell/2}\sqrt{m}\ep}\big)\,.
	\end{equation} 
	Now using 
	\begin{equation}
	\label{eq:K2}
	K_\ell \leq K \leq C\log (\e m) \big( 1 + \,{(t\sqrt{\frac{m}{n}} + \delta)\ep^{-1}}\big)\,
	\end{equation}
	for the first term inside the parenthesis in \eqref{eq:bnd_Jkl} (recall the definition of $K_\ell$ and $K$ from the proof of Lemma~\ref{lem:main}) and using
	$K_\ell \leq \frac{2^{\ell}}{m \epsilon^2}$ for the second, we get
	\begin{equation*}
	J_{k, \ell} \leq C(\log(\e m))^2 + C(\log(\e m))^2\big(t\sqrt{\frac{m}{n}} + t + \delta \big) \big( \frac{1}{\ep} + \frac{1}{\sqrt{m}\epsilon^2} \big)\,.
	\end{equation*}
	Further noticing that $\ep \leq 1 / \sqrt{m}$, so that $\tfrac{1}{\ep} \leq \tfrac{1}{\sqrt{m}\epsilon^2}$, we obtain
	\begin{equation}
	\label{eq:Jklcase1}
	J_{k, \ell} \leq C(\log(\e m))^2 \Big(1 + \frac{1}{\sqrt{m}\epsilon^2}\big(t\sqrt{\frac{m}{n}} + t + \delta\big)^{2\downarrow} \Big)\,
	\end{equation}
	(recall that $x^{2\downarrow} = x + x^2$). 
	On the other hand we have $K_\ell = k \le J_{k, \ell}$ for $C > 1$.
	Plugging these bounds into the right hand side of \eqref{eq:covering1} and rewriting the expression in terms of $\mathcal L_m(x) = x \log (\e \log(\e m)^2 x)$ we obtain
	\begin{equation}
	\label{eq:covering2}
	\max(\log |\mathcal S_{\epsilon}^\ell|,\, \max_{S \in \mathcal S_{\epsilon}^\ell} \dim(S)) \leq  C (\log(\e m))^2\mathcal L_m\big( 1 + \frac{1}{\sqrt{m}\epsilon^2}\big(t\sqrt{\frac{m}{n}} + t + \delta\big)^{2\downarrow} \big)\,.
	\end{equation}
	where we used the fact that $\log (C \e \log(\e m)^2 x) \leq C\log (\e \log(\e m)^2 x)$ for all $x \geq 1$ and large enough $C$.
	
	\smallskip
	
	\noindent\emph{Case~2: $K_\ell \geq \frac{2^{\ell}}{m \epsilon^2}$. } Notice that in this case we can choose $k = \floor{\frac{2^{\ell}}{m \epsilon^2}}$ and Lemma~\ref{lem:prep} gives us
	\begin{equation}
	\label{eq:covering3}
	\max(\log |\mathcal S_{\epsilon}^\ell|,\, \max_{S \in \mathcal S_{\epsilon}^\ell} \dim(S)) \leq CJ_{{k}, \ell}^*\log (\e K_\ell J_{k, \ell}^*)\,.
	\end{equation}
	We will show below that the right hand side of \eqref{eq:Jklcase1} also serves as an upper bound for $J_{k, \ell}^*$ and $K_\ell$, and consequently the upper bound in \eqref{eq:covering2} holds in this case as well, thus 
	proving the claim. To this end we will use the bounds 
	\eqref{eq:choose_k2} and \eqref{eq:C_epkl}. First observe that the bound on $J_{k, \ell}$ is same as in the previous 
	case since the only bounds we used there were $k \leq K_\ell$ and $k \leq \frac{2^{\ell}}{m \epsilon^2}$, both of which 
	are valid in this case. On the other hand, $C_{\epsilon, k, \ell}$ can be bounded by
	\begin{equation}
	\label{eq:K3}
	C\sqrt{\log (\e m)}(t\sqrt{\frac{m}{n}} + \delta)\,\big(\frac{2^{\ell/2}\sqrt{K_\ell}}{\sqrt{m}\ep} + \frac{\sqrt{(2t + \delta)K_\ell}}{m^{1/4}\ep }\big)\,.
	\end{equation}
	Since $K_\ell \geq \frac{2^{\ell}}{m \epsilon^2}$ and $\ep \leq 1 / \sqrt{m}$, we have
	\begin{align*}
	\frac{\sqrt{K_\ell}}{2^{-\ell/2}\sqrt{m}\ep} \leq K_\ell  &\overset{\eqref{eq:K2}}{\leq} C\log (\e m) \big( 1 + \,{(t\sqrt{\frac{m}{n}} + \delta)\frac{\sqrt{m}\ep}{\sqrt{m}\epsilon^2}}\big)\\
	&\leq C\log(\e m) + C\log (\e m)  \,{(t\sqrt{\frac{m}{n}} + \delta)}\frac{1}{\sqrt{m}\epsilon^2}\,
	\end{align*}
	(cf. the right hand side of \eqref{eq:Jklcase1}). Similarly we can bound
	\begin{align*}
	\frac{\sqrt{(2t + \delta){K_\ell}}}{m^{1/4}\ep} &\leq C\sqrt{\log(\e m)}\sqrt{t + \delta}\,\big(\frac{1}{m^{1/4}\ep} + \frac{\sqrt{t\sqrt{\frac{m}{n}} + \delta}}{m^{1/4}\epsilon^{3/2}}\big)\\
	&\leq C\sqrt{\log(\e m)}\sqrt{t + \delta}\,\Big(1 + \sqrt{t\sqrt{\frac{m}{n}} + \delta}\Big)\frac{1}{m^{1/4}\ep^{3/2}}\\
	&= C\sqrt{\log(\e m)}\sqrt{t + \delta}\,\Big(1 + \sqrt{t\sqrt{\frac{m}{n}} + \delta}\Big)\frac{\sqrt{\sqrt{m}\ep}}{\sqrt{m}\epsilon^2}\\ &\leq C\sqrt{\log(\e m)}\sqrt{t + \delta}\,\Big(1 + \sqrt{t\sqrt{\frac{m}{n}} + \delta}\Big)\frac{1}{\sqrt{m}\epsilon^2}\,.
	\end{align*}
	Plugging these bounds into the \eqref{eq:K3} we get 
	\begin{equation*}
	C_{\epsilon, k, \ell} \leq C(\log(\e m))^2 \Big(1 + \frac{1}{\sqrt{m}\epsilon^2}\big(t\sqrt{\frac{m}{n}} + t + \delta\big)^{2\downarrow} \Big)\,.
	\end{equation*}
	where used the simple fact that $x^{3/2} \leq 
	x^{2\downarrow}$. Combined with \eqref{eq:Jklcase1} and the discussion preceding the display \eqref{eq:K3}, this yields us a similar upper bound for $J_{k, \ell}^*$. 
\end{proof}
We are only left with the proof of Lemma~\ref{lem:prep}.
\begin{proof}[Proof of Lemma~\ref{lem:prep}] 
	The proof is split into two parts. In the first part we try to construct, for any 
	given $\theta \in \mathcal A(m, n , u, v, t)$, another matrix $\hat \theta$ satisfying 
	$\|\theta - \hat \theta\| \leq \tau$ such that $\hat \theta$ is piecewise constant on 
	rectangles with as few blocks as possible. These blocks define a partition $P$ of $L_{m, n}$ and let $\mathcal P$ denote the set of all such partitions. It is then clear that $S_\tau \coloneqq \{S_P: P \in \mathcal P\}$ forms a $\tau$ subspace cover of $\mathcal A(m, n , u, v, t)$ (see the proof of Theorem~\ref{Thm:1} in Section~\ref{sec:thm1} for the notation and similar notions). In the second and the final part we bound $\max_{P \in \mathcal P}|P|$ and $|\mathcal P|$ which, in view of the definition above, yield the desired upper bounds on $\max_{S \in \mathcal S_\tau}\dim(S)$ and $|\mathcal S_\tau|$.
	
	\smallskip
	
	{\bf Approximating $\theta$ by a piecewise constant matrix. } This part consists of three steps. In the ``zeroth'' step, we divide $\theta$ equally into $k$ submatrices by 
	\emph{horizontal divisions}. We do not choose, a priori, any specific value of $k$ which is the reason why our final 
	bound depends on $k$. Then in step~1, each of these submatrices is divided into submatrices by 
	\emph{vertical divisions} which are again subdivided in 
	step~2 by horizontal divisions. The rectangles corresponding to these submatrices will be the final level 
	sets of $\hat \theta$. We now elaborate the steps.
	
	\smallskip
	
	{\em Step~0: Horizontal Divisions. } Fix a positive integer 
	$1 \leq k \leq m$ 
	and divide $L_{m, n}$ into $k$ submatrices as follows:
	$$L_{m, n} = \begin{bmatrix}
	R_1 \\
	R_2\\
	\vdots \\
	R_k
	\end{bmatrix}\,$$
	where each $R_i$ has either $\ceil{m / k}$ or $\floor{m/k}$ many rows. We want to stress that we use the same partitioning for every $\theta$ in this step. 
	
	\smallskip
	
	{\em Step~1: Vertical Divisions. } Next we want to subdivide each $R_i$ (where $i \in [k]$) by making $j_i$ many vertical divisions:
	$$R_i = [R_{i, 1} \rvert R_{i, 2} \rvert \ldots \rvert R_{i, j_i}]$$
	such that $\TVr(\theta_{|R_{i, j}}) \leq \tau_k$ for all $j \in 
	[j_i]$ and some $\tau_k > 0$ to be chosen shortly. We can do this by the $(\TVr, \tau_k)$ scheme applied to the columns of $\theta_i$ so that Lemma~\ref{lem:division2} gives us the bounds
	\begin{equation}\label{eq:up0}
	j_i \leq \log_2(4n)\big(1 + \frac{\TVr(\theta_i)}{\tau_k}\big)\,.
	\end{equation}
	Replacing each element in every row of $\theta_{i, j} \coloneqq \theta_{|R_{i, j}}$ with the corresponding row mean, we then obtain a new matrix
	$$\tilde \theta_{i} = [\tilde \theta_{i, 1} \rvert \tilde \theta_{i, 2} \rvert \ldots \rvert \tilde \theta_{i, j_i}]\,.$$ By construction, each $\tilde \theta_{i, j}$ has identical columns. Finally, let us define
	$$\tilde \theta = \begin{bmatrix}
	\tilde \theta_1 \\
	\tilde \theta_2\\
	\vdots \\
	\tilde \theta_k
	\end{bmatrix}\,$$
	From the Cauchy-Schwarz inequality, it is clear that 
	$\|\tilde \theta\| \leq \| \theta \|$. One important observation we need make at this point is that while this 
	averaging procedure might increase the value of $\TVc(\tilde \theta)$, it does not increase the value of $\TVc(\tilde\theta_{i, j})$ for any $i$ and $j$. Indeed by a computation exactly similar to that performed in 
	\eqref{eq:preserve_tvr} we get 
	\begin{align}
	\label{eq:preserve_tvc}
	\TVc(\tilde \theta_{i, j}) \leq \TVc(\theta_{i, j})\,.
	\end{align}
	
	Let us now try to bound $\|\theta  - \tilde \theta\|$. To this end notice that
	\begin{align}\label{eq:error1}
	\|\theta - \tilde \theta\|_2^2 =& \sum_{i \in [k], j \in [j_i]}\sum_{i' \in [\nr(R_{i})]} \|\theta_{i, j}[i',\,] - \tilde \theta_{i, j}[i',\,]\|_2^2\nonumber\\ \leq& \sum_{i \in [k], j \in [j_i]}\nc(R_{i, j})\sum_{i' \in [\nr(R_{i})]} \TV(\theta_{i, j}[i',\,])^2\,
	\end{align} 
	where in the final step we used Lemma~\ref{prop:1dtvapprox}. Since $\TVr(\theta_{i, j}) \leq \tau_k$, we can then deduce 
	\begin{align}\label{eq:up3}
	\|\theta - \tilde \theta\|_2^2 \leq& \sum_{i \in [k], j \in [j_i]}\nc(R_{i, j})\big(\sum_{i' \in [\nr(\theta_{i})]} \TV(\theta_{i, j}[i',\,])\big)^2 \nonumber \\ 
	\leq& \sum_{i \in [k], j \in [j_i]}\nc(R_{i, j})\tau_k^2 = nk\tau_k^2\,.
	\end{align}
	Setting $\tau_k  = \tau / 2\sqrt{nk}$, we get $\|\theta - \tilde \theta\|_2 \leq \tau /2$.
	
	%
	
	\smallskip 
	
	{\em Step~2: Horizontal Divisions. }In this step, we are going to make horizontal divisions within each $R_{i, j}$ obtained from step~1 so that the total variation of columns of $\tilde \theta_{i, j}$ restricted to each subdivision is smaller 
	than some fixed, small number. To this end fix $\tau_{k}' > 0$ whose exact value will be chosen later. Now use the $(\TVc, \tau_k')$ scheme applied to the rows of $R_{i, j}$ to obtain the following subdivision:
	$$
	R_{i, j} = \begin{bmatrix}
	R_{i, 1; j} \\
	R_{i, 2; j} \\
	\vdots \\
	R_{i, \ell_{i, j}; j} \\
	\end{bmatrix}\,$$
	where, with $\tilde \theta_{i, \ell; j} \coloneqq \tilde \theta_{|R_{i, \ell; j}}$, $\TVc[\tilde \theta_{i, \ell; j}] \leq \tau_k'$ for all $\ell \in [\ell_{i, j}]$. From Lemma~\ref{lem:division2} we can deduce
	\begin{equation}\label{eq:up2}
	\ell_{i, j} \leq \log_2 (4m)\big(1 + \frac{\TVc(\tilde \theta_{i, j})}{\tau_k'}\big)\,.
	\end{equation}
	Like in the definition of $\tilde \theta_{i, j}$, we now replace every element in each column of $\tilde \theta_{i,\ell; j}$ (recall at this point that $\tilde \theta_{i, j}$ and hence $\tilde \theta_{i, \ell; j}$ has identical columns) with the corresponding column mean and obtain a new matrix
	$$
	\hat \theta_{i, j} = \begin{bmatrix}
	\hat \theta_{i, 1; j} \\
	\hat \theta_{i, 2; j} \\
	\vdots \\
	\hat \theta_{i, \ell_{i, j}; j} \\
	\end{bmatrix}\,$$
	By construction, $\hat \theta_{i, \ell; j}$ is a constant matrix. Let $\hat \theta \in \R^{m \times n}$ be such that $\hat \theta_{|R_{i, \ell; j}} = \hat \theta_{i, \ell; j}$. By the Cauchy-Schwarz inequality we have $\|\hat \theta\| \leq \|\tilde \theta\| \leq \|\theta\|$.

	We now want to bound the distance between $\tilde \theta$ 
	and $\hat \theta$. Notice that, since the columns of $\tilde \theta_{i,\ell; j}$ are identical, we get from Lemma~\ref{prop:1dtvapprox}
	$$\|\tilde \theta_{i, \ell; j}[\,, j'] - \hat \theta_{i, \ell; j}[\,, j']\|_2^2 \leq \nr(\tilde \theta_{i, \ell; j}) (\tau_k'  / \nc(\tilde \theta_{i, j}))^2\,$$
	for every $j' \in \nc(\tilde \theta_{i, j}) = \nc(\tilde \theta_{i, \ell; j})$. Summing over $i,j, \ell$ and $j'$, we then deduce
	\begin{align*}
	\|\tilde \theta - \hat \theta\|_2^2 &\leq \sum_{i \in [k], j \in [j_i], \ell \in [\ell_{i, j}]}\frac{\nr(R_{i, \ell;j})}{\nc(R_{i, j})}\tau_k'^2 = \tau_k'^2\sum_{i \in [k], j \in [j_i]}\frac{\nr(R_{i, j})}{\nc(R_{i, j})} \\
	&\leq \frac{2\tau_k'^2m}{k}\sum_{i \in [k], j \in [j_i]}\frac{1}{\nc(R_{i, j})}\,.
	\end{align*}
	Let us choose 
	\begin{equation}\label{eq:choice_tauk'}
	{\tau_k'}^2 = \frac{\tau^2k}{8m\sum\limits_{i \in [k], j \in [j_i]}\frac{1}{\nc(R_{i, j})}}\,,
	\end{equation}
	so that $\|\tilde \theta - \hat \theta\|_2 \leq \tau / 2$ and hence 
	$$\|\theta - \hat \theta\|_2 \leq \|\theta - \tilde \theta\|_2 + \|\tilde \theta - \hat \theta\|_2 \leq \tau / 2 + \tau / 2 = \tau\,.$$
	
	\medskip
	
	{\bf Counting the number of possible partitions for any $\theta$. } Fix any vertical division of $\theta$ obtained in step~1. Now summing \eqref{eq:up2} over all $i$ and $j$ we get
	\begin{equation}
	\label{eq:ell_ijbnd1}
	\sum_{i\in [k], j \in [j_i]}\ell_{i, j} \leq \log_2 (4m)\,(\sum_{i\in [k]}j_i + v / \tau_k')
	\end{equation}
	where we used the following fact
	$$\sum_{i \in [k], j \in [j_i]} \TVc(\tilde \theta_{i, j}) \overset{\eqref{eq:preserve_tvc}}\leq \sum_{i \in [k], j \in [j_i]} \TVc(\theta_{i, j}) \leq \TVc (\theta) \leq v\,.$$
	On the other hand \eqref{eq:choice_tauk'} allows us to deduce a naive lower bound on $\tau_k'$ as follows:
	$$\tau_k' \geq \frac{\tau \sqrt{k}}{4\sqrt{2m}\sqrt{\sum_{i \in [k]}j_i}}\,.$$
	Plugging this into \eqref{eq:ell_ijbnd1} we get for a universal constant $C > 0$,
	\begin{equation}
	\label{eq:ell_ijbnd2}
	n_{piece}(\hat{\theta}) := \sum_{i\in [k], j \in [j_i]}\ell_{i, j} \leq C\log(\e m)\big( J + \sqrt{J} \frac{v\sqrt{m}}{\tau\sqrt{k}} \big)
	\end{equation}
	where $n_{piece}(\hat{\theta})$ is 
	the total number of rectangular level sets of $\hat \theta$ and $J \coloneqq \sum_{i\in [k]}j_i$. From now onwards we will implicitly assume that $C$ is a positive, universal constant whose exact value may vary from one line to the next.

	Therefore the number of tuples $(\ell_{1, 1}, \ell_{1, 2}, \ldots, \ell_{k, j_k})$ satisfying \eqref{eq:ell_ijbnd2} is at most
	\begin{equation}
	\label{eq:ell_ijcard}
	(C\log (\e m))^{J}\big(J + \sqrt{J} \frac{v\sqrt{m}}{\tau\sqrt{k}}\big)^{J}\,.
	\end{equation}
	Similarly, in order to bound $J$ we sum \eqref{eq:up0} over all $i$ to obtain
	\begin{align}
	\label{eq:Jkbnd}
	J &= \sum_{i \in [k]} j_i \leq \log_2(4n)\big(k + \frac{1}{\tau_k}\sum_{i \in [k]} \TVr(\theta_i)\big) = \log_2(4n)\big(k + \frac{1}{\tau_k} \TVr(\theta)\big)\nonumber\\
	&\leq C \log (\e n) \big(k + \frac{u}{\tau_k}\big) \leq C \log (\e n)\big(k + \frac{u\sqrt{nk}}{\tau}\big) \eqqcolon J_k\,,
	\end{align}
	where in the final step we used $\tau_k = \tau / 2 \sqrt{nk}$ (see the end of step~1 in the previous part).
	
	It remains to count the number of possible vertical divisions in step~1. To this end let us fix a tuple $(j_1, j_2, \ldots, j_k)$ satisfying $\sum_{i \in [k]}j_i \leq 
	J_k$. The number of possible vertical divisions in this case is bounded by $\prod_{i \in [k]}n^{j_i} = 
	n^{J_k}$. On the other hand, in view of \eqref{eq:up0} and \eqref{eq:Jkbnd}, the number of tuples $(j_1, j_2, \ldots, j_k)$ is bounded by the number of nonnegative integral solutions to the inequality $\sum_{i \in [k]} j_i \leq J_k$ which in turn is bounded by 
	$(J_k)^k$. Putting all of these together with \eqref{eq:ell_ijcard} and \eqref{eq:Jkbnd}, we can now deduce the following upper bound on the total number of possible partitions for any $\theta \in \mathcal A(m, n, u, 
	v, t)$:
	\begin{equation}
	\label{eq:partition_card}
	(J_k)^kn^{J_k}(C\log (\e m))^{J_k}\big(J_k + \sqrt{J_k} \frac{v\sqrt{m}}{\tau\sqrt{k}}\big)^{J_k}\,.
	\end{equation}
	From this and \eqref{eq:ell_ijbnd2} we can derive the bound for any $1 \le k < m$. For the second bound, that is when $k = m$, recall that the second summand in the right hand side of \eqref{eq:ell_ijbnd1} comes from the horizontal division conducted in step~2. Since this step becomes void for $k = m$, the required bound follows in exactly similar fashion with $J_k$ replacing $J_k + \sqrt{J_k}\tfrac{v\sqrt{m}}{\tau\sqrt{k}}$.
	%
	%
\end{proof}


\section{Proof of Theorem~\ref{Thm:notuning}}
To prove Theorem~\ref{Thm:notuning} we apply the general machinery developed in~\cite{chatterjee2015high} with suitable modifications. 
Let us define $w = y - \overline{y}$ to be the centered data matrix, $w^* = \theta^* - \overline{\theta^*}$ to be the centered ground truth matrix and let
\begin{equation}\label{eq:def}
\hat{w} \coloneqq \argmin_{v:\, \overline{v} = 0,\, \|w - v\|^2 \leq (n^2 - 1) \hat{\sigma}^2} \TV(v)\,.
\end{equation}
Also, for any $V \geq 0,$ let $\hat{w}_{V}$ denote the Euclidean projection of $w$ onto the convex set $K_n^0(V).$ Recall that $K_n^0(V) \coloneqq \{\theta \in \R^{n \times n}: \TV(\theta) \leq V,\, \overline{\theta} = 0\}\,.$

\subsection{Sketch of Proof} To show that $\hat{\theta}_{\mathrm{notuning}}$ is a good 
estimator of $\theta^*$ it clearly suffices to show that $\hat{w}$ is a good estimator 
of $w^*.$ If we knew $\TV(\theta^*) = \TV(w^*) = V^*$, a similar argument as in the 
proof of Theorem~\ref{Thm:1} would tell us that $\hat{w}_{V^*}$ attains the 
$\tilde{O}(\frac{V^*}{\sqrt{N}})$ rate that we desire. Of course, the aim here is to 
get the same rate without knowing $V^*$ and $\sigma.$ One part of our proof deals with 
showing that using $\hat{\sigma}$ in the definition of our estimator is not much worse 
than if we knew $\sigma$ and used it in defining our estimator. This is shown by 
showing that $\hat{\sigma} \approx \sigma$ using a concentration of measure argument 
where ``$\approx$'' is a somewhat informal notation conveying the meaning of 
approximately equal to.

To analyze the risk of $\hat{w},$ a natural first step is to decompose the risk as follows:
\begin{equation*}
\|\hat{w} - w^*\|^2 \leq 2 \|\hat{w}_{V^*} - w^*\|^2 + 2 \|\hat{w} - \hat{w}_{V^*}\|^2\,.
\end{equation*}
Here we used the elementary inequality $\|a + b\|^2 \leq 2 \|a\|^2 + 2 \|b\|^2$. The above decomposition has a natural interpretation as twice the sum of the ideal risk (achievable when $V^*$ is known) and an excess risk due to not knowing $V^*$ and $\sigma.$ The main task therefore is to upper bound the excess risk term $\|\hat{w} - \hat{w}_{V^*}\|^2.$ 

We now need to look at two different cases.
The first case is when $\hat{w} \neq \textbf{0}.$ In this case we first show that the minimum of the optimization problem defined in~\eqref{eq:def} is attained on the boundary. This would mean we have $\|\hat{w} - w\|^2 = (n^2 - 1) \hat{\sigma}^2 \approx (n^{2} - 1) \sigma^2.$ 
Letting $\hat{V} = \TV(\hat{w})$, a simple geometric argument also shows that $\hat{w}_{\hat{V}} = \hat{w}.$ Thus, both $\hat{w}_{V^*}$ and $\hat{w}$ are Euclidean projections onto $K_n^0(V)$ for two possibly different choices of $V.$ Thus, we can now  use standard characterizations of Euclidean projections onto convex sets (content of Lemma~\ref{lem:projfact}) for both $\hat{w}_{V^*}$ and $\hat{w}$ to obtain a bound on the excess risk as follows:
\begin{equation*}
\|\hat{w} - \hat{w}_{V^*}\|^2 \leq \big|\|\hat{w} - w\|^2 - \|w - \hat{w}_{V^*}\|^2\big|\,.
\end{equation*}  
Since $\|\hat{w} - w\|^2 \approx (n^{2} - 1) \sigma^2$ we can then conclude
\begin{equation*}
\big|\|\hat{w} - w\|^2 - \|w - \hat{w}_{V^*}\|^2\big| \approx \big|(n^{2} - 1) \sigma^2 - \|w - \hat{w}_{V^*}\|^2\big|\,.
\end{equation*} 
Further, since $\hat{w}_{V^*}$ is known to be a good estimator of $w^*$ we can write
\begin{equation*}
\|w - \hat{w}_{V^*}\|^2 \approx \|w - w^{*}\|^2 = \|Z - \overline Z \textbf{1}\|^2 \sigma^2 \approx (n^{2} - 1) \sigma^2\,.
\end{equation*}
where the last approximation is again by a simple concentration of measure argument. The last three displays then suggest that $\hat{w}$ is close to $\hat{w}_{V^*}.$ Quantifying the last three displays gives us the desired upper bound on the excess risk.

The second case is when $\hat{w} = \textbf{0}.$ By definition we have 
$\|\hat{w}\|^2 \leq (n^{2} - 1) \hat{\sigma}^2 \approx (n^{2} - 1) \sigma^2.$ Since $\textbf{0} \in K_n^0(V^*)$ and $\hat{w}_{V^*}$ is the projection of $w$ onto $K_n^0(V^*)$, a standard fact about Euclidean projections onto convex sets gives $\langle w - \hat{w}_{V^*}, \hat{w}_{V^*} \rangle \geq 0.$ This implies 
\begin{equation*}
\|\hat{w} - \hat{w}_{V^*}\|^2 = \|\hat{w}_{V^*}\|^2 \leq \|w\|^2 - \|w - \hat{w}_{V^*}\|^2 \lesssim (n^{2} - 1) \sigma^2 - \|w - \hat{w}_{V^*}\|^2\,.
\end{equation*} 
The rest of the proof then follows similarly as in the previous case.

\subsection{Full Proof}
While proving Theorem~\ref{Thm:notuning} we will prove a few intermediate results. Our first lemma is a basic fact about Euclidean projections onto $K_n^0(V)$ for two different choices of $V.$ This also appears as Lemma $5.1$ in~\cite{chatterjee2015high}. For the sake of completeness, we give a proof below.
\begin{lemma}\label{lem:projfact}
	Let $y \in \R^{n \times n}$ and recall $K_n^0(V) \coloneqq \{\theta \in \R^{n \times n}: \TV(\theta) \leq V, \overline{\theta} = 0\}.$ Let $V_1 > V_2 \geq 0$ and let $\pi_1(y),\pi_2(y)$ be the Euclidean projection of $y$ onto the convex sets $K_n^0(V_1),K_n^0(V_2)$ respectively. Then we have the following inequality:
	\begin{equation*}
	\|\pi_1(y) - \pi_2(y)\|^2 \leq \|y - \pi_2(y)\|^2 - \|y - \pi_1(y)\|^2.
	\end{equation*}
\end{lemma}
\begin{proof}
	Since $\pi_2(y) \in K_n^0(V_1)$ by definition, the standard KKT condition for projections onto convex sets implies
	$\langle y - \pi_1(y), \pi_2(y) - \pi_1(y)\rangle \leq 0.$ Therefore we can write
	\begin{align*}
		\|y - \pi_2(y)\|^2 &= \|y - \pi_1(y)\|^2 + \|\pi_1(y) - \pi_2(y)\|^2 + 2 \langle y - \pi_1(y),\pi_1(y) - \pi_2(y) \rangle \\
		&\geq \|y - \pi_1(y)\|^2 + \|\pi_1(y) - \pi_2(y)\|^2.
	\end{align*}
	This finishes the proof of the lemma. 
\end{proof}

Our next lemma is the following pointwise inequality.
\begin{lemma}\label{lemma:ptwisefact}
	Let $w = y - \overline{y} \textbf{1}$ be the centered version of $y.$ For any $V \geq 0,$ let $\hat{w}_{V}$ denote the projection of $w$ onto the convex set $K_n^0(V).$ Let 
	\begin{equation}\label{eq:opt}
	\hat{w} =  \argmin_{v:\, \overline{v} = 0,\, \|w - v\|^2 \leq (n^2 - 1) \hat{\sigma}^2} \TV(v).
	\end{equation}
	Then we have the following pointwise inequality;
	\begin{equation*}
	\|\hat{w} - \hat{w}_{V^*}\|^2 \leq |(n^2 - 1) \hat{\sigma}^2 - \|w - \hat{w}_{V^*}\|^2|.
	\end{equation*}
\end{lemma}

\begin{proof}
	Let us first consider the case when $\hat{w} \neq \textbf{0}.$ 
	Define $\hat{V} \coloneqq \TV(\hat{w}).$ We claim that $\hat{w}_{\hat{V}} = \hat{w}$ and further
	\begin{equation}\label{eq:onbdry}
	\|w - \hat{w}\|^2 = (n^2 - 1) \hat{\sigma}^2.
	\end{equation}

	To prove the above claim, suppose $\hat{w}_{\hat{V}} \neq \hat{w}.$ Then we have 
	$\|w - \hat{w}_{\hat{V}}\|^2 < \|w - \hat{w}\|^2 \leq (n^2 - 1) \hat{\sigma}^2$ because of uniqueness of Euclidean projections onto convex sets. Therefore, we have $\|w - \hat{w}_{\hat{V}}\|^2 <  (n^2 - 1) \hat{\sigma}^2$ and $\|w - \textbf{0}\|^2 > (n^2 - 1) \hat{\sigma}^2$ by assumption. Let us now draw a line segment connecting $\hat{w}_{\hat{V}}$ to the origin and select the point which cuts the boundary of the $\sqrt{(n^2 - 1)} 
	\,\hat{\sigma}$ ball around $w$ and call it $w^{{\rm bdry}}.$ 
	Then by construction we have
	\begin{equation}\label{eq:tvless}
	\TV(w^{{\rm bdry}}) < \TV(\hat{w}_{\hat{V}}) \leq \TV(\hat{w}).
	\end{equation}
	Since $w$ has zero mean, it is not hard to see that $\hat{w}_{\hat{V}}$ has mean zero as 
	well because $\hat{w}_{\hat{V}}$ is the Euclidean projection of $w$ onto 
	$K_n^0(\hat{V}).$ Therefore any point falling on the line segment between 
	$\hat{w}_{\hat{V}}$ and the origin also must have mean zero, including $w^{{\rm bdry}}.$ Thus $w^{{\rm bdry}}$ is feasible for the optimization problem defined in~\eqref{eq:opt}. Together with~\eqref{eq:tvless} this contradicts the definition of $\hat{w}.$ Therefore $\hat{w}_{\hat{V}}$ must be equal to $\hat{w}$ and~\eqref{eq:onbdry} must hold.

	Letting $V^* = \TV(\theta^*)$, we can now write
	\begin{align*}
	\|\hat{w} - \hat{w}_{V^*}\|^2 = \|\hat{w}_{\hat{V}} - \hat{w}_{V^*}\|^2 \leq |\|w - \hat{w}_{\hat{V}}\|^2 - \|w - \hat{w}_{V^*}\|^2| = |(n^2 - 1) \hat{\sigma}^2 - \|w - \hat{w}_{V^*}\|^2|
	\end{align*}
	where we have applied Lemma~\ref{lem:projfact} in the first inequality and used~\eqref{eq:onbdry} in the last equality.

	Now let us consider the case when $\hat{w} = \textbf{0}.$ In this case we can write
	\begin{align*}
	&\|\hat{w} - \hat{w}_{V^*}\|^2 = \|\hat{w}_{0} - \hat{w}_{V^*}\|^2 \leq |\|w\|^2 - \|w - \hat{w}_{V^*}\|^2| = \|w\|^2 - \|w - \hat{w}_{V^*}\|^2 \leq \\& (n^2 - 1) \hat{\sigma}^2 - \|w - \hat{w}_{V^*}\|^2.
	\end{align*}
	The first inequality uses Lemma~\ref{lem:projfact} and the second equality follows from the definition of $\hat{w}_{V^*}$ upon observing that $\bm 0 \in K_n^0(V^*)$. Finally the third inequality uses the fact that $\|w\|^2 \leq (n^2 - 1) \hat{\sigma}^2$ since $\hat{w} = \textbf{0}$. This finishes the proof of the lemma. 
\end{proof}

Our next result is a proposition which gives a pointwise upper bound to the squared loss. 
\begin{proposition}\label{prop:ptwise}
	Let $V^* = \TV(\theta^*).$ Let $w = y - \overline{y} \textbf{1}$ and $w^* = \theta^* - \overline{\theta^*} \textbf{1}$ be the centered versions of $y$ and $\theta^*$ respectively. Also let $\hat{w}_{V^*}$ denote the Euclidean projection of $w$ onto $K_n^0(V^*).$ Then the following pointwise risk inequality holds:
	\begin{align*}
	\|\hat{\theta} - \theta^*\|^2 \leq \:8\:\sigma\:\sup_{v \in K_n^0(2V^*)} \langle Z, v\rangle\, +\, &|\overline{y} - \overline{\theta^*}|^2 n^2 +   2|\|w - w^*\|^2 - (n^2 - 1) \sigma^2|\\ + &2(n^2 - 1)\:|\hat{\sigma}^2 - \sigma^2|.
	\end{align*}
\end{proposition}


\begin{proof}
	By definition of $\hat{\theta}$ and Pythagorean theorem we have
	\begin{equation}\label{eq:tr1}
	\|\hat{\theta} - \theta^*\|^2 = \|\overline{y}\:\textbf{1} - \overline{\theta^*}\:\textbf{1}\|^2 + \|\hat{w} - w^*\|^2 \leq \|\overline{y}\textbf{1} - \overline{\theta^*}\textbf{1}\|^2 + 2 \|\hat{w} - \hat{w}_{V^*}\|^2 + 2\|\hat{w}_{V^*} - w^*\|^2.
	\end{equation}


	We can now use Lemma~\ref{lemma:ptwisefact} and the triangle inequality to write
	\begin{align}\label{eq:tr2}
	&\|\hat{w} - \hat{w}_{V^*}\|^2 \leq |(n^2 - 1)\hat{\sigma}^2 - \|w - \hat{w}_{V^*}\|^2| \leq 
	(n^2 - 1)\:|\hat{\sigma}^2 - \sigma^2|\, + \nonumber\\&|\|w - w^*\|^2 - (n^2 - 1) \sigma^2| + |\|w - \hat{w}_{V^*}\|^2 - \|w - w^*\|^2|
	\end{align}
	Let us now bound the third term above on the right side. 
	\begin{align*}
	|\|w - \hat{w}_{V^*}\|^2 - \|w - w^*\|^2| &= |\|w^* - \hat{w}_{V^*}\|^2 + 2 \langle w - w^*, w^* -  \hat{w}_{V^*}\rangle|  \\&\leq \|w^* - \hat{w}_{V^*}\|^2 + 2 \sup_{v \in K_n^0(2V^*)} \langle w - w^*, v\rangle.
	\end{align*}
	We now observe that for any mean zero matrix $v$, we can write
	$$\langle w - w^*,v \rangle = \langle y - \theta^* - (\overline{y} - \overline{\theta^*}) \textbf{1}, v \rangle = \langle y - \theta^*,v \rangle = \sigma\:\langle Z,v \rangle.$$ 
	The last two displays then imply that 
	\begin{equation}\label{eq:tr3}
	|\|w - \hat{w}_{V^*}\|^2 - \|w - w^*\|^2| \leq \|w^* - \hat{w}_{V^*}\|^2 + 2\:\sigma\: \sup_{v \in K_n^0(2V^*)} \langle Z, v\rangle.
	\end{equation}
	Further, from the basic inequality $\|w - \hat{w}_{V^*}\|^2 \leq \|w - w^*\|^2$ we can conclude
	\begin{equation*}
	\|w^* - \hat{w}_{V^*}\|^2 \leq 2 \langle \hat{w}_{V^*} - w^*, w - w^* \rangle = 2 \langle \hat{w}_{V^*} - w^*, y - \theta^* \rangle \leq 2\:\sigma\: \sup_{v \in K_n^0(2V^*)} \langle Z,v \rangle.
	\end{equation*}
	The last display along with ~\eqref{eq:tr1},~\eqref{eq:tr2} and ~\eqref{eq:tr3} finish the proof of the proposition. 
\end{proof}

We are now in a position to finally prove Theorem~\ref{Thm:notuning}.

\begin{proof}[Proof of Theorem~\ref{Thm:notuning}]
	It suffices to take expectation over the four terms which consists in the upper bound given in Proposition~\ref{prop:ptwise}. We now sequentially bound the expectation of these terms. We will use $C$ to denote a positive, universal constant whose exact value may change from one line to the next.
	
	The first term is just $8 \sigma$ times the Gaussian width of $K_n^0(2V^*)$ and we can use~\ref{eq:gwkv} to upper bound it. As for the second term, it is clear that 
	\begin{equation*}
	n^2 \E (\overline{y} - \overline{\theta^*})^2 = n^2 \text{Var} (\overline{y}) = \sigma^2.
	\end{equation*}

	Also we observe that $\frac{\|w - w^*\|^2}{\sigma^2}= \sum_{i = 1}^{n} \sum_{j = 1}^{n} (Z_{ij} - \overline{Z})^2 \approx \chi^2_{n^2 - 1}.$ This is a standard fact about standard normal random variables. Therefore we can write 
	\begin{align*}
	\E |\|w - w^*\|^2 - (n^2 - 1) \sigma^2| &\leq \big(\E |\|w - w^*\|^2 - (n^2 - 1) \sigma^2|^{2}\big)^{1/2}\\ &\leq \sigma^2 \big(\text{Var}(\chi^2_{n^2 - 1}))^{1/2} = \sigma^2 \sqrt{2\:(n^2 - 1)} \leq \sqrt{2}\:\sigma^2 n
	\end{align*}
	where the first inequality follows from the Cauchy Schwartz inequality and the last equality follows because $\text{Var} (\chi^2_{k})) = 2\:k$ for any positive integer $k.$

	Next we bound $\E |\hat{\sigma}^2 - \sigma^2|.$ We can write
	\begin{align}\label{eq:square}
	|\hat{\sigma}^2 - \sigma^2| \leq &|\hat{\sigma} - \sigma|^2 + 2 \sigma |\hat{\sigma} - \sigma|\,.
	\end{align}
	Recalling the definition of $\hat{\sigma}$ we have
	\begin{align*}
	|\hat{\sigma} - \sigma| &= |\frac{\TV(\theta^* + \sigma Z) - \sigma \E \TV(Z)}{\E \TV(Z)}|\nonumber\\ 
	&\leq \frac{\TV(\theta^*)}{\E \TV(Z)} + \sigma \frac{|\TV(Z) - \E \TV(Z)|}{\E \TV(Z)}\,.
	\end{align*}
	Thus we can write
	\begin{align}\label{eq:sigmaest}
	&|\hat{\sigma} - \sigma|^2 \leq 2 (\frac{V^*}{\E \TV(Z)})^2 + 2 \sigma^2 (\frac{|\TV(Z) - \E \TV(Z)|}{\E \TV(Z)})^2\,.
	\end{align}
	Now, since $\TV(Z)$ is a sum of $N(0,2)$ random variables it is easy to check that $\E \TV(Z) = \frac{4\:n\:(n - 1)}{\sqrt{\pi}}.$ Also by Lemma~\ref{lem:TVdiff} we can upper bound the variance of $\TV(Z)$ to get
	\begin{equation*}
	\var (\TV(\mb Z)) \leq C n(n-1)\,.
	\end{equation*}
	Taking expectation on both sides of~\eqref{eq:sigmaest} we obtain
	\begin{equation*}
	\E |\hat{\sigma} - \sigma|^2 \leq 2 \big (\frac{V^* \sqrt{\pi}}{4\:n\:(n - 1)}\big)^2 + 2 \sigma^2 \frac{C \:\pi\:n(n-1)}{(4\:n\:(n - 1))^2} \leq C \big(\frac{(V^*)^2}{n^4} + \frac{\sigma^2}{n^2}\big)\,.
	\end{equation*}
	Using~\eqref{eq:square}, the last display and the Cauchy-Schwarz inequality to bound $\E 
	|\hat \sigma - \sigma|$, we can deduce 
	\begin{equation*}
	\E |\hat{\sigma}^2 - \sigma^2| \leq \E |\hat{\sigma} - \sigma|^2 + 2 \sigma \big(\E |\hat{\sigma} - \sigma|^2\big)^{1/2} \leq C \big(\frac{(V^*)^2}{n^4} + \frac{\sigma^2}{n^2}\big) + C \sigma \big(\frac{V^*}{n^2} + \frac{\sigma}{n}\big)\,.
	\end{equation*}

	Collecting the bounds we have obtained in this proof for the four terms comprising the upper bound given in Proposition~\ref{prop:ptwise}, we can conclude that
	\begin{equation*}
	{\rm MSE}(\hat{\theta}_{\mathrm{notuning}},\theta^*) \leq C \big(\sigma \frac{V^*}{N} \log (\e n) \log (2 + 2V^* n^2) + \big(\frac{V^*}{N}\big)^2 + \frac{\sigma^2}{\sqrt{N}} + \frac{\sigma^2}{N}\big)\,.
	\end{equation*}
	This finishes the proof of Theorem~\ref{Thm:notuning}.
\end{proof}

It only remains to prove the following lemma. 
\begin{lemma}\label{lem:TVdiff}
	There exists a universal constant $C > 0$ such that
	$$\var (\TV(\mb Z)) \leq C n(n-1) \,.$$ 
\end{lemma}
\begin{proof}
	Expanding $\var(\TV(\mb Z))$ we get
	\begin{align}
	\label{eq:TVdiff1}
	\var(\TV(\mb Z)) = \sum_{e, e'\in E_n}\cov (|\Delta_e \mb Z|, |\Delta_{e'} \mb Z|) = \sum_{e\in E_n}\sum_{e' \in E_n, e' \sim e}\cov(|\Delta_e \mb Z|, |\Delta_{e'} \mb Z|)\,
	\end{align}
	where in the second step we used the observation that $\cov(|\Delta_e \mb Z|, |\Delta_{e'} \mb Z|) = 0$
	for all non-adjacent $e, e'$, i.e., $e, e'$ which do not share any vertex. Here $e' \sim e$ means the edges $e,e'$ are adjacent. 
	Since each edge $e$ is adjacent to finitely many edges (including $e$ itself) 
	we get 
	from \eqref{eq:TVdiff1} that $\var(\TV(\mb Z)) \leq C|E_n|$ 
	for some universal constant $C > 0$. The lemma now follows by noting that $|E_n| = 2n(n-1)$. 
\end{proof}

\section{Appendix}\label{Sec:appendix}
\subsection{Some auxiliary results}\label{subsec:prelim}
\begin{lemma}
	\label{lem:simple}
	Suppose $\{f_i,g_i,h_i\}_{i = 1}^{n}$ are non negative real numbers satisfying the following inequality for each $i \in [n],$
	\begin{equation*}
	f_i \geq g_i - h_i.
	\end{equation*}
	Let $\{w_i\}_{i = 1}^{m}$ be some other non negative numbers.
	In addition, also suppose the following inequality holds for some $\delta > 0,$
	\begin{equation*}
	\sum \limits_{i = 1}^{n} f_i +  \sum \limits_{i = 1}^{m} w_i \leq \sum \limits_{i = 1}^{n} g_i + \delta.
	\end{equation*}
	Then the following is true:
	\begin{equation*}
	\sum_{i = 1}^{n} (f_i - g_i)_{+} + \sum \limits_{i = 1}^{m} w_i \leq \delta + \sum_{i = 1}^{n} h_i\,,
	\end{equation*}
	where $a_{+} = \max\{a,0\}$ for any $a \in \R$.
\end{lemma}

\begin{proof}
	The first equation in the above proposition basically says $(f_i - g_i)_{-} \leq  h_i$ for $i \in [n]$ where $a_- = 
	(-a)_+$ for any $a \in \R$. Therefore we can write
	\begin{align*}
	\delta &\geq \sum_{i = 1}^{n} (f_i - g_i) + \sum \limits_{i = 1}^{m} w_i = \sum_{i = 1}^{n} (f_i - g_i)_{+} - \sum_{i = 1}^{n} (f_i - g_i)_{-} + \sum \limits_{i = 1}^{m} w_i \nonumber \\
	&\geq \sum_{i = 1}^{n} (f_i - g_i)_{+} - \sum_{i = 1}^{n} h_i + \sum \limits_{i = 1}^{m} w_i
	\end{align*}
	which finishes the proof of the lemma.
\end{proof}

We state the following lemma which appears as Lemma $D.1$ in~\cite{guntuboyina2017adaptive}.
\begin{lemma}[Guntuboyina et al.]
	\label{lem:Guntu}
	Suppose $p,n \geq 1$ and let $\Theta_1,\dots,\Theta_p$ be subsets of $\R^n$ each containing the origin and contained in the closed Euclidean ball of radius $D > 0$ centered at the origin. Then for $Z \sim N(0,\sigma^2 I)$ we have
	\begin{equation*}
	\E \big(\max_{i \in [p]} \sup_{\theta \in \Theta_i} \langle Z,\theta \rangle \big) \leq \max_{i \in [p]} \E \big(\sup_{\theta \in \Theta_i} \langle Z,\theta \rangle \big) + D \sigma \big(\sqrt{2 \log p} + \sqrt{\frac{\pi}{2}}\big).
	\end{equation*}
\end{lemma}

Recall that for a vector $v \in \R^n$ we define 
$$ \TV(v) = \sum_{i = 1}^{n - 1} |v_{i + 1} - v_{i}|.$$

\begin{lemma}\label{prop:1dtvapprox}
	Let $\theta \in \R^n.$ Let us define $\overline{\theta} = (\sum_{i = 1}^{n} \theta_i)/n.$ Then we have the following inequality:
	\begin{equation*}
	\sum_{i = 1}^{n} \big(\theta_i - \overline{\theta}\big)^2 \leq n \TV(\theta)^2\,.
	\end{equation*}
\end{lemma}
\begin{proof}
	Define $\alpha_1 = \theta_1, \beta_1 = 0$ and for every $i > 2$ define
	$$\alpha_i = \alpha_{i - 1} + (\theta_{i} - \theta_{i - 1})_{+}.$$ Now define $\beta = \alpha - \theta.$ Observe that as defined, $\alpha,\beta$ are monotonically non decreasing vectors. Also, we have the equality
	$$\TV(\theta) = \TV(\alpha) + \TV(\beta) = (\alpha_n - \alpha_1) + (\beta_n - \beta_1).$$
	Now we can expand:
	\begin{align*}
	\sum_{i = 1}^{n} \big(\theta_i - \overline{\theta}\big)^2 &= \sum_{i = 1}^{n} \big(\alpha_i - \overline{\alpha}\big)^2 + \sum_{i = 1}^{n} \big(\beta_i - \overline{\beta}\big)^2 - 2 \sum_{i = 1}^{n} (\alpha_i - \overline{\alpha}) (\beta_i - \overline{\beta}) \nonumber \\
	&\leq \sum_{i = 1}^{n} \big(\alpha_i - \overline{\alpha}\big)^2 + \sum_{i = 1}^{n} \big(\beta_i - \overline{\beta}\big)^2 + 2 \sum_{i = 1}^{n} |\alpha_i - \overline{\alpha}| |\beta_i - \overline{\beta}|  \nonumber \\
	&\leq n(\alpha_n - \alpha_1)^2 + n(\beta_n - \beta_1)^2 + 2n (\alpha_n - \alpha_1)(\beta_n - \beta_1) \nonumber \\
	&= n(\alpha_n - \alpha_1 + \beta_n - \beta_1)^2 = n\TV(\theta)^2\,,
	\end{align*}
	thus giving us the lemma.
\end{proof}

\begin{lemma}\label{approxdecreasestv}
	Let $\alpha \in \R^n$ and let $B_1,B_2,\dots,B_k$ be a 
	partition of $[n]$ into contiguous blocks. Let $\alpha_{B_j}$ denote the restriction of $\alpha$ to the block $B_j.$ Also let $\tilde{\alpha} \in \R^n$ be defined so that 
	$$\tilde{\alpha}_{B_j} = \frac{1}{|B_j|} \sum_{i 
		\in B_j} \alpha_{i}\,.$$ 
	In other words, $\tilde{\alpha}$ is the best Euclidean approximation to $\alpha$ within the subspace of all vectors which are constant on each block $B_j.$ We then have the following inequality:
	\begin{equation*}
	\TV(\tilde{\alpha}) \leq \TV(\alpha).
	\end{equation*}
\end{lemma}
\begin{proof} 
	For any set of indices $i_1 \in B_1,\dots, i_k \in B_k,$ we have the following inequality:
	\begin{equation*}
	\TV(\alpha) \geq \sum_{j = 1}^{k - 1} |\alpha_{i_{j + 1}} - \alpha_{i_j}|.
	\end{equation*}
	Now averaging over the indices $i_j \in B_j$ and using Jensen's inequality gives us
	\begin{equation*}
	\sum_{j = 1}^{k - 1} |\alpha_{i_{j + 1}} - \alpha_{i_j}| \geq \sum_{j = 1}^{k - 1} |\tilde{\alpha}_{B_{j + 1}} - \tilde{\alpha}_{B_j}|.
	\end{equation*}
	The last two displays finish the proof of the proposition.
\end{proof}

\subsection{Proof of Lemma~\ref{lem:division2}}\label{subsec:division2}
\begin{proof}
	Let $P_0 = [n]$ be the initial partition. At every step we take the blocks $b_i \in P_i$ for which $T(b_i) > \epsilon$ and divide $b_i$ into two equal parts. Let $n_i$ be the number of blocks of the partition $P_i$ and $s_i$ equal the number of blocks $B_i$ in $P_i$ that are divided to obtain $P_{i + 1}.$ Define $s_0 = 0.$ Therefore we have $n_{i + 1} = n_i + s_i.$ Note that, due to superadditivity of $T,$ we must have $s_i \leq \lceil \frac{t}{\epsilon} \rceil.$ This implies in particular that $n_i \leq 1 + i \lceil \frac{t}{\epsilon} \rceil.$ Now the division scheme can go on for atmost $N = \lceil \log_2 n \rceil$ rounds. Therefore we have
	\begin{equation*}
	\max_{P \in \mathcal P(t, n, \ep, T)} |P_{U; T, \ep}| \leq 1 + \lceil \log_2 n \rceil \lceil \frac{t}{\epsilon} \rceil \leq 1 + (1 + \log_2 n) (1 + \frac{t}{\epsilon})\,.\qedhere
	\end{equation*}
	%
\end{proof}

\bibliographystyle{chicago}
\def\noopsort#1{}

\end{document}